\documentclass[10pt,reqno]{amsart}
\usepackage{hyperref}\hypersetup{colorlinks=true, citecolor=blue}
\usepackage{graphicx}
\usepackage{xfrac}
\usepackage{esint,amsfonts,amsmath,amssymb,epsfig,mathrsfs,bm}
\usepackage{ifthen,color,mathtools,yfonts}
\usepackage{epic,eepic,stmaryrd}
\usepackage{a4wide,accents}
\newtheorem{theorem}{Theorem}[section]
\newtheorem{lemma}[theorem]{Lemma}
\newtheorem{proposition}[theorem]{Proposition}
\newtheorem{corollary}[theorem]{Corollary}

\theoremstyle{definition}
\newtheorem{definition}[theorem]{Definition}

\theoremstyle{remark}
\newtheorem{remark}[theorem]{Remark}
\theoremstyle{remark}

\numberwithin{equation}{section}

\newcommand{\R}{\mathbb{R}}

\newcommand{\dist}{{\textup {dist}}}

\newcommand{\de}{\partial}
\newcommand{\ph}{\varphi}

\newcommand{\sign}{\textrm{sign}\,}
\renewcommand{\d}{{\rm d}}

\newcommand{\spt}{{\rm spt}\,}
\newcommand{\res}{\mathop{\hbox{\vrule height 7pt width .5pt depth 0pt
\vrule height .5pt width 6pt depth 0pt}}\nolimits}
\newcommand{\loc}{\textup{loc}}

\renewcommand{\and}{\quad \text{and} \quad}

\newcommand{\reg}{\textup{Reg}}
\newcommand{\sing}{\textup{Sing}}
\newcommand{\other}{\textup{Other}}
\renewcommand{\div}{\textup{div}}
\newcommand{\ctg}{\textup{cotg}\,}
\renewcommand{\top}{\textup{top}}
\newcommand{\low}{\textup{low}}
\newcommand{\sgn}{\textup{sign}}

\newcommand{\cG}{{\mathcal{G}}}

\newcommand{\cL}{{\mathcal{L}}}
\newcommand{\cH}{{\mathcal{H}}}

\newcommand{\cT}{{\mathcal{T}}}

\newcommand{\cS}{{\mathcal{S}}}

\newcommand\N{{\mathbb N}}

\newcommand\C{{\mathbb C}}

\newcommand{\ie}{\textit{i.e.}}
\newcommand{\eg}{\textit{e.g.}}

\title[The free boundary of the lower dimensional obstacle problem]
{On the measure and the structure of the free boundary 
of the lower dimensional obstacle problem}

\author[M.~Focardi]{Matteo Focardi}
\address{DiMaI, Universit\`a degli Studi di Firenze}
\curraddr{Viale Morgagni 67/A, 50134 Firenze (Italy)}
\email{focardi@math.unifi.it}
\author[E.~Spadaro]{Emanuele Spadaro}
\address{Universit\"at Leipzig}
\curraddr{Augustusplatz 10, 04109 Leipzig (Germany)}
\email{Emanuele.Spadaro@math.uni-leipzig.de}
\thanks{E.~S.~has been partially supported by the Gruppo Nazionale per l'Analisi Matematica, 
la Probabilit\`a e le loro Applicazioni (GNAMPA) of the Istituto Nazionale di Alta Matematica 
(INdAM) through a Visiting Professor Fellowship. 
E.~S.~is very grateful to the DiMaI ``U. Dini'' of the University of Firenze 
for the support during the visiting period. M.~F. is a member of the GNAMPA of INdAM. 
He warmly thanks the University of Leipzig for the support during a visit when part of this
work was conceived.}


\date{}
\begin{document}
\begin{abstract}
We provide a thorough description of the
free boundary for the
lower dimensional obstacle problem in $\R^{n+1}$
up to sets of null $\cH^{n-1}$ measure.
In particular, we prove
\begin{itemize}
\item[(i)] local finiteness of the 
$(n-1)$-dimensional Hausdorff measure  
of the free boundary,
\item[(ii)] $\cH^{n-1}$-rectifiability of the free boundary,
\item[(iii)] classification of the frequencies
up to a set of dimension at most $(n-2)$ and classification of 
the blow-ups at $\cH^{n-1}$ almost every free boundary 
point.
\end{itemize}
\end{abstract}

\maketitle

%
%
\section{Introduction}

Thin obstacle-type problems naturally appear
in several models of applied sciences,
such as contact mechanics (cf.~the classical Signorini problem)
and, as pointed out more recently, in
free boundary problems for fractional diffusions,
such as quasi-geostrophic flows,
American options' pricing, anomalous diffusions etc{\ldots}
Due to their character of prototypical nonlinear and
non-local equations,
in the recent years this class of problems
has been intensively studied, culminating in
several important contributions and breakthroughs 
(cf., \textit{e.g.},~\cite{AtCa04, CaSi07,
AtCaSa08, CaSaSi08, GaPe09, DaSa15, KPS15, BaFiRo16, FoSp16, And16}).
Nevertheless, many important questions
are not yet answered, most importantly the ones
concerning the global structure of the free boundary,
which according to the available results in the literature
is not excluded to have infinite measure or to be fractal,
already in the simplest model cases.

Here we answer to this and to other related
questions, such as the uniqueness of blow-ups 
and the structure of the free boundary 
for solutions to the thin obstacle problem, 
giving a complete description of the top-stratum of the free 
boundary up to a set of $\cH^{n-1}$-measure zero. These results 
are new also in the framework of the classical 
Signorini problem in elasticity (for the antiplane case)
and they are obtained by a combination of analytical and geometric 
measure theory arguments which can be 
suitably exploited also for similar
free boundary type problems.

\subsection{The problem}
In this article we consider a class of
lower dimensional obstacle problems.
In order to state them,
for any subset $E \subset \R^{n+1}$ we set
\[
E^+ := E\cap\big\{x\in\R^{n+1}:\,x_{n+1}>0\big\}
\and
E' := E \cap \big\{x_{n+1}=0\big\}.
\]
For any point $x \in \R^{n+1}$ we will write $x = (x', x_{n+1})\in \R^{n} 
\times \R$.
Moreover, $B_r(x) \subset \R^{n+1}$ denotes the open ball centered
at $x\in \R^{n+1}$ with radius $r>0$, and $\overline{B_r}(x)$ its closure
(we omit to write the point $x$ if the origin).
For every $R>0$, we  denote by $\mathscr{A}_R$
the set of functions in the weighted Sobolev space 
$H^1(B_R,|x_{n+1}|^a\cL^{n+1})$, with $a\in(-1,1)$,
which are even symmetric with respect to $x_{n+1}$
and which have positive traces on $B_R'$:
\begin{align*}
\mathscr{A}_R:=\Big\{v\in H^1(B_R,|x_{n+1}|^a\cL^{n+1}): 
\,v(x', x_{n+1})=v(x',- 
x_{n+1})
\and v(x',0)\geq 0 \Big\}.
\end{align*}
The thin obstacle problems we consider are then the following:
\begin{equation}\label{e:ob-pb local}
\begin{cases}
u(x',0) \geq 0 & \text{for }\; (x',0)\in B_R',\\
u(x',x_{n+1}) = u (x', -x_{n+1}) &\text{for }\; x=(x',x_{n+1})\in 
B_R,\\
\div\big(|x_{n+1}|^a \nabla u(x)\big) = 0  & \text{for }\; 
x \in B_R \setminus \big\{(x',0)\,:\,
u(x',0) = 0 \big\},\\
\div\big(|x_{n+1}|^a \nabla u(x)\big) \leq 0 & \text{in the sense of 
distribution in }\,B_R,\\
u (x) = g(x) & \text{for }\; x\in \de B_R,
\end{cases}
\end{equation}
where $g \in \mathscr{A}_R$ is a given boundary value datum.
Note that \eqref{e:ob-pb local} are the Euler--Lagrange
equations satisfied by the unique minimizer of the energy
\[
\int_{B_R} |\nabla v|^2 |x_{n+1}|^a\d x
\]
in the class 
$\mathscr{A}_{R,g}:=\mathscr{A}_R \cap \big\{g + 
H^1_0(B_R,|x_{n+1}|^a\cL^{n+1})\big\}$. 
In particular, in case $a=0$ problem
\eqref{e:ob-pb local}
corresponds to the well-known scalar Signorini problem.
We denote by $\Lambda(u)$ the \textit{coincidence set}
of a solution $u$,
\[
\Lambda(u) := \big\{(x', 0) \in B_R'\,: u(x', 0) =0 \big\},
\]
and by $\Gamma(u)$ its \textit{free boundary},
which is the topological boundary of $\Lambda(u)$ in the relative topology of 
$B_R'$.
In order to avoid unnecessary complications, 
in this work we consider the case of zero obstacle 
prescribed on flat hypersurfaces only.
Nevertheless, the techniques developed in the paper
can be generalized to consider non-constant and
non-flat obstacles, 
as well as for other free boundary problems
(such as the fractional obstacle problem, for
which the analogous results of this paper are going to appear
in a future work). Moreover, we set 
\[
s:=\frac{1-a}{2}
\]
throughout the whole paper.

\subsection{A short survey of the existing literature}
In the last years there has been an intensive research activity
in trying to set up the regularity properties of the solutions to 
\eqref{e:ob-pb local} and the corresponding free boundaries. 
We resume in what follows the state of the art for what concerns
the zero obstacle case. To this aim 
we introduce the following notation for the rescalings
of a solution $u$:
for every $x_0 \in \Gamma(u)$ and $r>0$, we set
\begin{gather}\label{e:rescaling-1}
\bar u_{x_0,r}(y) :=
\frac{r^{\frac{n+a}{2}}u(x_0+r\,y)}{\Big(\int_
{\de B_r} 
u^2\,|x_{n+1}|^{a}\,\d\cH^{n}\Big)^{\sfrac12}}
\quad\quad \forall\; y\in 
B_{\frac{R-|x_0|}{r}}.
\end{gather}
By \cite[Section 6]{CaSaSi08} the collection of functions 
$\{\bar u_{x_0,r}\}_{r>0}$ is pre-compact in the weighted Sobolev
space $H^1_{\loc}(\R^{n+1},|x_{n+1}|^a\,\cL^{n+1})$.
Their limiting points are
called \textit{blow-ups} of $u$ at $x_0$ and are
homogeneous functions, whose
homogeneity depends only on $x_0$ and not 
on the extracted subsequence (for a proof see also
Corollary~\ref{c:compactness} and Remark~\ref{r:freq modificata} 
below). 
The set of all blow-ups of a solution $u$ at $x_0$ is denoted by 
$\textup{BU}(x_0)$, and their common homogeneity 
$\lambda(x_0)$ is called the \textit{infinitesimal homogeneity} 
or the \textit{frequency} of $u$ at $x_0$
(this is indeed the limiting value, as the radius vanishes, 
of an Almgren's type frequency function). 

\medskip

The following statements summarize several results 
available in the current literature.

\smallskip

\noindent{\bf A. Optimal regularity of $u$.}
The solutions $u$ to \eqref{e:ob-pb local} are one-sided 
$C^{1,s}$, $s= \sfrac{(1-a)}{2}$.
More precisely, 
$u \in \textrm{Lip}(B_1)\cap C^{1,s}(B_1^\pm\cup B_1')$, as 
proved by Athanasopoulos and Caffarelli \cite{AtCa04} for $a=0$, and
by Caffarelli, Salsa and Silvestre \cite{CaSaSi08} for all $a \in (-1,1)$
(see also
\cite{Freh75, Freh77, Caffa79,Kind81,Ural85, Ural87,Si07}
for previous results).

\smallskip

\noindent{\bf B. Free boundary regularity.}
The free boundary $\Gamma(u)$ can 
be split as:
\begin{equation}\label{e:decomposizione free boundary}
\Gamma(u) = \reg(u) \cup \sing(u)\cup \other(u),
\end{equation}
with these subsets being pairwise disjoint, and more precisely
\begin{itemize}
\item[(i)] $\reg(u)$ is the subset of points in $\Gamma(u)$ in which 
blow-ups
are $(1+s)$-homogeneous.
$\reg(u)$ is relatively open in $\Gamma(u)$ and it is
an analytic $(n-1)$-dimensional submanifold of $\R^{n+1}$
(the $C^{1,\alpha}$ regularity has been shown in \cite{AtCaSa08,CaSaSi08}
-- see also \cite{FoSp16, Geraci}
for a different proof based on the epiperimetric inequality;
higher regularity follows from \cite{DaSa15, KPS15});

\item[(ii)] $\sing(u)$ is the subset of points in $\Gamma(u)$ for 
which the blow-ups are $2m$-homogeneous. In the case of 
the Signorini problem $a=0$, they are also characterized by the 
fact that their contact sets have density zero with respect to $\cH^n$. 
Furthermore, in such a case Garofalo and Petrosyan \cite{GaPe09} 
proved that $\sing(u)$ is contained in a countable 
union of $C^1$-regular $(n-1)$-dimensional submanifolds.
\end{itemize}

\smallskip

\noindent{\bf C. Blow-up analysis.}
The blow-ups of $u$ at a free 
boundary point $x_0$ satisfy the ensuing properties:
\begin{itemize}
\item[(i)] $\textup{BU}(x_0)\subseteq\cH_{\lambda(x_0)}$, 
the latter set being the positive cone of 
$\lambda(x_0)$-homogeneous local solutions to 
\eqref{e:ob-pb local} even with respect to 
$x_{n+1}$. 
Moreover, the possible values of the frequency $\lambda(x_0)$ 
lie in the set $\{1+s\}\times[2,+\infty)$ 
(cf.~\cite{CaSaSi08});

\item[(ii)] the blow-ups are unique both at every point of $\reg(u)$ 
(cf.~\cite{CaSaSi08}), and at every point of $\sing(u)$ 
for the Signorini problem $a=0$ (cf.~\cite{GaPe09}).
\end{itemize}

Despite these significant achievements, 
many issues on the analysis of the regularity of the free boundary
and the corresponding blow-ups of solutions to \eqref{e:ob-pb local}
remain still unsolved, even for the scalar Signorini problem.
The most striking 
fact is
that nothing is known about the global nature of the
free boundary, which in principle 
is not known to
have the right dimensionality of a boundary in 
$\R^n\times\{0\}$ ({\ie}, $n-1$), nor it is known to
retain any boundary-like structure (as far as we know,
$\Gamma(u)$ can be even fractal).
In particular, there are no results about the subset 
of free boundary points 
$\other(u)$, which are neither regular nor singular
(according to the definitions in literature).
On the other hand, explicit examples show that
$\other(u)$ is in general not empty, and indeed it may coincide with the full 
free boundary (cf. \S~\ref{s:blow-up})! 

\subsection{The main results of the paper}
In this paper we answer to some of the questions mentioned above,
such as that concerning the dimension of the free boundary, and
we give a comprehensive description of the set $\other(u)$ 
and $\sing(u)$ in the general case $a \in (-1,1)$ up to a null $\cH^{n-1}$ set.
Our results are already new for the case of the Signorini problem $a=0$
and extend in various directions what was previously known.
In particular, the short outcome of our analysis is the
global picture of the free boundary
of the thin obstacle problem as an $(n-1)$-dimensional set
with locally finite measure (in fact with
finite Minkowski content) satisfying almost everywhere
a similar stratification as for the 
classical obstacle problem (including some uniqueness
results of the blow-ups), 
cf.~\cite{Caffa77, Caffa98, Weiss, Monneau03}.

\medskip

We start off showing that the free boundary is 
$(n-1)$-dimensional in a strong
measure theoretic sense.

\begin{theorem}\label{t:misura}
Let $u$ be a solution to the thin obstacle problem 
\eqref{e:ob-pb local} in $B_1$.
Then, the free boundary $\Gamma(u)$ has locally finite
$(n-1)$-dimensional Minkowski content:
\ie, for every $K \subset \subset B_1'$ there
exists a constant $C(K)>0$ such that
\begin{equation}\label{e:misura1}
\cL^{n+1}\big(\cT_r(\Gamma(u) \cap K) \big) \leq C(K)\,r^2
\quad \forall\; r \in (0,1),
\end{equation}
where $\cT_r(E) := \{ x \in \R^{n+1}: \dist(x,E) <r\}$ for all 
$E\subseteq\R^n$.
\end{theorem}

Next, we prove the following geometric regularity result 
for the free boundary
establishing its $\cH^{n-1}$-rectifiability.

\begin{theorem}\label{t:rect}
Let $u$ be a solution to the
thin obstacle problem \eqref{e:ob-pb local}
in $B_1$. Then, there exist at most 
countably many  $C^1$-regular submanifolds $M_i$
of dimension $n-1$ in $\R^{n+1}$ such that 
\begin{equation}\label{e:rect1}
\cH^{n-1}\Big(\Gamma(u) \setminus \bigcup_{i\in\N} M_i\Big) = 0.
\end{equation}
\end{theorem}

The last result concerns one of the major open
question in the field, namely to determine the 
possible values of the frequency $\lambda(x_0)$, or 
equivalently the smallest set 
$J\subset(0,\infty)$ for which $\textup{BU}(x_0)
\subseteq\cup_{\lambda\in J}\cH_\lambda$
for every $x_0\in\Gamma(u)$,
recall that 
$\cH_\lambda$ denotes the set of $\lambda$-homogeneous solutions
to \eqref{e:ob-pb local}. 
As explained in C.~(i) above, it is known that
\[
\{2m, 2m-1+s\}_{m\in\N\setminus\{0\}}
\subseteq
J
\subseteq 
\{1+s\}\times[2,\infty).
\]
Moreover,  by definition
$\lambda(x_0) = 1+s$ for all
$x_0\in\reg(u)$ and 
$\lambda(x_0) \in2\N\setminus\{0\}$ for all
$x_0\in\sing(u)$.
In the following theorem we make a step forward to clarify
this stage.

\begin{theorem}\label{t:frequency}
Let $u$ be a solution to the lower dimensional obstacle 
problem \eqref{e:ob-pb local} in $B_1$. Then, there exists a 
subset $\Sigma(u)\subset \Gamma(u)$ with Hausdorff dimension 
at most $n-2$ such that
\[
\lambda(x_0) \in \{2m, 2m-1+s, 2m+2s\}_{m\in\N\setminus\{0\}}
\quad\forall\;x_0\in \Gamma(u)\setminus \Sigma(u).
\]
In addition, for $\cH^{n-1}$-a.e.~point 
$x_0\in \Gamma(u)\setminus \Sigma(u)$
with frequency
$\lambda(x_0)\in\{2m,2m-1+s\}_{m\in\N\setminus\{0\}}$
the blow-up of $u$ at $x_0$ is unique and depends on
two variables only: namely,
\[
\textup{BU}(x_0) = \Big\{\bar h_{\lambda(x_0)}(x\cdot e_{x_0},x_{n+1})\Big\},
\]
for some $e_{x_0} \in \R^{n+1}$ with
$|e_{x_0}|=1$ and $e_{x_0}\cdot e_{n+1} = 0$,
and $\bar h_{\lambda(x_0)}$ uniquely determined by $\lambda(x_0)$.
\end{theorem}

\subsection{Comments on the main results}\label{s:commnets}

A few remarks are in order.

\subsubsection{Finite measure}
The main consequence of Theorem~\ref{t:misura} 
is that the free boundary has locally finite $\cH^{n-1}$ measure:
\begin{equation}\label{e:misura2}
\cH^{n-1} \big(\Gamma(u)\cap K \big) <+\infty \quad \forall \; K \subset \subset 
\R^n.
\end{equation} 
Nevertheless, the estimate on the Minkowski content is significantly 
stronger: among the other consequences,
\eqref{e:misura1} implies, for instance, 
that the free boundary is nowhere dense.
In addition, Theorem~\ref{t:rect} establishes that the 
free boundary is a $\cH^{n-1}$-rectifiable set, a 
piece of information which cannot be deduced nor implies 
the estimate on the Hausdorff measure \eqref{e:misura2}.

The estimate on the Hausdorff 
dimension of the free boundary can be deduced
independently from Theorem~\ref{t:misura} 
and Theorem~\ref{t:rect} by a different
and more direct stratification argument 
(cf.~Theorem~\ref{t:straFICA}).

\subsubsection{Structure of the free boundary}
Theorem~\ref{t:rect} extends the analysis of 
the structure of the free boundary points to a subset of full measure 
of $\sing(u) \cup \other(u)$.
Note that the structure of the points in $\sing(u)$ for $a \neq 0$
had not been dealt with before in the literature.
Nevertheless, Theorem~\ref{t:rect} does not imply the pointwise results 
in B.~(i) {\&} (ii) for $\reg(u)$ and $\sing(u)$, for the latter 
set if $a=0$, because we prove a measure theoretic regularity 
property for $\Gamma(u)$, namely its $\cH^{n-1}$-rectifiability 
(cf. \eqref{e:rect1}).

\subsubsection{Frequency}
Points with frequencies $2m-1+s$ and $2m+2s$, 
with $m\in\N\setminus\{0,1\}$, belong 
to $\other(u)$, though it is not known whether 
they do exhaust such a set or not in general.
In other words, the problem of classifying all possible
frequencies for free boundary points is settled
by Theorem~\ref{t:frequency} only up to sets 
of dimension at most $n-2$, but it remains open pointwise.

Moreover, if on one hand there are examples
of free boundary points with frequency 
$2m$ and $2m-1+s$, on the other hand
there are no examples of points with frequency $2m+2s$.
In dimension $n=1$ one can show that such points
do not exist, that is
$\other(u)=\{2m-1+s\}_{m\in\N\setminus\{0,1\}}$ and 
also that $\Sigma(u)=\emptyset$
(see \S~\ref{s:blow-up} -- 
the case $a=0$ has been discussed in \cite{GaPe09}).
In higher dimensions it is then natural to 
conjecture the same results. 

The reason why we are unable to rule out points with 
frequencies $2m+2s$ if $n\geq2$
is related to the existence of $(2m+2s)$-homogeneous 
solutions with contact set $\Lambda = \R^n\times\{0\}$, which potentially
could arise as blow-ups in a free boundary point (with the free boundary 
disappearing in the limit).
This possibility might seem an apparent and 
striking discrepancy with the measure estimate and the 
structure result of Theorems~\ref{t:misura} 
and \ref{t:rect} and make these results in some sense
surprising (see \S~\ref{ss:optimality} for further comments).

Finally, the estimate on the Hausdorff dimension 
of $\Sigma(u)$ follows from its inclusion in the subset of 
points of $\Gamma(u)$ whose blow-ups have at most $(n-2)$ 
directions of invariance, 
for such a set the dimensional estimate is actually sharp 
(cf. Theorem~\ref{t:straFICA}).

\subsubsection{Blow-ups}
The uniqueness of blow-ups provided by
Theorem~\ref{t:frequency} at points
of the free boundary with frequency $2m$ and $2m-1+s$
univocally describes the infinitesimal behaviour 
of the solution $u$.
In particular, it shows that the solutions look
locally like a homogeneous function of a single horizontal 
variable and of $x_{n+1}$.
Note that, for any different choice of the 
renormalization of the rescalings \eqref{e:rescaling-1},
either the limit does not exist or it is a multiple of $\bar h_{\lambda}$,
thus justifying the notion of unique limiting profile.

For the prospective points with frequency $2m+2s$,
as a by-product of the results in Appendix~\ref{a:appendix}, we are
also able to classify all possible blow-ups.


\subsection{Concerning the proofs}
Our analysis is based on geometric 
measure theory techniques, which
exploit and develop some ideas recently  introduced
in the context of minimal surfaces theory.
The point of view we adopt is new in the theory
of free boundaries and we believe it has many 
potentialities for other related problems.

The proof is based on the ideas and the techniques
recently introduced by Naber--Valtorta \cite{NaVa1, NaVa2}
in the context of minimal surfaces and harmonic maps.
The main ingredients of our study are (a variant of)
Almgren's frequency function 
and the Peter Jones' number $\beta^{(n-1)}_\mu$ pertaining
to a suitable measure $\mu$ supported on the free boundary $\Gamma(u)$
(the terminology \textit{mean flatness} is also adopted in 
literature to term the $\beta$-numbers, since they provide 
an integral control of the flatness of the support of the 
underlying measure $\mu$, see \cite{AmFuPa}).
The starting point is the striking observation
by Naber--Valtorta \cite{NaVa1, NaVa2}
that the square power of the mean flatness can be 
controlled by an average of the oscillations of 
a monotone density.
Indeed, when this happens, a careful covering argument
\cite{NaVa1, NaVa2}, and recently developed rectifiability criteria
by David--Toro \cite{DaTo12}, Azzam--Tolsa
\cite{AzTo15} and Naber--Valtorta \cite{NaVa1, NaVa2}
lead to the local finiteness and the rectifiability of the singular sets of 
minimal surfaces and harmonic maps (see also the paper by De Lellis, Marchese,
Spadaro and Valtorta \cite{DMSV17} 
for an extension to a special case in higher co-dimension).

For our analysis of the thin obstacle problem,
we generalize and develop these approaches.
The starting point is an estimate of the mean flatness 
with respect to a Borel measures $\mu$ supported on the free boundary 
with the spatial oscillation of Almgren's frequency function
(cf. Proposition~\ref{p:mean-flatness vs freq}).
Note that the case of the frequency function
is different from the mass ratio of a minimal surface, because
the renormalization factor is intrinsically defined by the solution itself
(usually a variant of the $L^2$-norm at the boundary of a ball), 
instead of being purely dimensional. This requires a novel estimate
for the frequency of the solutions to the lower dimensional obstacle problem, 
which is based on a different set of spatial variations
and is proven in Proposition~\ref{p:D_x frequency}: here we follow closely 
ideas of \cite{DMSV17}, where an analogous estimate is proved in the context 
of multiple-valued functions as a result of this spatial variations of the
frequency.

A careful analysis of the rigidity properties of 
homogeneous solutions to the thin obstacle problem \eqref{e:ob-pb local}
(cp.~Proposition~\ref{p:rigidity}) is then necessary for our
argument. To this aim, as in general no growth estimate
from below for solutions 
from the free boundary are at disposal, it is mandatory for us 
to introduce the set of nodal points.

Using such rigidity results and the 
mentioned estimate on the mean flatness via the frequency
we use the covering argument and the discrete Reifenberg theorem
by Naber--Valtorta in \cite{NaVa1,NaVa2} in order to infer 
Theorem~\ref{t:misura}.
Then, Theorem~\ref{t:rect} is obtained by means of the
rectifiability criterion recently established by Azzam--Tolsa 
\cite{AzTo15} and indipendently by Naber--Valtorta in \cite{NaVa1,NaVa2}, 
while Theorem~\ref{t:frequency} is a consequence of Almgren's stratification 
principle (see, \textit{e.g.}, \cite{FMS-15}) and 
the classification of homogeneous solutions
of the PDE \eqref{e:ob-pb local} given in 
\S~\ref{s:blow-up} and \S~\ref{a:appendix}.

\subsection{Structure of the paper}

We start off introducing several preliminaries
in \S~\ref{s:preliminari}.
More precisely, in \S~\ref{s:prelim thin} we
collect the results concerning the
regularity of the solutions to the thin obstacle problem. 
In \S~\ref{s:frequency} we introduce the variant of the 
frequency function we are going to use and we derive  
several useful properties.
We then show in \S~\ref{s:frequency estimate} how to 
deduce from these an oscillation estimate of the frequency.
The aforementioned control of the flatness of the free 
boundary (defined in terms of the Peter
Jones' numbers),
with the oscillation of the frequency is established in 
Proposition~\ref{p:mean-flatness vs freq}.
Next, \S~\ref{s:rigidity} is devoted to the classification results 
for homogeneous solutions to the PDE in 
\eqref{e:ob-pb local} under several conditions
and to study the rigidity properties of almost homogeneous solutions.
Full proofs of the classification are provided in the 
Appendix~\ref{a:appendix}. 

We then proceed with the proofs of 
Theorems~\ref{t:misura}, \ref{t:rect} and 
\ref{t:frequency} in \S~\ref{s:misura}, 
\S~\ref{s:rect} and \S~\ref{s:blow-up}, respectively. 
In the corresponding section we recall the
analytical results 
that we exploit in the proofs, namely the 
discrete Reifenberg theorem by Naber--Valtorta \cite{NaVa1, NaVa2},
the rectifiability criterion by Azzam--Tolsa \cite{AzTo15}, 
and Almgren's stratification principle 
following the abstract version provided  in
our paper in collaboration with Marchese \cite{FMS-15}.

%
%

\section{Preliminaries on the thin obstacle problem}\label{s:preliminari}
In this section we recall some of the known results on the 
thin obstacle problem.

\subsection{Optimal regularity}\label{s:prelim thin}
The following is the main existence and regularity theorem
by Caffarelli, Salsa and Silvestre \cite{CaSaSi08}.

\begin{theorem}\label{t:reg}
For every $g\in \mathscr{A}_{1}$, there exists a unique
solution $u$ to the thin obstacle problem \eqref{e:ob-pb local}
in $B_1$.
Moreover, $\de_{x_i} u\in C^s(B_1)$ for $i=1,\ldots, n$,
$|x_{n+1}|^a\de_{x_{n+1}} u \in C^\alpha(B_1)$, $0<\alpha<1-s$,
and there exists a constant $C_{\ref{t:reg}}>0$ such that
\begin{equation}\label{e:C1_1/2 reg}
\|\nabla_\tau u\|_{C^{s}(B_{\sfrac{1}{2}})}+ 
\|\sgn(x_{n+1})\,|x_{n+1}|^a\de_{x_{n+1}} u\|_{C^{\alpha}(B_{\sfrac{1}{2}})}
\leq C_{\ref{t:reg}}\, \|u\|_{L^2(B_1, |x_{n+1}|^a\cL^{n+1})},
\end{equation}
where $\nabla_\tau u = (\de_{x_1}u, \ldots, \de_{x_n}u)$
is the horizontal gradient.
\end{theorem}

\begin{remark}\label{r:reg}
The estimates in \cite[Proposition~4.3]{CaSaSi08} are given
in terms of the $C^0$ norm of $u$
on the right hand side of the inequality.
Nevertheless, $u^{+}:=\max\{u,0\}$
and $u^{-}:=\max\{-u,0\}$ satisfy
$\div(|x_{n+1}|^a\nabla u^{\pm}(x)) \geq 0$ in $\mathcal{D}'(B_1')$.
Therefore, by the $L^\infty$-estimate in
\cite[Theorem~2.3.1]{FaKeSe82} we have that
\begin{align}\label{e:rmk reg}
\|u^{+}\|_{C^0(B_{\sfrac34})} +
\|u^{-}\|_{C^0(B_{\sfrac34})} \leq C\,\|u\|_{L^2(B_1, |x_{n+1}|^a\cL^{n+1})},
\end{align}
and then \eqref{e:C1_1/2 reg} follows by combining
\cite[Proposition~4.3]{CaSaSi08} and \eqref{e:rmk reg}.
\end{remark}

\begin{remark}\label{r:reg2}
For later purposes we also need the estimate
\begin{align}\label{e: stima reg cacona}
\sup_{x \in B_{\rho}} |\nabla u(x)\cdot x| \leq 
C\,\rho^{-\frac{n+1+a}{2}}\|u\|_{L^2(B_{2\rho}, |x_{n+1}|^a\cL^{n+1})},
\end{align}
which follows straightforwardly from \eqref{e:C1_1/2 reg}.
\end{remark}

In particular, the function $u$ is analytic in $\big\{x_{n+1} >0\big\} \cap B_1$
(see, \eg, \cite{Hopf32}) and
the following boundary conditions holds.

\begin{corollary}\label{c:bd conditions}
Let $u$ be a solution to the thin obstacle problem 
\eqref{e:ob-pb local} in $B_1$. Then, 
\begin{align}
\lim_{x_{n+1}\downarrow 0^+} x_{n+1}^a \de_{n+1} u(x',x_{n+1}) = 0
& \quad \text{for }\; (x',0) \in B_1'\;\text{with }\;
u(x',0) > 0,\label{e:bd condition1}\\
\lim_{x_{n+1}\downarrow 0^+} x_{n+1}^a \de_{n+1} u(x',x_{n+1}) \leq 0
& \quad \text{for }\; (x',0) \in B_1',\label{e:bd condition2}\\
u(x)\,\div(|x_{n+1}|^a\nabla u(x)) = 0 & \quad \text{in }\;\mathscr{D}'(B_1).\label{e:bd distribution}
\end{align}
\end{corollary}

\begin{proof}
Set for simplicity
$f(x',0):= \lim_{x_{n+1}\downarrow 0^+} x_{n+1}^a \de_{n+1} u(x',x_{n+1})$,
and note that by Theorem~\ref{t:reg} we have that $f \in C^\alpha(B_1')$, $0<\alpha<1-s$.
By the even symmetry of $u$, for every $\ph \in C^1_c(B_1)$
even symmetric we get the following: let $T:=-\div(|x_{n+1}|^a\nabla u(x))$ in $\mathscr{D}'(B_1)$, then
\begin{align*}
T(\ph(x))
& = 2 \int_{B_1^+} \nabla u(x)\cdot \nabla \ph(x)\,|x_{n+1}|^a\,\d x\\
&= 2 \lim_{\epsilon\downarrow0}
\int_{\{x_{n+1}\geq \epsilon\}\cap B_1}
\nabla u(x) \cdot \nabla \ph(x)\,|x_{n+1}|^a\,\d x\\
& \stackrel{\eqref{e:ob-pb local}}{=}
- 2 \lim_{\epsilon\downarrow0} \int_{\{x_{n+1}= \epsilon\} \cap
B_1} \de_{n+1}u (x',\epsilon)\,\ph(x',\epsilon)\,\epsilon^a\d x'\\
& = -2\int_{B_1'} f(x',0)\, \ph(x',0)\,\d x'.
\end{align*}
This shows that
\[
\div(|x_{n+1}|^a\nabla u(x)) = -2\, f(x',0)\, \cH^{n}\res B_{1}'.
\]
Thus, \eqref{e:bd condition1} and \eqref{e:bd condition2} 
follow directly from \eqref{e:ob-pb local}.
Moreover, $u(x)\,\div(|x_{n+1}|^a\nabla u(x))$ is well-defined
as a measure, and using \eqref{e:ob-pb local} we also infer 
\eqref{e:bd distribution}, because
$f(x',0)\,u(x', 0) = 0$ for all $(x',0) \in B_1'$.
\end{proof}

%
%

\subsection{The frequency function}\label{s:frequency}
As firstly noticed by Athanasopoulos, Caffarelli and 
Salsa in \cite{AtCaSa08}, one of the main 
quantities which are relevant to the analysis
of the solutions to the thin obstacle problem is
\textit{Almgren's frequency function}.
Several variants of the frequency function have
been introduced in the literature. For our purposes, we use
the analog of that introduced in \cite{DS2}
in the context of higher co-dimension minimal surfaces.

Let $\phi:[0,+\infty) \to [0,+\infty)$ be the function given by
\[
\phi(t) := 
\begin{cases}
1 & \text{for } \; 0\leq t \leq \frac{1}{2},\\
2\,(1-t) & \text{for } \; \frac{1}{2} < t \leq 1,\\
0 & \text{for } \; 1 < t .\\
\end{cases}
\]
We define the frequency of
a solution $u$ to \eqref{e:ob-pb local}
at a point $x_0 \in B_R'$ by
\[
I_u(x_0,r) := \frac{r D_u(x_0,r)}{H_u(x_0,r)}
\quad  \forall\; r<R-|x_0|,
\]
where
\[
D_u(x_0,r) := \int \phi\big(\textstyle{\frac{|x-x_0|}{r}}\big)\,|\nabla u(x)|^2\,
|x_{n+1}|^a\d x,
\]
and
\[
H_u(x_0,r) := - \int \phi'\Big(\textstyle{\frac{|x-x_0|}{r}}\Big)
\,\frac{u^2(x)}{|x-x_0|}\,|x_{n+1}|^a
\,\d x.
\]
Note that the frequency is well-defined as long as $H_u(x_0,r)>0$.
As $H_u(x_0,r) = 0$ implies $u\equiv 0$
by the analyticity of $u$ in $B_R\setminus B_R'$, 
we infer then that the frequency is always well-defined
for non-trivial solutions $u$.
For later convenience, we introduce also the notation
\[
E_u(x_0,r):= \int -\phi'\Big(\textstyle{\frac{|x-x_0|}{r}}\Big)
\frac{|x-x_0|}{r^2}
\Big(\nabla u(x) \cdot
\frac{x-x_0}{|x-x_0|}\Big)^2\,|x_{n+1}|^a\, \d x.
\]
In what follows, when $x_0 = 0$ we shall omit to write the base point 
$x_0$ in the notation of $I_u$, $D_u$, $H_u$ an $E_u$. 

\begin{remark}
The principal advantage of the frequency function $I_u(x_0,r)$ is that
it retains some average information of the solution $u$
on the annulus $B_{r}\setminus B_{\sfrac{r}{2}}(x_0)$,
whereas the classical Almgren's frequency function only
involves the $L^2$ norm of $u$ on the sphere $\de B_{r}(x_0)$.
\end{remark}

\begin{remark}\label{r:scaling}
If $u$ is a solution to the thin obstacle problem in $B_R$, then
for every $r \in (0,R-|x_0|)$, $x_0\in B_R'$ and for every $c>0$,
the function $v:B_1 \to \R$
\[
v(y) := c\,u(x_0 +r\,y) 
\]
solves \eqref{e:ob-pb local} in $B_1$ with respect to its own boundary conditions.
Moreover, $I_v(0,\rho) = I_u(x_0,\rho\,r)$ for every $\rho \in (0,1)$.
This shows that the frequency function is 
scaling invariant, and in the sequel we will use
this property repeatedly. 
\end{remark}

\subsection{Monotonicity of the frequency}
The following is a simple variant of the well-known 
monotonicity of the frequency (cf.~\cite{CaSi07}).

\begin{proposition}\label{p:monotonia+lower}
Let $u$ be a solution to the thin obstacle problem \eqref{e:ob-pb local} in $B_R$.
Then, for all $x_0 \in \Gamma(u)\cap B_R'$,
the map $(0,R-|x_0|)\ni r\mapsto I_u(x_0,r)$ is nondecreasing
and
\begin{equation}\label{e:monotonia freq}
I_u(x_0,r_1) - I_u(x_0,r_0)  =
\int_{r_0}^{r_1}\frac{2\,t}{H_u^2(x_0,t)} \Big(H_u(x_0,t) \; E_u(x_0,t) - D_u^2(x_0,t) \Big)
\,\d t
\end{equation} 
for $0 < r_0 < r_1 < R - |x_0|$.
Moreover, $I_u(x_0,\cdot)=\kappa$
for every $t\in (r_0,r_1)$ if and only if $u$ is 
$\kappa$-homogeneous with respect to $x_0$.
\end{proposition}

\begin{proof}
We start off collecting some useful identities:
\begin{gather}
D_u(x_0,t)=-\frac1t\int \phi'\big(\textstyle{\frac{|x-x_0|}{t}}\big)\,u(x)
\nabla u(x)\cdot\frac{x-x_0}{|x-x_0|}\,|x_{n+1}|^a\,\d x,\label{e:D}\\
H_u'(x_0,t)
:=\frac{\d}{\d t}H_u(x_0,t)
=\frac{n+a}{t}\,H_u(x_0,t)+2\,D_u(x_0,t),
\label{e:Hprime} \\
D_u'(x_0,t):=\frac{\d}{\d t}D_u(x_0,t)=\frac{n+a-1}{t} \, 
D_u(x_0,t)+2\,E_u(x_0,t).\label{e:Dprime}
\end{gather}
To show \eqref{e:D}, \eqref{e:Hprime} and \eqref{e:Dprime},
we assume without loss of generality that $x_0 = 0$.
For \eqref{e:D} we consider the vector field
$V(x) := \phi\big(\frac{|x|}{t}\big)\,u(x)\,\nabla u(x)\,|x_{n+1}|^a$.
Clearly $V$ has compact support, and $V \in C^{\infty}(\R^{n+1}\setminus B_1',\R^n)$
by Theorem~\ref{t:reg}. Moreover, for $x_{n+1}\neq 0$
\[
V(x)\cdot e_{n+1}=\phi\big(\frac{|x|}{t}\big)\,{u_{x_0}}(x)\,
\frac{\partial u_{x_0}}{\partial x_{n+1}}(x)\,|x_{n+1}|^a\,,
\]
thus, $\lim_{y\downarrow (x',0)}V(y)\cdot e_{n+1}=0$. Indeed, 
if $(x',0)\in\Lambda_\varphi(u)$ it suffices to take into account the one-sided $C^{1,\alpha}$
regularity of $u$ in Theorem~\ref{t:reg} to conclude
\[
\lim_{y\downarrow (x',0)}u(y)|y_{n+1}|^a=0.
\]
Instead, if $(x',0)\notin\Lambda_\varphi(u)$ we use \eqref{e:bd condition1} in 
Corollary~\ref{c:bd conditions}. 
Thus, the distributional divergence of $V$ is the $L^1$ function given by
\begin{align*}
\div\, V (x) & = \phi\big(\textstyle{\frac{|x|}{t}}\big)|\nabla u(x)|^2\,|x_{n+1}|^a
+ \phi'\big(\textstyle{\frac{|x|}{t}}\big)u(x)
\nabla u(x)\cdot\frac{x}{t\,|x|}\,|x_{n+1}|^a.
\end{align*}
Therefore, \eqref{e:D} follows from the divergence theorem by taking into 
account that $V$ is compactly supported.

Next \eqref{e:Hprime} is a consequence of \eqref{e:D} 
and the direct computation
\begin{align*}
H_u'(t) &= \frac{\d}{\d t} \left(-t^{n+a} \int \phi'(|y|)\, \frac{u^2(t\,y)}{|y|}\,|y_{n+1}|^a \,\d y\right)\\
& = \frac{n+a}{t}\,H_u(t) - 2\,t^{n+a} \int \phi'(|y|)\, \frac{u(t\,y)\,\nabla
u(t\,y)\cdot y}{|y|}\,|y_{n+1}|^a \,\d y\\
& \stackrel{\eqref{e:D}}{=} \frac{n+a}{t}\,H_u(t) + 2\, D_u(t).
\end{align*}
Finally, to prove \eqref{e:Dprime} we consider the vector field
\[
W(x) = \left(\frac{|\nabla u|^2}{2}x-(\nabla u\cdot x)\nabla u\right)
\phi\Big(\textstyle{\frac{|x|}{t}}\Big)|x_{n+1}|^a.
\]
By Theorem~\ref{t:reg} we have that $W \in C^0_c(B_1,\R^n) \cap 
C^\infty(B_1\setminus B_1',\R^n)$.
Moreover, Corollary~\ref{c:bd conditions} implies that
$W (x',0) \cdot e_{n+1} = 0$ for all $(x',0) \in B_1'$.
Thus $\div\,W$ has no singular part in $B_1'$, and we 
can compute pointwise
\begin{align*}
\div\, W (x) & =  \phi'\big(\textstyle{\frac{|x|}{t}}\big)\cdot
\,\frac{x}{t\,|x|} \Big(\frac{|\nabla u|^2}{2}x-(\nabla u\cdot x)\nabla u\Big)
|x_{n+1}|^a + \phi\big(\textstyle{\frac{|x|}{t}}\big)\,\frac{n+a-1}{2}\,|\nabla 
 u(x)|^2|x_{n+1}|^a.
\end{align*}
Therefore, we infer that
\[
0=\int\div\,W(x)\,\d x =\int \phi'\Big(\textstyle{\frac{|x|}{t}}\Big)
\frac{|x|}{2\,t}\,|\nabla u(x)|^2|x_{n+1}|^a\,\d x +
t\,E_u(t) + \frac{n+a-1}{2}\,D_u(t),
\]
and we conclude \eqref{e:Dprime} by direct differentiation
\[
D_u'(t)=-\int \phi'\Big(\textstyle{\frac{|x|}{t}}\Big)
\frac{|x|}{t^2}\,|\nabla u(x)|^2|x_{n+1}|^a\,\d x.
\]

By collecting \eqref{e:Hprime} and \eqref{e:Dprime}, we finally compute
the derivative of $\log I_u(t)$: 
\begin{equation*}
\frac{I_u'(t)}{I_u(t)}=\frac 1t+\frac{D_u'(t)}{D_u(t)}-\frac{H_u'(t)}{H_u(t)}
 =2\,\frac{E_u(t)}{D_u(t)}-2\frac{D_u(t)}{H_u(t)}.
\end{equation*}
In particular, identity \eqref{e:monotonia freq} follows
at once by multiplying by $I_u(t)$ and by integrating over $(r_0,r_1)$.
In addition, by the Cauchy--Schwarz inequality, 
$r\mapsto I_u(r)$ is
non-decreasing. Finally, 
if $I_u(t) = k$ for every $t\in (r_0,r_1)$, then
\[
H_u(t) \, E_u(t) = D_u^2(t) \quad\forall \,t\in (r_0,r_1).
\]
In particular, by the equality case in the Cauchy--Schwarz
inequality, we deduce that there exists a constant $\lambda\in\R$ such
that
\[
\nabla u (x) \cdot x = \lambda\,u(x) \quad \forall\;x\in B_{r_1}\setminus 
B_{\sfrac{r_0}2},
\]
{\ie}~$u(x) = |x|^\lambda u(\sfrac{x}{|x|})$  for all 
$x\in B_{r_1}\setminus B_{\sfrac{r_0}2}$.
It then follows that $\lambda = k$
and by analyticity we conclude that $u$ is $k$-homogeneous in the whole $B_R$.
\end{proof}

From the monotonicity of the frequency, we infer the following consequences.

\begin{corollary}\label{c:H}
Let $u$ be a solution to the thin obstacle problem \eqref{e:ob-pb local} in $B_R$.
Then, for all $x_0 \in \Gamma(u) \cap B_R'$ and $0 < r_0 < r_1 < R - |x_0|$, we have
\begin{equation}\label{e:monotonia H}
\frac{H_u(x_0,r_1)}{r_1^{n+a}} = \frac{H_u(x_0,r_0)}{r_0^{n+a}}\,
e^{2\int_{r_0}^{r_1} \frac{I_u(x_0,t)}{t} \d t}.
\end{equation}
In particular, if $A_1 \leq I(x_0,t) \leq A_2$ 
for every $t \in (r_0, r_1)$, then 
\begin{gather}
(r_0, r_1)\ni r\mapsto\frac{H_u(x_0,r)}{r^{n+a + 2A_2}}
\quad\text{is monotone decreasing},\label{e:H1}\\
(r_0, r_1)\ni r\mapsto\frac{H_u(x_0,r)}{r^{n+a + 2A_1}}
\quad\text{is monotone increasing}.\label{e:H2}
\end{gather}
Moreover,
\begin{align}\label{e:L2 vs H}
\int_{B_r(x_0)}|u|^2|x_{n+1}|^a\,\d x\leq r\,H_u(x_0,r).
\end{align}
\end{corollary}

\begin{proof}
The proof of \eqref{e:monotonia H}
(and hence of \eqref{e:H1} and \eqref{e:H2})
follows from the differential equation 
\eqref{e:Hprime}.
The proof of \eqref{e:L2 vs H} is now a direct consequence:
\begin{align*}
\int_{B_r(x_0)}|u|^2|x_{n+1}|^a\,\d x
&= \sum_{k\in\N} \int_{B_{\sfrac{r}{2^k}}
\setminus B_{\sfrac{r}{2^{k+1}}}(x_0)}|u|^2|x_{n+1}|^a\,\d x\\
&\leq \sum_{k\in\N} \frac{r}{2^k}\,H_u\big(x_0,\sfrac{r}{2^k}\big)
\leq r\,H_u(x_0,r),
\end{align*}
where in the last inequality we used that
$H_u(x_0,s)\leq H_u(x_0,r)$ for 
$s\leq r$
by \eqref{e:H2}.
\end{proof}

\subsection{Lower bound on the frequency and compactness}
We first show that the frequency of a solution to \eqref{e:ob-pb local} 
 at free boundary points
is bounded from below by a universal constant.

\begin{lemma}\label{l:freq lower}
There exists a dimensional constant $C_{\ref{l:freq lower}}>0$
such that, for every solution $u$ to the thin obstacle problem 
\eqref{e:ob-pb local} in $B_R$ and for every $x_0 \in \Gamma(u)$, 
we have
\begin{equation}\label{e:freq lower}
I_u(x_0,r) \geq C_{\ref{l:freq lower}} \quad \forall\; r\in (0, R-|x_0|).
\end{equation}
\end{lemma}

\begin{proof}
By the co-area formula for
Lipschitz functions we check 
that 
\begin{align}
H_u(x_0,r) = 2\int_{\frac{r}{2}}^r \frac{dt}{t}
\int_{\de B_t(x_0)} |u(x)|^2\,|x_{n+1}|^a\,\d \cH^{n}(x),
\label{e:formule integrate2}
\end{align}
and
\begin{align*}
D_u(x_0,r) & = \int_0^{\frac r2}dt
\int_{\partial B_t(x_0)}|\nabla u(x)|^2|x_{n+1}|^a\,\d \cH^{n}(x)\\
&\quad +\frac 2r\int_{\frac r2}^{r}
\int_{\partial B_t(x_0)}(r-t)|\nabla u(x)|^2\,|x_{n+1}|^a\,\d 
\cH^{n}(x).
\end{align*}
An integration by parts then gives
\begin{align}
D_u(x_0,r) = \frac{2}{r} \int_{\frac{r}{2}}^r dt\int_{B_t(x_0)} |\nabla u(x)|^2\,
|x_{n+1}|^a\,\d x.\label{e:formule integrate1}
\end{align}
Therefore, we can conclude 
the lower bound \eqref{e:freq lower} by using the
Poincar\'e inequality in \cite[Lemma~2.13]{CaSaSi08}
\[
\frac{1}{t}
\int_{\de B_t(x_0)} |u(x)|^2\,|x_{n+1}|^a\,\d \cH^{n}(x)
\leq C\, \int_{B_t(x_0)} |\nabla u(x)|^2\,
|x_{n+1}|^a\,\d x.\qedhere
\]
\end{proof}

We can then give the following compactness 
result which will be instrumental for the analysis we 
develop. To this aim it is mandatory to introduce 
the nodal set of $u$:
\[
\mathcal{N}(u) := \Big\{
(x',0)\in B_R'\,:\,u(x',0) = |\nabla_\tau u(x',0)|
= \lim_{t\downarrow 0^+}t^a\de_{n+1}u(x',t) = 0
\Big\}.
\]
Notice that $\Gamma(u) \subseteq \mathcal{N}(u)$
by Corollary~\ref{c:bd conditions}.

\begin{corollary}\label{c:compactness}
Let $(u_k)_{k\in\N}$ be a sequence of solutions
to the thin obstacle problem \eqref{e:ob-pb local}
in $B_1$, with 
$\sup_{k} I_{u_k}(1) <+\infty$, $H_{u_k}(1)\leq 1$
and $0 \in \Gamma(u_k)$ for every $k \in \N$.
Then, there exist a subsequence $(u_{k_j})_{j\in\N} \subset
(u_{k})_{k\in\N}$ and a solution $u_0$ to the thin obstacle
problem in $B_1$ such that as $j\to \infty$
\begin{gather}
u_{k_j} \to u_{0} \quad \text{in 
$H^1_{\loc}(B_1, |x_{n+1}|^a\,\cL^{n+1})$,}
\label{e:cpt1}\\
\nabla_\tau u_{k_j} \to \nabla_\tau u_{0} \quad 
\text{in $C^{\alpha}_{\loc}(B_1)$, $\forall\;\alpha <s$}, 
\label{e:cpt2}\\
\sgn(x_{n+1})\,|x_{n+1}|^a\de_{x_{n+1}} u_{k_j} \to 
\sgn(x_{n+1})\,|x_{n+1}|^a\de_{x_{n+1}} u_{0} 
\text{ in $C^{\alpha}_{\loc}(B_1)$,
$\forall\;\alpha <1-s$,}
\label{e:cpt3}\\
u_{k_j} \to u_{0} \quad \text{in  
$C^{\alpha}_{\loc}(B_1)$, $\forall\;\alpha <\min\{1,2 s\}$.} 
\label{e:cpt4}
\end{gather}
Moreover, if there is a sequence of points $x_{k_j}\in \Gamma(u_{k_j})$
such that $x_{k_j}\to x_0\in B_1$, then
\begin{gather}
x_0\in\mathcal{N}(u_0).\label{e:cpt5}
\end{gather}
\end{corollary}

\begin{proof}
For every $t<1$, we have that
\begin{align*}
\int_{B_t}|\nabla u_k(x)|^2\,|x_{n+1}|^a\,dx 
& \leq \frac{D_{u_k}(1)}{2(1-t)}
= \frac{I_{u_k}(1)\,H_{u_k}(1)}{2(1-t)}
\leq \frac{M}{2(1-t)},
\end{align*}
where we have set for convenience $M:=\sup_{k} I_{u_k}(0,1)$.
Moreover, from \eqref{e:L2 vs H}
we have that $\|u_k\|_{L^2(B_1,|x_{n+1}|^a\,\cL^{n+1})} 
\leq H_{u_k}(1) \leq 1$. The sequence $(u_{k})_{k\in\N}$
is equi-bounded in $H^1(B_t, |x_{n+1}|^a\,\cL^{n+1})$
for every $t<1$.
Therefore, 
\eqref{e:cpt1} -- \eqref{e:cpt4} follow from Theorem~\ref{t:reg} (cf.~also 
\cite[Lemma~4.4]{CaSaSi08}).
Moreover, since $x_{k_j}\in \Gamma(u_{k_j}) \subseteq 
\mathcal{N}(u_{k_j})$,
\eqref{e:cpt5} follows from \eqref{e:cpt2}--\eqref{e:cpt4}.
\end{proof}

\subsection{Blow-up profiles}
An important consequence of the monotonicity of the frequency
in Proposition~\ref{p:monotonia+lower}
is the existence of blow-up profiles.
For $u:B_R\to\R$ solution of \eqref{e:ob-pb local} we introduce the rescalings
\begin{equation}\label{e:rescaling0}
u_{x_0,r} (y) := 
\frac{r^{\frac{n+a}{2}}\,u(ry+x_0)}{H^{\sfrac12}(x_0,r)}
\quad \forall \; r \in(0,R - |x_0|), \;\forall\; y \in 
B_{\frac{R-|x_0|}{r}}.
\end{equation}

\begin{proposition}\label{p:blow-up}
Let $u$ be a solution to the thin
obstacle problem \eqref{e:ob-pb local} in $B_R$.
Then, for every $x_0 \in \Gamma(u)$ and for every sequence
of numbers $(r_j)_{j\in\N} \subset (0,1-|x_0|)$ with
$r_j \downarrow 0$, there exists a subsequence
$(r_{j_k})_{k\in \N}$ and function
$u_0 \in H^1_{\loc}(\R^{n+1},|x_{n+1}|^a\cL^{n+1})$ such that
$u_0$ satisfies \eqref{e:ob-pb local}, $u_0$ is homogeneous of
degree $I(x_0,0^+)$ and
\begin{equation}\label{e:blow-up}
u_{x_0,r_{j_k}} \to u_0 \quad \text{in }\; C^{0}_{\loc}(\R^{n+1})\cap
H^1_{\loc}(\R^{n+1}, |x_{n+1}|^a\cL^{n+1}).
\end{equation}
\end{proposition}

\begin{proof}
For every $\ell>0$, by Remark~\ref{r:scaling} we have
$I_{u_{x_0,r_j}}(\ell) = I_u(x_0,r_j\,\ell) \leq I_{u}(x_0,1-|x_0|)$.
Therefore, from Corollary~\ref{c:H} we infer that
there exists a constant $C=C(\ell)>0$ such that
\[
H_{u_{x_0,r_j}}(\ell) \leq C\, H_{u_{x_0,r_j}}(1) = C
\qquad\forall \; r_j < \ell^{-1}(R-|x_0|).
\]
We can then use Corollary~\ref{c:compactness}
and a diagonal argument to infer the existence of a subsequence
$(r_{j_k})_{k\in\N}$ and a solution $u_0$
such that \eqref{e:blow-up} holds.
We only need to show that $u_0$ is homogeneous.
To this aim we notice that, by taking into account
Lemma~\ref{l:freq lower}, we have for every $\ell>0$
\[
I_{u_0}(\ell) = \lim_{k\to \infty} 
I_{u_{x_0,r_{j_k}}}(\ell) = 
\lim_{k\to \infty} I_u(x_0,r_{j_k}\,\ell) = I_u(x_0,0^+)
\geq C_{\ref{l:freq lower}},
\]
In particular, by Proposition~\ref{p:monotonia+lower}
we conclude the homogeneity of $u_0$ of degree $I_u(x_0,0^+)$.
\end{proof}

\begin{corollary}\label{c:minima freq}
Let $u$ be a solution to the thin obstacle problem 
\eqref{e:ob-pb local} in $B_R$ . Then,
\begin{equation}\label{e:lower bound}
I_u(x_0,r) \geq 1+s \quad \forall\; x_0 \in \Gamma(u),\; \forall\; r \in(0,R-|x_0|).
\end{equation}
\end{corollary}

\begin{proof}
We consider the rescaling $u_{x_0,r_j}$ and a blow-up limit $u_0$.
By Proposition~\ref{p:blow-up} we know that $u_0$ is homogeneous
of degree $I_u(x_0,0^+)$. 
Since solutions to \eqref{e:ob-pb local}
are $C^{1,s}(B_R)$ (cf.~\cite{CaSaSi08}), we easily conclude that
$I_u(x_0,0^+) \geq 1+s$ and \eqref{e:lower bound} follows by monotonicity.
\end{proof}

\begin{remark}
In general the limiting profile $u_0$ is not known to be unique.
Uniqueness for $\cH^{n-1}$-almost every free boundary point
with infinitesimal homogeneity $2m$ and $2m-1+s$
will be established in
Theorem~\ref{t:frequency}, while uniqueness at every regular
point follows from \cite{CaSaSi08} (see also \cite{FoSp16, Geraci}
for an approach via the epiperimetric inequality) and 
at every singular point for $s=\sfrac12$ from \cite{GaPe09}.
\end{remark}

\begin{remark}\label{r:freq modificata}

It is more common in the literature to define
the blow-up rescalings $\bar u_{x_0,r}$ as in 
\eqref{e:rescaling-1}.
Nevertheless, by the same computations above, one can show that
the height function $h_u(x_0,t) := \int_{\de B_t(x_0)} u^2\,\d\cH^{n}$
satisfies the analogous monotonicity properties of Corollary~\ref{c:H}
(see \cite{CaSaSi08})
and moreover by \eqref{e:formule integrate2} it is comparable to 
$H_u(x_0,t)$ (with a constant depending only on an upper bound
of the frequency).
In particular, this implies that the blow-ups with respect 
to these two 
different renormalizations only differ by a constant and all the results 
concerning them (\textit{e.g.}~the uniqueness)
can be indifferently proven for either of the two definitions.

Due to our definition of the frequency, in the sequel
we will always consider the rescalings 
defined in \eqref{e:rescaling0}.
\end{remark}

%
%

\section{Main estimates on the frequency}\label{s:frequency estimate}

In this section we prove the principal estimates on the frequency
that we are going to exploit in the sequel.

\begin{lemma}\label{l:lim uniforme}
For every $A>0$ there exists $C_{\ref{l:lim uniforme}}= 
C_{\ref{l:lim uniforme}} (A) >0$ such that,
if $u$ is a solution to the thin obstacle problem 
\eqref{e:ob-pb local} in $B_{2r}(x_0)$,
with $r>0$, $x_0 \in \Gamma(u)$ and $I_u(x_0,2r) \leq A$, then
for every $x \in B'_{\sfrac{r}{2}}(x_0)$ 
\begin{gather}
\label{e:H limitato}
\frac{1}{C_{\ref{l:lim uniforme}}} \leq \frac{H_u(x_0,r)}{H_u(x,r)} \leq 
C_{\ref{l:lim uniforme}}
\and
\frac{1}{C_{\ref{l:lim uniforme}}}
\leq \frac{D_u(x_0,r)}{D_u(x,r)} \leq C_{\ref{l:lim 
uniforme}},\\
\Big\vert I_u(x_0,r) - I_u(x,r) \Big\vert
\leq C_{\ref{l:lim uniforme}}.\label{e:I limitato}
\end{gather}
\end{lemma}

\begin{proof}
By rescaling it is enough to consider the case $x_0=0$,
$r =1$ and $H_{u}(0,1) =1$ (cf.~Remark~\ref{r:scaling}).
In oder to prove \eqref{e:H limitato}, we argue by 
contradiction:
assume there exists functions $u_k$ 
and points $x_k\in B'_{\sfrac12}$
contradicting the first inequality \eqref{e:H limitato}, {\ie}
\[
\lim_{k\to +\infty}{H_{u_k}(x_k,1)} \in \{0, +\infty\}.
\]
Note that, since $I_{u_k}(0,2)\leq A$, it follows from \eqref{e:H1}
that $H_{u_k}(0,2)\leq 2^{n+a+2A}$.
In particular, we can apply Corollary \ref{c:compactness} and
(up to passing to a subsequence, not relabeled),
there exist $u_\infty$ and $x_\infty\in \bar B_{\sfrac12}$
such that $u_k\to u_\infty$
in $H^1_{\loc}(B_2, |x_{n+1}|^a\cL^{n+1})$
and $x_k\to x_\infty\in \bar B'_{\sfrac12}$, with 
$u_\infty$ solution to the thin obstacle problem in $B_R$ for every 
$R<2$.
By the strong convergence of $u_k$ to $u_\infty$ we then deduce that
$H_{u_\infty}(x_\infty,1)\in \{0,\infty\} \cap \R = \{0\}$.
Given that $u_\infty$ is analytical in $B_{2}\setminus \{x_{n+1} =0\}$,
by unique continuation we conclude that $u_\infty \equiv 0$ in $B_2$, against 
the assumption $H_{u_\infty}(0,1) = \lim_k H_{u_k}(0,1) =1$.

The second inequality in \eqref{e:H limitato} is proven by the same argument.
Indeed, under the same assumption $H_u(0,1)=1$, considering that 
$0\in \Gamma(u)$, we have that $D_u(0,1) = I_u(0,1) \in [1+s,A]$.
Therefore, given a sequence $u_k$ contradicting the claim, we deduce the 
existence of a solution $u_\infty$ such that $0=D_{u_\infty}(0,1) = \lim_{k} 
D_{u_\infty}(0,1)\in [1+s, A]$, which is impossible.

Finally, \eqref{e:I limitato} follows straightforwardly from \eqref{e:H 
limitato}:
\begin{align*}
\Big\vert I_u(0,1) - I_u(x,1) \Big\vert
& = \Big\vert \frac{D_u(0,1)}{H_u(x,1)} 
\left( \frac{H_u(x,1)}{H_u(0,1)} - 
\frac{D_u(x,1)}{D_u(0,1)}\right)  
\Big\vert \stackrel{\eqref{e:H limitato}}{\leq} C.
\end{align*}
\end{proof}

\begin{lemma}\label{l:monotonia}
For every $A>0$ there exists $C_{\ref{l:monotonia}} =
C_{\ref{l:monotonia}}(A)>0$ such that,
if $u$ is a solution to the thin obstacle problem 
\eqref{e:ob-pb local} in $B_{2r_1}(x_0)\subset \R^{n+1}$ with 
$x_0\in\Gamma(u)$ and $I_u(x_0,2r_1) \leq A$, then
for every $r_0\in (\sfrac{r_1}{8}, r_1)$
\begin{align}\label{e:monotonia con resto}
\int_{B_{r_1}(x_0)\setminus B_{r_0}(x_0)} 
\Big(\nabla u(z) \cdot (z - x_0) - &
I_u(x_0,r_0)\,u(z)\Big)^2 \frac{|z_{n+1}|^a}{|z-x_0|}\,\d z\notag\\
&\leq C_{\ref{l:monotonia}} H_u(x_0,2\,r_1)\,\big(I_u(x_0, 2\,r_1) - I_u(x_0, r_0)\big).
\end{align}
\end{lemma}

\begin{proof}
By rescaling, it suffices to prove the lemma for $x_0 = 0$
and $r_1 = 1$.
We start off with the following computation:
\begin{align}\label{e:I mon}
2 \int_{B_t\setminus B_{\sfrac{t}{2}}} &
\Big(\nabla u(z) \cdot z - I_u(t)\,u(z)\Big)^2 
\frac{|z_{n+1}|^a}{|z|} \d z\notag\\
&= \int - \phi'\big(\textstyle{\frac{|z|}{t}}\big)
\Big(\nabla u(z) \cdot z - I_u(t)\,u(z)\Big)^2 
\frac{|z_{n+1}|^a}{|z|} \d z
\notag\allowdisplaybreaks\\
&= t^2 E_u(t) - 2\, t\,I_u(t)\,D_u(t) + I_u^2(t) H_u(t)\notag
\allowdisplaybreaks\\
&= \frac{t^2}{H_u(t)}\,\Big[E_u(t) H_u(t) - D_u(t)^2\Big]
\stackrel{\eqref{e:monotonia freq}}{=} \frac{t}{2}\,I_u'(t)\,H_u(t).
\end{align}
We now use the following integral estimate (whose elementary proof
is left to the readers)
\begin{align}\label{e:fubini}
\int_{B_{1}\setminus B_{r_0}} f(z) \d z& \leq 
r_0^{-1}\int_{r_0}^{2} \int_{B_t\setminus B_{\sfrac{t}{2}}}
f(z)\, \d z\,\d t
\quad\text{$\forall$ $f\geq0$, $r_0\leq 1$},
\end{align}
in order to deduce
\begin{align}
\int_{B_{1}\setminus B_{r_0}} &
\Big(\nabla u(z) \cdot z - 
I_u(r_0)\,u(z)\Big)^2 \frac{|z_{n+1}|^a}{|z|}\,\d z\notag
\allowdisplaybreaks\\ & 
\stackrel{\eqref{e:fubini}}{\leq}
r_0^{-1}\int_{r_0}^{2} \int_{B_t\setminus B_{\sfrac{t}{2}}}
\Big(\nabla u(z) \cdot z - 
I_u(r_0)\,u(z)\Big)^2 \frac{|z_{n+1}|^a}{|z|}\,\d z\,\d t\notag
\allowdisplaybreaks\\
&\leq 2\,
r_0^{-1}\int_{r_0}^{2} \int_{B_t\setminus B_{\sfrac{t}{2}}}
\Big[\big(\nabla u(z) \cdot z - 
I_u(t)\,u(z)\big)^2 
+\big(I_u(t) - I_u(r_0)\big)^2u^2(z)\Big]
\frac{|z_{n+1}|^a}{|z|}\,\d z\,\d t\notag
\allowdisplaybreaks\\
& 
\stackrel{\eqref{e:I mon}}{\leq}
r_0^{-1}\int_{r_0}^{2} \frac{t\, H_u(t)}{2}\, I_u'(t)\, \d t
+ 2\,
r_0^{-1}\,\big(I_u(2) - I_u(r_0)\big)^2\,
\int_{r_0}^{2} H_u(t)\,\d t.\label{e:prima stima0}
\end{align}
Now recall that by \eqref{e:H2} we have that
$H_u(t) \leq H_u(2)$ for all $t \leq 2$.
Hence, from \eqref{e:prima stima0} we get
\begin{align*}
\int_{B_{1}\setminus B_{r_0}} &
\Big(\nabla u(z) \cdot z - 
I_u(r_0)\,u(z)\Big)^2 \frac{|z_{n+1}|^a}{|z|}\,\d z\notag\\ 
& \leq
r_0^{-1}\,H_u(2) \int_{r_0}^{2} I_u'(t)\, \d t+4
r_0^{-1} \,H_u(2) \big(I_u(2) - I_u(r_0)\big)^2\leq C\,H_u(2)\,\big(I_u(2) - 
I_u(r_0)\big),
\end{align*}
where we used that $r_0\geq \frac{1}{8}$ and 
and $I_u(2) \leq A$.
\end{proof}

\subsection{Oscillation estimate of the frequency}
We introduce the following notation for the radial variation of the frequency
at a point $x \in \Gamma(u)$: given $0<\rho<r$, we set
\[
\Delta^r_{\rho}(x) := I_u(x,r) - I_u(x,\rho).
\]
The following lemma shows how the spatial oscillation of the frequency in two 
nearby points at a given scale is in turn controlled by the radial variations
at comparable scales. Here, we exploit for the thin obstacle problem an argument 
introduced in \cite[Theorem~4.2]{DMSV17} for multiple-valued functions.

\begin{proposition}\label{p:D_x frequency}
For every $A>0$ there exists 
$C_{\ref{p:D_x frequency}}(A)>0$ such that,
if $\rho>0$, $R>6$
and $u : B_{4R\rho}(x_0) \to \R$ is a solution to the thin
obstacle problem \eqref{e:ob-pb local} 
in $B_{4R\rho}$, with $x_0\in\Gamma(u)$ and 
$I_u(x_0,4R\rho) \leq A$, then
\begin{equation}\label{e:D_x frequency}
\big\vert I_u\big(x_1, R\rho\big) - I_u\big(x_2, R\rho\big)\big\vert
\leq C_{\ref{p:D_x frequency}} \,
\left[\Big(\Delta^{2(R+2)\rho}_{\sfrac{(R-4) \rho}{2}}(x_1)\Big)^{\sfrac{1}{2}} 
+ 
\Big(\Delta^{2(R+2)\rho}_{\sfrac{(R-4)\rho}{2}}(x_2)\Big)^{\sfrac{1}{2}}\right],
\end{equation}
for every $x_1, x_2 \in B'_{\rho}$.
\end{proposition}

\begin{proof}
\noindent{\bf 1.}
Without loss of generality, we show the proposition
for $x_0=0$ and $\rho =1$.
The proof is based on estimating the tangential derivative
of the frequency function $x\mapsto I_u(x,t)$ for a fixed radius $t>0$.
Thus, we start off noticing that the functions $x\mapsto H_u(x,t)$
and $x\mapsto D_u(x,t)$ are differentiable and, for every 
$e \in \R^{n+1}$ with $e \cdot e_{n+1} = 0$,
we have that
\begin{align}\label{e:dH}
\de_e H_u(x,t) = - 2 \, \int \phi'\big(\textstyle{\frac{|y|}{t}}\big)\,
u(y+x)\, \de_e u(y+x)\,\frac{|y_{n+1}|^a}{|y|} \,\d y,
\end{align}
and
\begin{align}\label{e:dD}
 \de_e D_u(x,t) & = 2\int 
\phi\big(\textstyle{\frac{|y|}{t}}\big)
\nabla u(y+x)  \cdot \nabla (\de_e u)(y+x)\, |y_{n+1}|^a\,\d y\notag\\
& = - 2\,t^{-1}\int \phi'\big(\textstyle{\frac{|y|}{t}}\big)\,
\de_e u(y+x)\, \nabla u(y+x) \cdot\frac{y}{|y|}\,|y_{n+1}|^a
\,\d y,
\end{align}
where the second equality follows from the divergence theorem
applied to the vector field $V(y):= \phi\big(\textstyle{\frac{|y|}{t}}\big)\,
\de_e u(y+x)\, |y_{n+1}|^a \nabla u(y+x)$
(note that $V \in C^{0}_c(B_t(x))$ and by Theorem~\ref{t:reg} 
and by Corollary~\ref{c:bd conditions} the divergence of $V$ 
does not concentrate on $B_1'$).
We consider next $e := x_2 - x_1$, and set
\[
\mathcal{E}_i(z) := \nabla u(z) \cdot(z-x_i) -  I_u(x_i,t)\,u(z)
\quad\text{for $i=1,2$,}
\]
\[
\Delta I:= I_u(x_1,t) - I_u(x_2,t)
\and
\Delta\mathcal{E}(z):=\mathcal{E}_1(z) - \mathcal{E}_2(z).
\]
Then, we have that
$\de_e u(z) = \Delta I \cdot u(z) + \Delta\mathcal{E}(z)$ and from 
\eqref{e:D}, 
\eqref{e:dH} -- \eqref{e:dD} we get also
\begin{align*}
\de_e D_u(x,t) &= 2\,\Delta I\cdot D_u(x,t)
-2\,t^{-1} \int \phi'\big(\textstyle{\frac{|y|}{t}}\big)  \, 
\Delta\mathcal{E}(y+x) \;
\nabla u (y+x) \cdot \frac{y}{|y|}\,|y_{n+1}|^a\,\d y,
\end{align*}
and
\begin{align*}
\de_e H_u(x,t) &= 2\Delta I \cdot H_u(x,t)
-2 \int \phi'\big(\textstyle{\frac{|y|}{t}}\big)  \, 
\Delta\mathcal{E}(y+x) \;
u(y+x)\,\frac{|y_{n+1}|^a}{|y|}\,\d y.
\end{align*}
In particular, by direct computation
\begin{align}\label{e:derivata2}
\de_e I(x,t) & =
\frac{t}{H_u(x,t)^2}\,\big(
H_u(x,t) \, \de_e D_u(x,t) - D_u(x,t)\, \de_e H_u(x,t) \big)\notag\\ 
& = \frac{2}{H_u(x,t)}\int -\phi'\big(\textstyle{\frac{|y|}{t}}\big)  \, 
\Delta\mathcal{E}(y+x)\;
\Big(\nabla u (y+x) \cdot y -
I_u(x,t)\, u(y+x)\Big)\,\frac{|y_{n+1}|^a}{|y|}\,\d y.
\end{align}

\medskip

\noindent{\bf 2.} We use now \eqref{e:derivata2} with $t = R$ and $x \in B_1'$. 
Note that, since $x\in B_1'$, by \eqref{e:C1_1/2 
reg} -- \eqref{e: stima reg cacona},
\eqref{e:L2 vs H} and \eqref{e:H limitato} we infer that
\begin{align*}
M&:=\sup_{y\in B_R}|\nabla u (y+x) \cdot y - I_u(x,R)\, u(y+x)|\\
& \leq 
\sup_{z\in B_{R+1}}\big(|\nabla u (z) \cdot z|+|\nabla_\tau u(z)|\big)
+ I_u(x,R)\|u\|_{C^0(B_{R+1})}\\
&\leq C\,R^{-\frac{n+a}{2}}\,H_u^{\sfrac12}\big(0,
2R +2\big), 
\end{align*}
for some constant $C=C(A)>0$. Hence, we have that 
\begin{align}\label{e:derivata2.5}
\de_e I_u(x,R) &\leq 
\frac{2M}{H_u(x,R)}
\int -\phi'\big(\textstyle{\frac{|y|}{R}}\big)
\big(|\mathcal{E}_1(y+x)| + |\mathcal{E}_2(y+x)|\big)
\frac{|y_{n+1}|^a}{|y|}\,\d y.
\end{align}
In order to estimate the integral term in \eqref{e:derivata2.5}, 
we notice that 
\[
B_{R}(x)\setminus B_{\sfrac{R}{2}}(x)
\subset B_{R+2}(x) \setminus B_{\sfrac{R}{2}-2}(x_i)
\quad \forall \,x\in B_1', \;\text{for }\;i=1,2;
\]
therefore
\begin{align}\label{e:conto integrale}
\int_{B_{R}\setminus B_{\sfrac{R}{2}}}
|\mathcal{E}_i(y+x)|\,\frac{|y_{n+1}|^a}{|y|}\,\d y
&\leq \frac{2\,(R+2)}{R}\int_{B_{R+2}(x_i) \setminus 
B_{\sfrac{R}{2}-2}(x_i)}|\mathcal{E}_i(z)| \,\frac{|z_{n+1}|^a}{|z-x_i|}\,\d 
z\notag\\
&\leq C\, R^{\frac{n+a}{2}}\left( 
\int_{B_{R+2}(x_i) \setminus B_{\sfrac{R}{2}-2}(x_i)}
\mathcal{E}_i^2(z) \,
\frac{|z_{n+1}|^a}{|z-x_i|}\,\d z,
\right)^{\sfrac12},
\end{align}
where we used $R>6$ and a direct computation to estimate
\[
\int_{B_{R+2}(x_i) \setminus B_{\sfrac{R}{2}-2}(x_i)}
\frac{|z_{n+1}|^a}{|z-x_i|}\,\d z \leq C\,R^{{n+a}},
\]
for a dimensional constant $C>0$.
We are in the position to apply Lemma~\ref{l:monotonia}:
\begin{align}
\int_{B_{R+2}(x_i) \setminus B_{\sfrac{R}{2}-2}(x_i)}
\mathcal{E}_i^2(z) \,
\frac{|z_{n+1}|^a}{|z-x_i|}\,\d z 
\leq C_{\ref{l:monotonia}}(A)\,H_u\big(x_i,2R+4\big)\,
\Delta^{2(R+2)}_{\sfrac{R}{2}-2}(x_i). 
\label{e:stima secondo fattore}
\end{align}
Using \eqref{e:derivata2.5} -- \eqref{e:stima secondo fattore},
we get
\begin{align}\label{e:finale con brio}
\de_e I_u(x,R)
&\leq C\,\Big(\Delta^{2(R+2)}_{\sfrac{R}{2}-2}(x_1)\Big)^{\sfrac12}
+ C\,\Big(\Delta^{2(R+2)}_{\sfrac{R}{2}-2}(x_2)\Big)^{\sfrac12},
\end{align}
having used \eqref{e:monotonia H} and \eqref{e:H limitato} 
to infer that
\[
H_u\big(0,2R+2\big)^{\sfrac12}\,H_u\big(x,R\big)^{-1}
H_u\big(x_i,2R+4\big)^{\sfrac12}
\leq C(A)\,\frac{H_u(0,2R+4)}{H_u(0,R)}\leq C(A).
\]
In this respect, recall that $x,x_i\in B_1'$ and $R>6$,
so that we are in the position to apply Lemma~\ref{l:lim uniforme}.
The conclusion now follows by integrating
\eqref{e:finale con brio} along the segment $\{x_1+ r\,e:\, r\in[0,1]\}$.
\end{proof}

%
%
\section{Mean-flatness and frequency function}

\subsection{Mean-flatness}\label{s:mean-flatness}
We are going to use the following generalization of the Jones' $\beta$-numbers
introduced in \cite{Jones90}, also called \textit{mean-flatness}, 
which have been already extensively used in the literature
(cf., for example, \cite{AmFuPa,AzTo15,DMSV17,NaVa1, NaVa2}
and the list of references therein).
We adopt the standard notation $\dist(y, E) := \inf_{x \in E} |y-x|$
for the distance of a point $y$ from a given subset $E\subset\R^{n+1}$.

\begin{definition}\label{d:mean-flat}
Given a Radon measure $\mu$ in $\R^{n+1}$, $ p \in [1,+\infty)$ and $k \in \{0, 1, \ldots, n\}$,
for every $x_0 \in \R^n$ and for every $r>0$, we set
\begin{equation}\label{e:beta}
\beta^{(k)}_{\mu,p} (x,r) := \inf_{\cL} \Big(
r^{-k-p} \int_{B_r(x)} \dist(y,\cL)^p\d\mu(y)\Big)^{\frac{1}{p}},
\end{equation}
where the infimum is taken among all affine $k$-dimensional planes $\cL \subset \R^{n+1}$.
\end{definition}

In the case $p=2$ we have the following elementary characterization.
Let $x_0 \in \R^{n+1}$ and $r>0$ be such that $\mu(B_{r}(x_0)) >0$, and
let us denote by $\bar x_{x_0,r}$ the barycenter of $\mu$ in $B_r(x_0)$, {\ie}
\[
\bar x_{x_0,r} := \frac{1}{\mu(B_{r}(x_0))} \int_{B_{r}(x_0)} x \, \d\mu(x).
\]
Let moreover ${\bf B}_{x_0}:\R^{n+1} \times \R^{n+1} \to \R$ be the symmetric 
positive 
semi-definite bilinear
form given by
\[
{\bf B}_{x_0}(v, w) := \int_{B_{r}(x_0)} \big((x-\bar x_{x_0,r}) \cdot v\big)\;\big( 
(x-\bar x_{x_0,r}) \cdot w\big)\,\d\mu(x) 
\quad\forall \; v,\,w \in \R^{n+1}.
\]
By standard linear algebra there exists
an orthonormal basis of vectors in $\R^{n+1}$ which diagonalizes the bilinear 
form ${\bf B}_{x_0}$:
namely, there exist $\{v_1,\ldots,v_{n+1}\} \subset \R^{n+1}$ (in general
not uniquely determined) such that
\begin{itemize}
\item[(i)] $\{v_1,\ldots,v_{n+1}\}$ is a Euclidean orthonormal basis,
{\ie}~$v_i \cdot v_j = \delta_{ij}$;
\item[(ii)] ${\bf B}_{x_0}(v_i, v_i) = \lambda_i$, for some $0\leq \lambda_{n+1} \leq 
\lambda_{n} \leq \cdots \leq \lambda_1$;
\item[(iii)] ${\bf B}_{x_0}(v_i, v_j) = 0$ for $i\neq j$.
\end{itemize}
The characterization is then the following: the infimum in 
the definition of $\beta^{(k)}_{\mu,2}$ is reached
by all the affine planes $\cL = \bar x_{x_0,r} + \textup{Span}\{v_1,\ldots, 
v_k\}$ and
\begin{equation}\label{e:identita chiave}
\int_{B_{r}(x_0)} \big((x-\bar x_{x_0,r}) \cdot v_i\big)\,x\,\d\mu(x) = \lambda_i\,v_i 
\quad \forall\; i=1, \ldots, n+1,
\end{equation}
\begin{equation}\label{e:beta-charac}
\beta^{(k)}_{\mu,2}(x_0, r) = \Big(r^{-k-2}\sum_{l=k+1}^{n+1}\lambda_{l}\Big)^{\frac12}.
\end{equation}

\medskip

In the ensuing sections we are going to consider only the case
$k= n-1$ and $p = 2$: in order to simplify the 
notation we will always write $\beta_\mu$ for $\beta^{(n-1)}_{\mu, 2}$.

\subsection{Control of the mean-flatness via the frequency}\label{s:mean-flatness vs freq}
The main link between Jones' $\beta$-numbers and the geometric properties
of the free boundary is given by the following observation:
the mean-flatness of an arbitrary measure $\mu$ supported on $\Gamma(u)$ 
is controlled by the integration with respect to $\mu$ of suitable radial 
oscillations of the frequency.
This follows closely the approach by Naber--Valtorta \cite[Theorem~7.1]{NaVa1}
for harmonic maps and minimal surfaces. 
Because of the intrinsic renormalization of the frequency function here we need to use
the estimate in Proposition~\ref{p:D_x frequency} as done in \cite{DMSV17}
for multiple-valued functions.

\begin{proposition}\label{p:mean-flatness vs freq}
For every $A>0$ and $R>5$, there exists a constant 
$C_{\ref{p:mean-flatness vs freq}} (A,R)>0$ with this property. 
Let $r>0$, $u \in H^1(B_{(4R+10) r}(x_0), |x_{n+1}|^a\cL^{n+1})$
be a solution to the thin obstacle problem \eqref{e:ob-pb 
local} in $B_{(4R+10)r}(x_0)$,
with $x_0 \in \Gamma(u)$ and $I(x_0,(4R+10)r) \leq A$, 
and let $\mu$ be a finite Borel measure with 
$\spt(\mu)\subseteq \Gamma(u)$; then 
\begin{equation}\label{e:mean-flatness vs freq}
\beta_{\mu}^2 (p,r) \leq 
\frac{C_{\ref{p:mean-flatness vs freq}}}{r^{n-1}}
\int_{B_{r}(p)} 
\Delta_{(R-5)\,\sfrac{r}{2}}^{(2R+4)\,r}(x)\,\d\mu(x)
\quad\forall\;p \in \Gamma(u) \cap B_{r}.
\end{equation}
\end{proposition}

\begin{proof}
\noindent{\bf 1.} 
We can assume without loss of generality that
$x_0=0$ and that $p \in \Gamma(u) \cap B_r$
is such that $\mu(B_r(p)) >0$ (otherwise, there is nothing to prove).
Let $\bar x = \bar x_{p,r}$ be the barycenter of $\mu$ in 
$B_{r}(p)$
and let $\{v_1, \ldots, v_{n+1}\}$ be any diagonalizing basis for the
bilinear form ${\bf B}_p$ introduced in \S~\ref{s:mean-flatness},
with corresponding eigenvalues $0\leq\lambda_{n+1} \leq \lambda_{n} 
\leq \cdots \leq\lambda_1$.
Note that, since by assumption $\spt(\mu) \subset \Gamma(u)
\subset\R^n\times\{0\}$,
we can assume without loss of generality that $v_{n+1} = e_{n+1}$,
$\lambda_{n+1} = 0$ and hence
$\beta_{\mu} (p,r) = (r^{-n-1}\lambda_{n})^{\sfrac12}$ by 
\eqref{e:beta-charac}. Therefore, without loss of generality
we may also assume that $\lambda_n>0$.

From \eqref{e:identita chiave} and the definition of barycenter
we deduce that, for every $i\in\{1, \ldots, n\}$, for every
$z \in B_{(2R+5)r}$
and for every constant $\alpha \in \R$, we have
\begin{equation*}
-\lambda_i\,v_i \cdot \nabla u(z) = 
\int_{B_{r}(p)} 
\big((x-\bar x_{}) \cdot v_i\big)\,\big((z-x) \cdot \nabla u(z) - 
\alpha\,u(z)\big)\d\mu(x).
\end{equation*}
For the rest of the proof we set
\[
\alpha:= \frac{1}{\mu(B_{r})} \int_{B_{r}(p)} I_u(x, 
(R-1)\,r)\d\mu(x).
\]
Using H\"older inequality we deduce that
\begin{align}
\lambda_i^2 |v_i \cdot \nabla u(z)|^2
& \leq 
\int_{B_{r}(p)} \big((x-\bar x_{}) \cdot v_i\big)^2\d\mu(x) \;\;
\int_{B_{r}(p)} \big((z-x) \cdot \nabla u(z) - \alpha\,u(z)\big)^2
\d\mu(x)\notag\\
& \hspace{-0.4cm}\stackrel{\S\ref{s:mean-flatness}\,\textup{(ii)}}{=} 
\lambda_i\,\int_{B_{r}(p)}
 \big((z-x) \cdot \nabla u(z) - \alpha\,u(z)\big)^2\d\mu(x).
\label{e:controllo autovalori}
\end{align}
Denoting with 
$\nabla_\tau u = (\de_1 u, \ldots, \de_n u)$
the tangential gradient, and recalling that 
\[
r^{n+1}\beta^2_{\mu} (p,r) = \lambda_{n}
\and
0=\lambda_{n+1}< \lambda_n\leq\ldots\leq\lambda_1,
\]
by integrating over $B_{(R+1)r}(p)\setminus B_{Rr}(p)$
the previous inequality with respect to the measure
$|z_{n+1}|^a\cL^{n+1}$
we get
\begin{align}\label{e:intermedio}
r^{n+1}\beta_{\mu}^2(p,r) &\int_{B_{(R+1)r}(p)
\setminus B_{Rr}(p)}
|\nabla_\tau u(z)|^2|z_{n+1}|^a\d z
= \lambda_{n} 
\int_{B_{(R+1)r}(p) \setminus B_{Rr}(p)}
|\nabla_\tau u(z)|^2|z_{n+1}|^a\d z \notag\\
& \leq \int_{B_{(R+1)r}(p) \setminus B_{Rr}(p)}
\sum_{i=1}^{n} \lambda_i|v_i \cdot \nabla u(z)|^2|z_{n+1}|^a\,\d z
\notag\\
&\stackrel{\eqref{e:controllo autovalori}}{\leq}
n
\int_{B_{(R+1)r}(p) \setminus B_{Rr}(p)}
\int_{B_{r}(p)} \big((z-x) \cdot \nabla u(z) - 
\alpha\,u(z)\big)^2\d\mu(x)\,|z_{n+1}|^a\d z
\notag\\
& \leq n\int_{B_{r}(p)} \int_{B_{(R+2)r}(x) \setminus 
B_{(R-1)r}(x)}
 \big((z-x) \cdot \nabla u(z) - \alpha\,u(z)\big)^2|z_{n+1}|^a\d z\;\d\mu(x).
\end{align}
Next we estimate the two sides of \eqref{e:intermedio}.

\medskip

\noindent{\bf 2.} For what concerns the
l.h.s. of \eqref{e:intermedio}, we claim the following: 
there exists a constant $C=C(A,R)>0$ such that
\begin{align}
D_u(p,(R+2)r)
\leq C\int_{B_{(R+1)r}(p)\setminus B_{Rr}(p)}
|\nabla_\tau u(z)|^2|z_{n+1}|^a\d z.
\label{e:stima dal basso}
\end{align}
We argue by contradiction. If the claim were not true,
after a suitable rescaling 
replacing $u$ with $u_{p,r}$, we could find a 
sequence of solutions
$u_k$ to the thin obstacle problem in $B_{2R+4}$
with $0\in \Gamma(u_k)$,
such that $I_{u_k}(R+3)\leq A':= A+C_{\ref{l:lim uniforme}}(A)$,
$H_{u_k}(R+2) = 1$
and
\[
\int_{B_{R+1}\setminus B_{R}}
|\nabla_\tau u_k(z)|^2|z_{n+1}|^a\d z
\leq \frac{D_{u_k}(R+2)}{k},
\]
(recall that $B_{(2R+4)r}(p)\subset B_{(4R+10)r}$
and by Lemma~\ref{l:lim uniforme}
we have $I_u(p,(R+3)r)\leq A'$).
By Corollary~\ref{c:H} we have that
$H_{u_k}(R+3) \leq \left(\sfrac{(R+3)}{(R+2)}\right)^{n+a+2A'}$
and hence by Corollary~\ref{c:compactness},
(up to subsequences, not relabeled) $u_k$ converge in 
$H^1_{}(B_{R+2},|x_{n+1}|^a\cL^{n+1})$
to a solution $u_0$ to the thin obstacle problem in $B_{R+2}$
with $H_{u_0}(R+2) = 1$ and
\[
\int_{B_{R+1}\setminus B_{R}}
|\nabla_\tau u_0(z)|^2|z_{n+1}|^a\d z = 0.
\]
We deduce from the latter equality that
$u_0$ depends only on the variable $x_{n+1}$
(recall that $u_0$ is analytic in $B_1^+$).
In particular, $u_0 (x) = - c\,|x_{n+1}|^{2s}$ for some  $c>0$,
and 
\[
I_{u_0}(t) = 2s < 1+s \leq I_{u_k}(t)\quad\forall
\;t \in (0, R+2) ,\;\forall\; k\in\N,
\]
where we used Lemma~\ref{l:freq lower}.
This contradicts 
$\lim_k I_{u_k}(t)=I_{u_0}(t)$ and concludes the
proof of \eqref{e:stima dal basso}.

\medskip

\noindent{\bf 3.} 
Now we estimate the r.h.s. of \eqref{e:intermedio} from above.
By the triangular inequality we have that
\begin{align*}
&\text{r.h.s. of \eqref{e:intermedio}}\leq\\
&\quad \leq 2n\int_{B_{r}(p)}\int_{B_{(R+2)r}(x) \setminus 
B_{(R-1)r}(x)}
\big((z-x) \cdot \nabla u(z) - 
I_u(x,(R-1)r)\,u(z)\big)^2|z_{n+1}|^a\,\d z\;\d\mu(x)\notag\\
& \qquad + 2n \int_{B_{r}(p)}
\int_{B_{(R+2)r}(x) \setminus 
B_{(R-1)r}(x)}
\Big(I_u(x,(R-1)r) - \alpha \Big)^2 \,u^2(z)|z_{n+1}|^a\,\d 
z\;\d\mu(x).\notag
\end{align*}
For every $x\in\spt(\mu) \cap B_r(p)$, 
\eqref{e:I limitato} in Lemma~\ref{l:lim uniforme} yields
\[
I_u(x,(R+2)r)\leq I_u(0,(R+2)r) + C_{\ref{l:lim 
uniforme}}(A) \leq A+  C_{\ref{l:lim 
uniforme}}(A) = A',
\]
since $B_{r}(p)\subseteq B_{2r}$
and $u$ is defined on $B_{(2R+4)r}(p)\subset 
B_{(4R+10)r}$.
By using Lemma~\ref{l:monotonia}, we can estimate the 
first integral above as follows:
\begin{align}
&\int_{B_{(R+2)r}(x) \setminus 
B_{(R-1)r}(x)}
\big((z-x) \cdot \nabla u(z) - 
I_u(x,(R-1)r)\,u(z)\big)^2|z_{n+1}|^a\,\d 
z\notag\\
&\quad\leq (R+2)\,r
\int_{B_{(R+2)r}(x) \setminus 
B_{(R-1)r}(x)}
\big((z-x) \cdot \nabla u(z) - 
I_u(x,(R-1)r)\,u(z)\big)^2
\,\frac{|z_{n+1}|^a}{|z-x|}\,\d z\notag\\
&\quad\stackrel{\eqref{e:monotonia con resto}}{\leq }
C_{\ref{l:monotonia}}(A')\,(R+2)\,r 
\,H\big(x,(2R+4)r\big)\,
\Delta^{(2R+4)r}_{(R-1)r}(x)\notag\\
&\quad\leq 
C(A)\,(R+2)\,r \,H\big(p,(2R+4)r\big)\,
\Delta^{(2R+4)r}_{(R-1)r}(x),
\label{e:stima dall'alto1}
\end{align}
in the last inequality we have used Lemma~\ref{l:lim uniforme}
(because $|p-x|<r$ and $u$ is defined in $B_{(4R+8)r}(p)\subset 
B_{(4R+10)r}$).
On the other hand, using Jensen's inequality 
and Proposition~\ref{p:D_x frequency} (recall that 
$\spt(\mu)\subseteq \Gamma(u)$), we deduce that
\begin{align}
\int_{B_r(p)}&\Big(I_u(x,(R-1)r) - \alpha \Big)^2\,\d \mu(x) 
\notag\\
& \leq
\frac{1}{\mu(B_{r}(p))} \int_{B_{r}(p)} \int_{B_{r}(p)}
\Big( I_u\big(x,(R-1)r\big) -  I_u\big(y, 
(R-1)r\big)\Big)^2 \d \mu(y)\;\d\mu(x)\notag\\
&\stackrel{\eqref{e:D_x frequency}}{\leq}
\frac{2\,C^2_{\ref{p:D_x frequency}}(A')}{\mu(B_{r}(p))}
\int_{B_{r}(p)} \int_{B_{r}(p)}
\Big(\Delta^{2(R+1)r}_{(R-5)\sfrac{r}{2}}(x) + 
\Delta^{2(R+1)r}_{(R-5)\sfrac{r}{2}}(y)\Big)\,
\d \mu(y)\;\d\mu(x)\notag\\
& \leq C 
\int_{B_{r}(p)}\Delta^{2(R+1)r}_{(R-5)\sfrac{r}{2}}(x)\, 
\d \mu(x).
\label{e:stima dall'alto2}
\end{align}
Finally, note that
\begin{align}
\int_{B_{(R+2)r}(x) \setminus 
B_{(R-1)r}(x)} u^2(z)|z_{n+1}|^a\,\d z &
\stackrel{\eqref{e:L2 vs H}}{\leq}
(R+2)\,r\,H_u\big(x, (R+2)r \big)\notag\\
&\stackrel{\eqref{e:H limitato}}{\leq} 
C\,(R+2)\,r\,H_u\big(p,(R+2)r \big),
\label{e:stima dall'alto3}
\end{align}
where once again in the last inequality we 
have used Lemma~\ref{l:lim uniforme}.
\medskip

\noindent{\bf 4.}
We can now collect the estimates \eqref{e:stima dal basso} --
\eqref{e:stima dall'alto3} to get
\begin{align*}
r^{n+1}\beta_{\mu}^2(p,r) D_u(p,(R+2)r)
&\leq C(A,R)\,r \,H_u\big(p, (2R+4)r\big)\,
\int_{B_{r}(p)}
\Delta^{2(R+1)r}_{(R-5)\sfrac{r}{2}}(x)\, \d \mu(x)\\
&\stackrel{\eqref{e:monotonia H}}{\leq} C(A,R)\,r \,H_u\big(p, 
(R+2)r\big)\,
\int_{B_{r}(p)}
\Delta^{2(R+1)r}_{(R-5)\sfrac{r}{2}}(x)\, \d \mu(x).
\end{align*}
From $I_u(p,(R+2)r) \geq 1+s$ (cf. Corollary~\ref{c:minima freq}),
one can then infer \eqref{e:mean-flatness vs freq}.
\end{proof}

%
%
\section{Homogeneous and almost homogeneous solutions}\label{s:rigidity}
For the proof of the main theorems, we need to discuss 
some results concerning homogeneous and
almost homogeneous solutions
to the thin obstacle problem \eqref{e:ob-pb local}.

\subsection{Spines of homogeneous solutions}
We denote by $\cH_\lambda$ the space of all (non-trivial) 
$\lambda$-homogeneous solutions to the thin obstacle problem 
\eqref{e:ob-pb local},
\[
\cH_{\lambda} := \Big\{ u\in H^1_\loc\big(\R^{n+1}, 
|x_{n+1}|^a\,\cL^{n+1}\big)\setminus\{0\}
:\, u(x) = |x|^\lambda\,u\big(\sfrac{x}{|x|}\big),\text{ 
$u\vert_{B_1}$ solves \eqref{e:ob-pb local}}\Big\},
\]
and set
$\cH:= \bigcup_{\lambda\geq 1+s}\cH_{\lambda}$.
The restriction $\lambda\geq 1+s$ is imposed in view of 
Corollary~\ref{c:minima freq}.
Recall next the definition of spine of homogeneous solutions (see, 
\textit{e.g.}, \cite{FMS-15}).

\begin{definition}\label{d:spine}
Let $u \in \cH$ be a homogeneous solution. The \textit{spine} $S(u)$
is the maximal subspace of invariance of $u$:
\[
S(u) := \Big\{ y\in\R^n\times \{0\}\;:\; u(x+y) = u(x) \quad \forall\; x\in 
\R^{n+1}\Big\}.
\]
\end{definition}

Simple characterizations of the spine are provided
in the next result.

\begin{lemma}\label{l:spine}
Let $u \in \cH$ be given. The following are equivalent conditions:
\begin{itemize}
\item[(i)] $x \in S(u)$,
\item[(ii)] $u$ is homogeneous with respect to $x$, 
{\ie}~$I_u(x,r) = I_u(x, 0^+)$ for all $r>0$,
\item[(iii)] $I_u(x,  0^+) = I_u(0,  0^+)$.
\end{itemize}
\end{lemma}

\begin{proof}
The very definition of spine yields straightforwardly that (i) implies (ii) and 
(iii). To see that (ii) implies (iii),
we consider the functions $u_{0,r_k}$ 
as defined in \eqref{e:rescaling0},
for a sequence of radii $r_k\uparrow +\infty$
such that $u_{0,r_k}$ converge to some $u_{\infty}$
in $H^1_{\loc}(\R^{n+1},|x_{n+1}|^a\cL^{n+1})$.
Then, by Remark~\ref{r:scaling}
we infer that
\begin{align*}
\left\vert I_u(x,  0^+) - I_u(0,0^+)\right\vert 
&\stackrel{\textup{(ii)}}{=}
\lim_{k\to+\infty}\left\vert I_u(x, r_k) - I_u(0,r_k)\right\vert
=\lim_{k\to+\infty}\left\vert I_{u_{0,r_k}}(\sfrac{x}{r_k}, 1) - 
I_{u_{0,r_k}}(0,1)\right\vert\\
&= \left\vert I_{u_{\infty}}(0,1) - 
I_{u_{\infty}}(0,1)\right\vert = 0.
\end{align*}
Similarly, (iii) implies (ii): let $r_k\uparrow+\infty$ be
a sequence as above, then using $u\in\cH$ we get
\begin{align*}
\lim_{k\to+\infty}\left\vert I_u(x,  r_k) - I_u(x,  0^+)\right\vert
 &\leq \lim_{k\to+\infty}
\left\vert I_u(x,  r_k) - I_u(0,r_k)\right\vert
+\left\vert I_u(x, 0^+) - I_u(0,0^+)\right\vert\\
&\stackrel{\textup{(iii)}}{=}
\lim_{k\to+\infty}
\left\vert I_u(x,  r_k) - I_u(0,r_k)\right\vert =0.
\end{align*}
In particular, taking into account the monotonicity of the frequency,
we infer that $I_u(x,r)=I_u(x,0^+)$ for every $r>0$, {\ie}~(ii).
Finally, we are left to show that (ii) and (iii) imply (i). 
By (ii) and (iii) we have that 
 \[
u(y + x) = t^\lambda\, u\left(\frac{y}{t}+x\right)\quad\forall\; 
y\in\R^n,\;\forall\,t>0,
\]
with $\lambda = I_u(0, 0^+)$. Hence, for every $y\in\R^n\times\{0\}$ 
we have
\begin{align*}
u(y) = u(x + y-x) = 2^\lambda\, u \big( x +\sfrac{(y-x)}2\big) 
= 2^\lambda \,u \big( \sfrac{(y+x)}2\big) = u(y+x).
&\qedhere
\end{align*}
\end{proof}

\subsection{Classification of solutions with maximal dimension of the 
spine}\label{s:classification}
There are no homogeneous functions $u \in \cH$ with spine of dimension $n$, 
because the only non-trivial solutions of \eqref{e:ob-pb local} even with respect to 
$x_{n+1}$ and depending only on the variable $x_{n+1}$
are of the form $c\,|x_{n+1}|^{2s}$ with $c < 0$. The 
latter functions 
have homogeneity $2s < 1+s$, that is not allowed for functions in $\cH$.
Therefore, the maximal dimension of the spine of a function in $\cH$ is at most
$n-1$. 
We say that $u\in\cH^\top$ if $u \in \cH$ and $\dim S(u) = n-1$,
and we set $\cH^\low:=\cH\setminus \cH^\top$ otherwise.
All functions in $\cH^\top$ are classified in the following list.

\begin{lemma}\label{l:classification}
$u\in \cH^\top$ if and only if there
exists a uniquely determined $\lambda$-homogeneous function
$h_\lambda:\R^2\to \R$, with 
$\lambda \in\{2m, 2m-1+s, 2m +2s\}_{m\in\N\setminus\{0\}}$, 
such that
\[
u(x) = c\,h_{\lambda}(x\cdot e, x_{n+1})
\and 
H_{h_\lambda(x_1,x_{n+1})}(0,1)=1,
\]
for some $c>0$ and $e\in \R^{n}\times \{0\}$ with $|e|=1$.
In particular, if $u\in \cH^\top$ then $\mathcal{N}(u) =S(u)$, 
and more precisely: if $u(x) = c\,h_{\lambda}(x\cdot e, x_{n+1})$, then
\begin{itemize}
\item[(I)] if $\lambda=2m$: $\Lambda(u) = \Gamma(u) = \mathcal{N}(u) =
S(u) = \big\{x\cdot e = 
x_{n+1}=0\big\}$;
\item[(II)] if $\lambda = 2m-1+s$: $\Lambda(u) =\big\{x\cdot e \leq 0,\, 
x_{n+1}=0\big\}$ and
$\Gamma(u) = \mathcal{N}(u) = S(u) = \big\{x\cdot e = x_{n+1}=0\big\}$;
\item[(III)] if $\lambda = 2m+2s$: $\Lambda(u) =\big\{x_{n+1}=0\big\}$,
$\Gamma(u) = \emptyset$ and $\mathcal{N}(u) =S(u) = \big\{x\cdot e = 
x_{n+1}=0\big\}$.
\end{itemize}
\end{lemma}

The proof of Lemma~\ref{l:classification} is a consequence
of the full characterization of the homogeneous
solutions $v:\R^2\to\R$ to the thin obstacle problem.
Introducing polar co-ordinates $v(\rho\cos\theta, \rho\sin\theta) = 
\rho^\lambda\,y(\theta)$ with $y:[0,\pi]\to\R$,
the system \eqref{e:ob-pb local} can be written in the form:
\begin{equation}\label{e:polar coord}
y''(\theta) + a\ctg\theta\,y'(\theta) 
+\lambda\,(\lambda +a)\,y(\theta) = 0,
\end{equation}
with the following four possible boundary conditions:
\begin{gather}
y(0)>0,\quad
y(\pi)>0
 \quad\text{and} \quad
\lim_{\theta\downarrow0^+} (\sin\theta)^{1-2s}y'(\theta)=
\lim_{\theta\downarrow\pi^-} (\sin\theta)^{1-2s}y'(\theta) =0,
\label{e:polar bd I}\\
y(0)=0 <y(\pi),\quad
\lim_{\theta\downarrow0} (\sin\theta)^{1-2s}y'(\theta) \leq 0
\quad\text{and} \quad 
\lim_{\theta\uparrow\pi} (\sin\theta)^{1-2s}y'(\theta)=0,
\label{e:polar bd II}\\
y(\pi)=0 <y(0),\quad
\lim_{\theta\downarrow0} (\sin\theta)^{1-2s}y'(\theta) = 0
\quad\text{and} \quad 
\lim_{\theta\uparrow\pi} (\sin\theta)^{1-2s}y'(\theta)\geq 0,
\label{e:polar bd IIbis}\\
y(0)=y(\pi)=0,\quad \lim_{\theta\downarrow0}
 (\sin\theta)^{1-2s}y'(\theta) \leq 0
\quad\text{and} \quad
\lim_{\theta\uparrow\pi}(\sin\theta)^{1-2s}y'(\theta) \geq 0.
\label{e:polar bd III}
\end{gather}
The four cases \eqref{e:polar bd I} -- \eqref{e:polar bd III} 
determine
the corresponding exponents $\lambda$ and
the solutions $h_{\lambda}$ as in (I), (II), (III) 
of the lemma.
The proof in general requires the use of the associated Legendre functions and 
is 
postponed to the appendix where we establish
also other complementary results 
that are mandatory
for the analysis in Section~\ref{s:blow-up} (cf. 
Proposition~\ref{p:classiFICAzione 2d}).
Here we give the details for the simplest case of the Signorini problem
$s=\sfrac12$, {\ie}~$a=0$.

\begin{proof}[Proof of Lemma \ref{l:classification} for $s=\sfrac12$.]
If $a=0$, the general integral of \eqref{e:polar coord} is 
\[y(\theta) = A_1\,\cos(\lambda\theta) + A_2\,\sin(\lambda\theta),
\]
with $A_1,A_2\in\R$. 
We can then discuss the four possible cases.
\begin{itemize}
\item[(I)] For \eqref{e:polar bd I} we have that $y'(0)=0$ implies $A_2=0$ and 
$A_1\neq 0$,
and $y'(\pi)=0$ gives $\lambda\in \N$. Considering that
$y(0)>0$ we find $A_1>0$, and by $y(\pi)>0$ one gets $\lambda = 2m$ with $m 
\in\N\setminus\{0\}$.

\item[(II)]
For \eqref{e:polar bd II}, we have that 
$y(0)=0$ gives $A_1=0$ and $A_2\neq 0$.
In turn $y'(0)\leq 0$ implies $A_2<0$.
Thus, $y'(\pi)=0$ yields $\cos(\lambda\pi) = 0$, {\ie}~$\lambda=m+\sfrac12$
for $m \in\N\setminus\{0\}$, and finally $m$ is odd since $y(\pi)>0$. 
One proceeds analogously in case \eqref{e:polar bd IIbis}.

\item[(III)] Finally, for \eqref{e:polar bd III}, we have that
$y(0)=y(\pi)=0$ implies $A_1=0$, $A_2\neq 0$ and $\lambda \in\N$.
Considering that $y'(0)\leq 0 \leq y'(\pi)$ we conclude that $A_2<0$ and  
$\lambda$ odd.
\end{itemize}
In all the cases the nonzero coefficient $A_i$
is chosen suitably in order 
to satisfy the normalization condition $H_{h_\lambda}(1)=1$.
The statements concerning $\Gamma(u)$, $\Lambda(u)$, 
$\mathcal{N}(u)$ and $S(u)$
are now direct consequences of the explicit 
formulas for the solutions.
\end{proof}

For the lowest frequency $1+s$, actually all
homogeneous solutions have maximal spine
as proved by Caffarelli, Salsa and Silvestre in~\cite{CaSaSi08}, 
this result can be equivalently stated by the inclusion
\begin{equation}\label{e:H1+s}
 \cH_{1+s}\subseteq\cH^{top}.
\end{equation}

\subsection{Almost homogeneous solutions}

We next introduce the notion of almost homogeneous solutions.

\begin{definition}\label{d:almost hom}
Let $\eta>0$ and $R>1$ be given constants.
A solution $u:B_R\to\R$ to thin obstacle problem
\eqref{e:ob-pb local} is called
\textit{$\eta$-almost homogeneous} if $0\in \Gamma(u)$ 
and
\[
I_u(1) - I_u\big(\sfrac12\big)\leq \eta.
\]
\end{definition}

The following lemma justifies this terminology.

\begin{lemma}\label{l:almost hom}
For every $\delta,A>0$ and $R>1$ there
exists $\eta>0$ with the following property:
let $u$ be a $\eta$-almost homogeneous solution
with $I_u(R)\leq A$ and $H_u(R)=1$; then, there exists
a homogeneous solution $w\in \cH$ such that
\begin{align}\label{e:almost hom}
\big\| u - w\big\|_{H^1(B_{R-1}, |x_{n+1}|^a\cL^{n+1})}
\leq \delta.
\end{align}
\end{lemma}

\begin{proof}
We argue by contradiction: assume there exist
$\delta, A,R$ as in the statement
and a sequence $(u_k)_{k\in\N}$ of $k^{-1}$-homogeneous
solutions in $B_R$ with $I_{u_k}(R)\leq A$
such that
\begin{align}\label{e:contra}
H_{u_k}(R) = 1
\and
\inf_{w\in \cH}\big\| u_{k} - w\big\|_{H^1(B_{R-1}, 
|x_{n+1}|^a\,\cL^{n+1})} > 
\delta>0.
\end{align}
We can then apply Corollary~\ref{c:compactness}
and find a subsequence (not relabeled) and a solution $u_0$
to the obstacle problem in $B_R$
such that $u_k \to u_0$ in $H^1(B_{R-1}, |x_{n+1}|^a\,\cL^{n+1})$.
Note that
\[
H_{u_0}(R-1)=\lim_{k\to+\infty}H_{u_k}(R-1)
\stackrel{\eqref{e:H1}}{\geq}\Big(\frac{R-1}{R}\Big)^{n+a+2A}
\lim_{k\to+\infty}H_{u_k}(R) = \Big(\frac{R-1}{R}\Big)^{n+a+2A}.
\]

In particular, $u_0$ is non-trivial and
\[
I_{u_0}(1) - I_{u_0}\big(\sfrac12\big) = 
\lim_{k\to+\infty} \Big(I_{u_k}(1) - I_{u_k}\big(\sfrac12\big)\Big)
= 0.
\]
By Proposition~\ref{p:monotonia+lower} we infer that $u_0$
is homogeneous of degree 
$I_{u_0}(1) = \lim_{k\uparrow+\infty}I_{u_k}(1)\geq 1+s$,
because $0\in\Gamma(u_k)$.
Therefore, $u_0\in \cH$ and this contradicts \eqref{e:contra}.
\end{proof}

Next we show a rigidity result which will be used crucially in the estimate
of the Hausdorff measure of the free boundary.

\begin{proposition}\label{p:rigidity}
For every $\tau,A>0$ 
there exists $\eta_{\ref{p:rigidity}}(\tau, A)>0$ with this property.
Let $u:B_4\to\R$ be a $\eta_{\ref{p:rigidity}}$-almost homogeneous
solution to the thin obstacle problem 
with $I_u(0,4)\leq A$.
Then, the following dichotomy holds:
\begin{itemize}
\item[(i)] either for every point $x\in \Gamma(u)\cap B_2$ we have 
\begin{align}\label{e:rigidity1}
\left\vert I_u(x,1) - I_u(0,1)\right\vert\leq\tau,
\end{align}
\item[(ii)] or there exists a linear subspace $V\subset\R^{n}\times\{0\}$
of dimension $n-2$ such that
\begin{align}\label{e:rigidity2}
y\in \Gamma(u)\cap B_{2},\; 
I_u(y,1) - I_u\big(y,\sfrac12\big)\leq \eta_{\ref{p:rigidity}}
\quad\Longrightarrow\quad
y\in \cT_\tau(V),
\end{align}
recall the notation $\cT_\tau(V) := \{x\in\R^{n+1}: 
\dist(x,V)<\tau\}$.
\end{itemize}
\end{proposition}

\begin{proof}
The proof is by contradiction.
We assume that there exist 
$\tau, A$ as in the statement and a sequence
$(u_k)_{k\in\N}$ of $k^{-1}$-almost homogeneous
solutions in $B_4$ with $I_{u_k}(4)\leq A$
contradicting the statement, {\ie}~both 
(i) and (ii) simultaneously fail: 
namely, there exists $x_k\in \Gamma(u_k)\cap B_{2}$ such that 
\begin{align}\label{e:rigidity1-contra}
\left\vert I_{u_k}(x_k,1) - I_{u_k}(0,1)\right\vert >\tau,
\end{align}
and for every linear subspace $V\in\R^n\times\{0\}$
of dimension $n-2$
there exists $y_k\in \Gamma(u_k)\cap B_{2}$ (a priori depending on $V$)
such that
\begin{align}\label{e:rigidity2-contra}
I_{u_k}(y_k,1) - I_{u_k}\big(y_k,\sfrac12\big)\leq k^{-1}
\quad\text{and}\quad y_k\not\in \cT_\tau(V).
\end{align}
By eventually rescaling
the functions of the sequence, we can assume without
loss of generality that $H_{u_k}(0,4)=1$.
In particular, it follows
from Lemma~\ref{l:almost hom} that 
\begin{align}\label{e:distanza omog}
\lim_{k\to+\infty} \inf_{w\in \cH}
\|u_k-w\|_{H^1(B_{3}, |x_{n+1}|^a\,\cL^{n+1})}=0.
\end{align}
Up to passing to a subsequence (not relabeled)
we distinguish to cases:
\begin{itemize}
\item[(a)] either there exists $w_k\in\cH^\top$ such that 
$\lim_k\|u_k-w_k\|_{H^1(B_{3}, |x_{n+1}|^a\,\cL^{n+1})}= 0$;
\item[(b)] or there exists $\delta>0$ such that
\begin{align}\label{e:distanza top}
\inf_{w\in \cH^\top}
\|u_k-w\|_{H^1(B_{3}, |x_{n+1}|^a\,\cL^{n+1})}\geq 
\delta>0\quad\forall\;k\in\N.
\end{align}
\end{itemize}
In case (a) we show that \eqref{e:rigidity1-contra}
cannot hold. Indeed, by
Corollary~\ref{c:compactness} there exist
a homogeneous function $w_0\in 
\cH^\top$ (note that
$\cH^\top$ is closed under locally strong $H^1$
convergence), a point $x_0\in \bar B_2$
and a subsequence (not relabeled) such that
\begin{gather*}
\lim_{k\to+\infty}\Big(\|u_k-w_0\|_{H^1(B_{3}, |x_{n+1}|^a\,\cL^{n+1})}
+ \|w_k-w_0\|_{H^1(B_{3}, |x_{n+1}|^a\,\cL^{n+1})}\Big)=0\\
x_k\to x_0 \in \mathcal{N}(w_0) \cap \bar B_{2}.
\end{gather*}
In particular,
\[
\left\vert I_{w_0}(x_0,1) - I_{w_0}(0,1)\right\vert =
\lim_{k\to+\infty}
\left\vert I_{u_k}(x_k,1) - I_{u_k}(0,1)\right\vert \geq \tau.
\]
Note that $w_0\not\equiv 0$, because 
$H_{w_0}(0,3) = \lim_k H_{u_k}(0,3) \geq (\sfrac34)^{n+a+2A}$ 
thanks to \eqref{e:H1} in Corollary~\ref{c:H}.
This implies that $I_{w_0}(x_0,1) = I_{w_0}(0,1)$, since
$x_0\in\mathcal{N}(w_0)= S(w_0)$ being $w_0\in \cH^{\top}$ 
and Lemma~\ref{l:classification},
which gives the desired contradiction.

In case (b), by combining \eqref{e:distanza omog} and
\eqref{e:distanza top},
and by the compactness in Corollary~\ref{c:compactness}
(up to passing to a subsequence, not relabeled)
there exists $v_0\in \cH$ such that
\[
\lim_{k\to+\infty}\|u_k-v_0\|_{H^1(B_{3}, |x_{n+1}|^a\,\cL^{n+1})}= 0.
\]
Moreover, from \eqref{e:distanza top} we deduce that 
$v_0\in\cH^\low$ (note that $v_0\not\equiv 0$ 
because we have that 
$H_{v_0}(0,3) = \lim_k H_{u_k}(0,3) \geq 
(\sfrac34)^{n+a+2A}$
by Corollary~\ref{c:H}).
We now show that we have a contradiction to \eqref{e:rigidity2-contra}
with $V$ any $(n-2)$-dimensional subspace containing $S(v_0)$.
Indeed, let $y_k$ be as in \eqref{e:rigidity2-contra} for such a choice 
of $V$; by compactness, up to passing to a subsequence (not relabeled), 
there exists $y_0\in \bar B_{2}$ such that
\begin{align}\label{e:rigidity2-contra 2}
I_{v_0}(y_0,1) - I_{v_0}\big(y_0,\sfrac12\big)=0
\quad\text{and}\quad y_0\not\in \cT_\tau(V).
\end{align}
Proposition~\ref{p:monotonia+lower} implies
that $I_{v_0}(y_0,r)
= I_{v_0}(y_0,0^+)$ for every $r>0$ and 
by Lemma~\ref{l:spine} we must have
$y_0 \in S(v_0)$, thus contradicting $S(v_0)\subseteq V$ and 
$y_0\not\in \cT_\tau(V)$.
\end{proof}

%
%
\section{The measure of the free boundary}\label{s:misura}
In this section we prove Theorem~\ref{t:misura}
that provides a local estimate of the Minkowski content, and thus of 
the Hausdorff measure, of the free boundary in the lower dimensional 
obstacle problem.
Here we use a modified version of the inductive covering 
argument in \cite[Section 8]{NaVa1}.
The key monotone quantity we consider is the maximal
function of the frequency
\begin{equation}\label{e:theta}
\Theta_u (x,\rho) := \sup_{y \in \bar{B}_{\rho}(x)
\cap \Gamma(u)} I_u(y, \rho).
\end{equation}
Theorem~\ref{t:misura} is a direct consequence of the
following proposition.

\begin{proposition}\label{p:induzione}
For every $L>0$, there exists
a constant $C_{\ref{p:induzione}}(L)>0$ with
this property:
for any solution $u\not\equiv 0$ to the
thin obstacle problem \eqref{e:ob-pb local} in $B_{2\rho}(z) \subset\R^{n+1}$
with $z\in \R^n\times\{0\}$, we have
\begin{equation}\label{e:IS}
\Theta_{u}(z,\rho)
\leq L
\quad \Longrightarrow \quad
\cL^{n+1}\Big(\cT_{r}\big( \Gamma(u) \cap B_{\sfrac{\rho}2}(z)\big)\Big)
\leq C_{\ref{p:induzione}}(L) \, r^2\,\rho^{n-1}\quad
\forall\;r \in (0,\rho).
\end{equation}
\end{proposition}

\medskip

\begin{proof}[Proof of Theorem~\ref{t:misura}]
We are given $u$ a solution to the 
lower dimensional obstacle problem
in $B_1$ and $K\subset\subset B_1$.
Set $\delta :=  4^{-1}\dist(K,\de B_1)$, 
let $\{B_\delta(x_i)\}_{i\in J}$, with
$J$ a (finite) maximal subset of points in $\Gamma(u) \cap K$, 
having mutual distance at least $\delta$.
Set $L := \max_{i\in J} \Theta_u(x_i,2\delta)$.
Then, by applying Proposition~\ref{p:induzione}
to every $B_{2\delta}(x_i)$,
we have that
\begin{align*}
\cL^{n+1}\Big(\cT_{r}\big( \Gamma(u) \cap K\big)\Big)
& \leq \sum_{i\in J} \cL^{n+1}\Big(\cT_{r}\big( \Gamma(u) \cap 
B_{\delta}(x_i)\big)\Big)\notag\\
&\leq \cH^0(J)\,C_{\ref{p:induzione}}(L) \,\delta^{n-1}\, r^2
=: C\,r^2
\qquad\forall\,r\in (0,2\delta).
\end{align*}
We point out that the constant $C$ depends only $n$, on $I_u(1)$ and
on $\dist(K,\de B_1)$. Indeed, $L$ depends on $I_u(1)$  via 
Lemma~\ref{l:lim uniforme};
and since the balls $B_{\sfrac\delta 2}(x_i)$ are disjoint, contained 
in $B_1$ and with centers in $B_1'$, we can estimate
$\cH^0(J)\leq (\sfrac{2}{\delta})^n$.
\end{proof}

The rest of the section is devoted to the proof of 
Proposition~\ref{p:induzione}.

\subsection{Proof of Proposition~\ref{p:induzione}}\label{ss:induction proof}
By rescaling it is enough to consider the case $z=0$ and $\rho=1$.
We start off with the case of minimal frequency $L=1+s$:
then, by Corollary~\ref{c:minima freq}
$\Theta_u(0,1) = 1+s$ and thus $u\in \cH_{1+s}$
(cf.~\eqref{e:H1+s}).
In turn, this implies that $\Gamma(u)$ is a $(n-1)$-dimensional 
hyperplane of $\{x_{n+1}=0\}$ and \eqref{e:IS} follows at once.

The proof is then completed by showing that $\sup \cS=+\infty$ where 
\[
\cS:= \{L\in \R: \;\text{Proposition~\ref{p:induzione}
holds for $L$}\}.
 \]
The latter claim is in turn implied by the following fact:
for every $L_0>1+s$ there exists a constant $\eta(L_0)>0$
such that, if $L\in \cS$ and $L<L_0$ then $L+\eta(L_0) 
\in \cS$.
In order to specify $\eta(L_0)$ we need to introduce several 
dimensional constants; to show the consistency of their choices,
we declare them at the beginning (the readers can skip this list 
and refer to it each time the constants are introduced):
\begin{itemize}
\item $C_{\ref{l:misura 2}}:= 
10^{3(n-1)}C_{\ref{t:Reif discreto}}(n)$,
where $C_{\ref{t:Reif discreto}}(n)>1$ is the dimensional constant
of Theorem~\ref{t:Reif discreto}; 
\item $\lambda = \min\{10^{-3}, 16^{2-n}\,C_{\ref{l:misura 
2}}^{-1}\}$;
\item $\tau = \min\big\{\lambda^2,
10^{-20n}\,C_{\ref{p:mean-flatness vs freq}}(2L_0,45)^{-2}
\,C_{\ref{l:misura 2}}^{-2}\,\lambda^{4n}\,
\delta_{\ref{t:Reif discreto}}^2(\lambda) \big\}$,
where $C_{\ref{p:mean-flatness vs freq}}(2L_0,45)$ is the constant
in Proposition~\ref{p:mean-flatness vs freq} corresponding to
$R=45$ and $A=2L_0$, and
$\delta_{\ref{t:Reif discreto}}(\lambda)$ is
the constant introduced in Theorem~\ref{t:Reif discreto};
\item $0<\eta\leq \min\big\{\eta_{\ref{p:rigidity}}(\tau,2L_0), \tau, L_0 \big\}$,
where $\eta_{\ref{p:rigidity}}(\tau,2L_0)$ is the constant introduced in 
Proposition~\ref{p:rigidity} with parameters, $\tau$ and $2L_0$.
\end{itemize}
Note that, for ever $L< L_0$ we have that $L+\eta\leq 2L_0$ 
with such a choice of $\eta$.

Then, the proof of Proposition~\ref{p:induzione} consists in showing that
\eqref{e:IS} holds for $L+\eta$, supposing that it has been verified 
for $L<L_0$. We proceed in several steps.
\medskip

\noindent{\bf 1.} 
Let $u$ be a solution in $B_2$ of the lower
dimensional obstacle problem with $\Theta_u(0,1)\leq L+\eta$, and let 
$r\in(0,1)$ be the size of the tubular neighborhood in \eqref{e:IS}
(recall that $\rho=1$ by scaling).
For every $x\in \Gamma(u)\cap B_{\sfrac12}$ we set
$s_x := \max\big\{r_x, 2r\big\}$ with
\begin{gather*}
r_x := 
\begin{cases}
\inf\big\{ t\in(0,1] \;:\;\Theta_u(x,t) > L \big\}
&\text{if } I_u(x,1)>L,\\
1 & \text{if } I_u(x,1)\leq L.
\end{cases}
\end{gather*}
By definition, if $I_u(x,1)> L$, then
\begin{align}\label{e: def y_x}
\exists\; y_{x} \in \bar B_{s_{x}}(x)\cap \Gamma(u)\quad
\text{such that }\;
I_u(y_x,s_{x})\geq L.
\end{align}
Let now $\{x_i\}_{i\in J}\subset \Gamma(u)\cap B_{\sfrac12}$
be a finite collection of points such that
the balls $B_{\sfrac{s_{x_i}}{2}}(x_i)$
constitute a Vitali covering of $\Gamma(u) \cap B_{\sfrac12}$: 
\ie
\begin{align}\label{e:covering}
B_{\sfrac{s_{x_i}}{10}}(x_i) \cap B_{\sfrac{s_{x_j}}{10}}(x_j)
=\emptyset\quad\forall\;i\neq j\in J
\quad\text{and}\quad
\Gamma(u)\cap B_{\sfrac12}\subset \bigcup_{i\in J}
B_{\sfrac{s_{x_i}}{2}}(x_i).
\end{align}
By construction, we have that
\begin{itemize}
 \item[(i)] $\cT_{r}\big(\Gamma(u)\cap B_{\sfrac{s_{x_i}}2}(x_i)\big)
\subseteq B_{s_{x_i}}(x_i)$, for all $i\in J$ because 
$2r\leq s_{x_i}$, 
\item[(ii)]
$\Theta_u(x_{i},s_{x_i}) = L$ if $s_{x_i} >  2r$.
\end{itemize}
The key estimate is to show that there exists a dimensional constant 
$\bar C>0$ such that
\begin{equation}\label {e:iii}
\sum_{i\in J} s_{x_i}^{n-1} \leq \bar C.
\end{equation}
Indeed, assuming momentarily \eqref{e:iii} 
we can prove \eqref{e:IS} for $L+\eta$ as follows: 
\begin{align}
\cL^{n+1}\Big(\cT_{r}\big( \Gamma(u) \cap B_{\sfrac12} \big)
\Big)
&\stackrel{\eqref{e:covering}}{\leq} \sum_{i\in J} 
\cL^{n+1}\Big(\cT_{r}(\Gamma(u)\cap B_{\sfrac{s_{x_i}}2}(x_i))\Big)
\notag\\ & 
\stackrel{\textup{(i)}}{\leq}\sum_{i\in J \, : \,s_{x_i} > 2r } 
\cL^{n+1}\Big(\cT_{r}(\Gamma(u)\cap B_{\sfrac{s_{x_i}}2}(x_i))\Big)
+ \sum_{i\in J \, :\, s_{x_i} =\, 2r }\cL^{n+1}\big( B_{s_{x_i}}(x_i)\big)
\notag\\ & 
\leq\sum_{i\in J \, : \,s_{x_i} > 2r }
C_{\ref{p:induzione}}(L) \, r^2\,s_{x_i}^{n-1}
+ \sum_{i\in J \, :\, s_{x_i} =\, 2r } \omega_{n+1}\,
(2r)^{n+1}\notag\\
& \leq \left(C_{\ref{p:induzione}}(L)+2^{n+1}\omega_{n+1}\right)\, 
r^2 \sum_{i\in J} s_{x_i}^{n-1}
\stackrel{\eqref{e:iii}}{\leq}
C_{\ref{p:induzione}}(L+\eta)\, r^2,\notag
\end{align}
with $C_{\ref{p:induzione}}(L+\eta) := 
\left(C_{\ref{p:induzione}}(L)+2^{n+1}\omega_{n+1}\right)\,\bar C$.
In the third inequality, we have used \eqref{e:IS} itself with bound $L$ 
on $\Theta_u(x_i,s_{x_i})$ in view of (ii).

\medskip

\noindent {\bf 2.}
Next we want to prove the claim \eqref{e:iii}. 
Let $\lambda$ be the constant introduced at the beginning,
we consider a suitable decomposition of the sets of centers
$\{x_i\}_{i\in J}$:
\[
\{x_i\}_{i\in J}
= A^{(0)}\cup \bigcup_{l=1}^{N(\lambda)}A^{(l)},
\]
with $A^{(0)} := \big\{x_i,\,i\in J\,:\,s_{x_i}\geq \lambda^2\big\}$,
$N(\lambda) = \lfloor 10^n\lambda^{-3n} \rfloor +1$ 
and $A^{(l)}$ satisfying the following condition for $l>0$:
\begin{align}\label{e:condizioni Al0}
\forall\; x_i,x_j\in A^{(l)},\;\; x_i\in B_{\sfrac{s_{x_j}}{\lambda}}(x_j)
\quad\Longrightarrow\quad s_{x_i}\leq \lambda^2\,s_{x_j}.
\end{align}
To see that such a decomposition exists,
we follow \cite[Lemma~5.4]{NaVa1} and
proceed inductively.
We order the points in $A\setminus A^{(0)}$
according to a decreasing order of the
corresponding radii: \ie, 
$A\setminus A^{(0)} =\{p_i\}_{i\in J'}$ with 
$s_{p_{i+1}}\leq s_{p_i}$. Then, $p_1\in A^{(1)}$ and,
if $p_1,\ldots, p_{i-1}$
have been sorted out, we 
assign $p_i$ to some $A^{(l)}$
so that $A^{(l)}$ does not contain any point $p_j$ with $j\leq 
i-1$, for which
\begin{align}\label{e:condizioni Al}
p_i\in B_{\sfrac{s_{p_j}}{\lambda}}
(p_j)
\quad\text{and}\quad
\lambda^2 s_{p_j}\leq s_{p_i}\leq
s_{p_j}.
\end{align}
Note that for every $j$ satisfying \eqref{e:condizioni Al} we have
$|p_i-p_j|\leq 
\sfrac{s_{p_j}}{\lambda}\leq 
\sfrac{s_{p_i}}{\lambda^3}$,
thus $p_j \in 
B_{\sfrac{s_{p_i}}{\lambda^3}}(p_i)$;
moreover, the balls $B_{\sfrac{s_{p_i}}{10}}(p_j)$ 
with $j$ as in \eqref{e:condizioni Al}
are disjoint, as
$s_{p_i}\leq s_{p_j}$ and 
$B_{\sfrac{s_{p_j}}{10}}(p_j)$
are disjoint by construction.
Therefore, since $N(\lambda)$ is strictly bigger 
than the number of disjoint balls with radius 
$\sfrac{1}{10}$ in $B_{\sfrac{1}{\lambda^3}}$
and center on $B_{\sfrac{1}{\lambda^3}}'$,
one can surely find
$l$ so that $p_j\not\in A^{(l)}$ for all $j$ as in 
\eqref{e:condizioni Al}.

Let us check that \eqref{e:condizioni Al0} holds. Indeed,
if $j<i$ ({\ie}~$s_{p_i}\leq s_{p_j}$) and
$p_i \in 
B_{\sfrac{s_{p_j}}{\lambda}}(p_j)$,
then the second condition in \eqref{e:condizioni Al} must fail
(being $p_j\in A^{(l)}$), {\ie}~$s_{p_i}< 
\lambda^2\,s_{p_j}$.
On the other hand, if $i<j$ ({\ie}~$s_{p_j}\leq 
s_{p_i}$), from 
$p_i \in 
B_{\sfrac{s_{p_j}}{\lambda}}(p_j)$, we deduce
$p_j \in 
B_{\sfrac{s_{p_j}}{\lambda}}(p_i)\subset 
B_{\sfrac{s_{p_i}}{\lambda}}(p_i)$
and, as $p_j\in A^{(l)}$,
the second condition in \eqref{e:condizioni Al} must fail,
{\ie}~$s_{p_j}< \lambda^2\,s_{p_i}$. But this is a 
contradiction because
$\lambda<\sfrac{1}{10}$ and
\[
\frac{s_{p_i}}{10}\leq |p_i-p_j|\leq  
\frac{s_{p_j}}{\lambda}< \lambda\,s_{p_i}.
\]

\medskip

\noindent {\bf 3.}
Next, for $l\in\{0,\ldots, N(\lambda)\}$ we introduce the measures:
\begin{equation}\label{e:mu l}
\mu^l := \sum_{x\in A^{(l)}} s_{x}^{n-1} \delta_{x}.
\end{equation}
To conclude \eqref{e:iii}, we show that
there exists a dimensional constant $C_0>0$ such that
\begin{equation}\label{e:stima misura}
\mu^l \big(B_{\sfrac12}\big) \leq C_0
\quad\forall\; l\in \{0,\ldots, N(\lambda)\}.
\end{equation}
Indeed, from \eqref{e:stima misura} we infer \eqref{e:iii} with
the constant $\bar C := \big(N(\lambda)+1\big)\, C_0$:
\begin{align*}
\sum_{i\in J} s_{x_i}^{n-1} &\leq 
\sum_{l=0}^{N(\lambda)}\mu^l(B_{\sfrac12})\leq 
\big(N(\lambda)+1\big)\, C_0 = \bar C.
\qedhere
\end{align*}
The case $l=0$ is straightforward:
since the balls $B_{\sfrac{\lambda^2}{10}}(x)$ with $x\in A^{(0)}$
are pairwise disjoint, contained in $B_1$ and with center $x\in 
B_{\sfrac12}'$, 
then $\cH^0\big(A^{(0)}\big)\leq \frac{10^n}{\lambda^{2n}}$. Being 
$s_x\leq 2$ we deduce that 
\[
\mu^0(B_{\sfrac12}) =
\sum_{x\in A^{(0)}} s_x^{n-1}\leq \cH^0\big(A^{(0)}\big)2^{n-1}
\leq \frac{20^n}{\lambda^{2n}},
\]
and estimate \eqref{e:stima misura} for $l=0$ 
follows as soon as $C_0\geq \sfrac{20^n}{\lambda^{2n}}$.

For the remaining cases, we are going to show the following
lemma.

\begin{lemma}\label{l:misura 2}
Let $\mu^l$ be the measures in \eqref{e:mu l} with $l\geq1$. Then, 
\begin{equation}\label{e:stima misura 2}
\mu^l \big(B_{\rho}(x)\big) \leq C_{\ref{l:misura 2}}\,\rho^{n-1},
\end{equation}
for every $x\in\spt(\mu^l)$ and for every 
$\rho\in (s_x,\lambda^2]$,
where $C_{\ref{l:misura 2}}>0$ is the dimensional
constant introduced at the beginning.
\end{lemma}

Lemma~\ref{l:misura 2} implies \eqref{e:stima misura}. Indeed, let us
consider a maximal subset of points $\{x_{i}\}_{i\in J^{(l)}}\subseteq 
A^{(l)}$ with $|x_{i} - x_{j}|\geq \lambda^2$
for all $i\neq j \in J^{(l)}$. Then, 
the balls $\{B_{\sfrac{\lambda^2}{2}}(x_{i})\}_{i\in J^{(l)}}$ are 
disjoint, contained in $B_1$ (as $\lambda<\sfrac12$), and with 
centers $x_{i}\in B_{\sfrac12}'$.
Thus, $\cH^0(J^{(l)})\leq \sfrac{2^n}{\lambda^{2n}}$
and by maximality of the number of points in $\{x_i\}_{i\in J^{(l)}}$
we have also $\spt(\mu^l)\subset
\cup_{i\in J^{(l)}}^{}B_{\lambda^2}(x_{i})$.
Then
\[
\mu^l(B_{\sfrac12}) \leq 
\sum_{i\in J^{(l)}}\mu^l\big(B_{\lambda^2}(x_{i})\big)
\stackrel{\eqref{e:stima misura 2}}{\leq} \cH^0(J^{(l)}) \, C_{\ref{l:misura 2}}
\lambda^{2(n-1)}\leq \frac{2^n}{\lambda^{2}}\, C_{\ref{l:misura 2}},
\]
and \eqref{e:stima misura} follows with
$C_0:=\max\big\{\sfrac{2^n\, C_{\ref{l:misura 2}}}{\lambda^{2}}, 
\sfrac{20^n}{\lambda^{2n}}\big\}$.

\medskip Proposition~\ref{p:induzione}, and hence Theorem~\ref{t:misura},
are now established once we show Lemma~\ref{l:misura 2}.

\subsection{Proof of Lemma~\ref{l:misura 2}}\label{ss:stima misura}
We fix $l\in\{1,\ldots, N(\lambda)\}$, and set 
$s_{\min} :=\min_{}\big\{s_{w}\;:\; w\in A^{(l)}\}$,
$k_{\max} := \lfloor\log_{\lambda}(s_{\min})\rfloor$.
Note that $k_{\max}\geq 2$ and
$\lambda^{k_{\max}+1}< s_{\min}\leq \lambda^{k_{\max}}$.
We prove \eqref{e:stima misura 2} for all
$\rho\in (\max\{\lambda^{k+1},s_{\min}\},\lambda^k]$ by induction 
on $k\in \{2,\ldots,k_{\max}\}$ in decreasing order.
More precisely, the base induction step is for $k=k_{\max}$.
In this case, for every point
$w_0\in\spt(\mu^l)=A^{(l)}$
with $s_{w_0}\leq \lambda^{k_{\max}}$
we have that $\spt(\mu^l)\cap B_{\lambda^{k_{\max}}}(w_0) =\{w_0\}$,
from which \eqref{e:stima misura 2} readily follows.
Indeed, if $w\in B_{\lambda^{k_{\max}}}(w_0)\cap A^{(l)}$
is different from $w_0$, then 
$w\in B_{\sfrac{s_{w_0}}{\lambda}}(w_0)\cap A^{(l)}$
as $s_{w_0}\in(\lambda^{k_{\max}+1},\lambda^{k_{\max}}
]$ and by \eqref{e:condizioni Al0} we reach a contradiction 
\[
\lambda^{k_{\max}+1}< s_{\min} \leq s_w
\stackrel{\eqref{e:condizioni Al0}}{\leq} \lambda^2 s_{w_0} \leq
\lambda^{k_{\max}+2}.
\]
We can then proceed inductively: we assume that we have shown 
\eqref{e:stima misura 2} for every $\rho\in (s_{\min},\lambda^{k+1}]$
for some $k\in\{2,\dots,k_{\max}-1\}$ and for all $w\in\spt(\mu^l)$ with 
$s_w\leq\lambda^{k+1}$, and then we prove that
\begin{equation}\label{e:stima misura2}
\mu^l \big(B_{t}(w_0)\big) \leq C_{\ref{l:misura 2}} \,t^{n-1}
\quad \forall\;t \in (\lambda^{k+1},\lambda^k],\quad
\forall\; w_0\in \spt(\mu^l)\;
\text{ with }\;s_{w_0}<t.
\end{equation}

\medskip

\noindent {\bf 1.}  
Let $t \in (\lambda^{k+1},\lambda^k]$ with $k\geq2$ and
$w_0\in \spt(\mu^l)$ be such that $s_{w_0}< t$.
We set $W := \spt(\mu^l)\cap B_{t}(w_0)$ and
\begin{gather*}
W^{(1)} := \big\{w\in W \;:\; I_u(w,s_w) < L - \tau\big\}
\quad\text{and}\quad
W^{(2)} := W \setminus W^{(1)},
\end{gather*}
where $\tau$ is the constant introduced at the beginning.
Next we order the points in $W^{(1)}$ in such a way that
$W^{(1)} = \{p_{h}\}_{h}$
with $s_{p_{h}}\geq s_{p_{h+1}}$;
and we define inductively $z_1:= y_{p_{1}}$ and
for $\alpha\geq2$
\begin{gather}
z_\alpha:=y_{p_{m_\alpha}} \;\textrm{ with $m_\alpha = \min\Big\{ h \;:\; y_{p_h}\not\in
\cup_{j=1}^{\alpha-1}B_{s_{z_j}}(z_j)\Big\}$},
\end{gather}
where the $y_p$'s are the points defined in \eqref{e: def y_x}
(which exist because $p_h\in A^{(l)}$ with $l\geq1$
implies $s_{p_h} < \lambda^2<1$, {\ie}~$I_u(p_h,1)>L$).
Let $Z$ be the set of the selected points $z_\alpha$'s
and set $s_{z_\alpha} := s_{p_{m_\alpha}}$ (with a slight abuse of notation),
$E:= \cup_{z_\alpha\in Z} B_{s_{z_\alpha}}(z_\alpha)$ and
\[
\mu^l_1 := \sum_{z_\alpha\in Z}
s_{z_\alpha}^{n-1}\,\delta_{z_\alpha}
+ \sum_{w\in W^{(2)}\setminus E} s_w^{n-1}\delta_w.
\]
The measure $\mu^l_1$ satisfies the following five properties:
\begin{gather}
\forall\; p\in \spt(\mu^l_1)\quad
\Delta^1_{s_p}(p) = I_u(p,1)-I_u(p,s_p)\leq \eta+\tau
\leq 2\,\tau,\label{e:P0}\\
\forall\; p \neq p'\in \spt(\mu^l_1)\quad
\max\big\{{s_{p}},{s_{p'}}\big\} \leq 20\,|p-p'|,\label{e:P1}\\
\spt(\mu^l_1)\subset B_{11t}(w_0),\label{e:P1.5}\\
\mu^{l}\big(B_{t}(w_0)\big) \leq 
2\,\mu^l_1(B_{11t}(w_0)),\label{e:P2}\\
\mu^l_1 \big(B_{\rho}(p)\big) \leq
10^{12 n}\,\lambda^{-2n}\,
C_{\ref{l:misura 2}}\,\rho^{n-1}
\quad\forall\;p\in \spt(\mu_1^l),
\;\forall\,\rho \in [\sfrac{s_p}{20}, 10^2\lambda^k].\label{e:P3}
\end{gather}
The properties \eqref{e:P0} and \eqref{e:P1} follows directly 
from the definition of $\mu^l_1$.
More precisely, for \eqref{e:P0} recall the choice $\eta\leq \tau$ and that 
by assumption $\Theta_u(0,1)\leq L+\eta$. Therefore, the conclusion follows 
either by \eqref{e: def y_x} if $p\in Z$ or 
otherwise by the very definition of $W^{(2)}$.  

For \eqref{e:P1} we distinguish three cases:
\begin{itemize}
\item[(i)] $p,\,p'\in Z$. Assume without loss of generality that 
$s_p\leq s_{p'}$, 
then by the selection procedure defining $Z$ itself $p\notin B_{s_{p'}}(p')$, 
and thus $s_{p'}<|p-p'|$;
\item[(ii)] $p\in Z$, $p'\in W^{(2)}\setminus E$. Then $p'\notin B_{s_p}(p)$ 
by definition of $E$, so that $s_p<|p-p'|$. Moreover, if $p=y_w$, 
with $w\in W^{(1)}$, we use \eqref{e:covering} to infer
\[
s_{p'}<10|w-p'|\leq 10(|w-y_w|+|y_w-p'|)\leq 10(s_p
+|p-p'|)\leq 20 |p-p'|.
\]
\item[(iii)] $p,\,p'\in W^{(2)}\setminus E$. 
Since $B_{\frac{s_p}{10}}(p)\cap B_{\frac{s_{p'}}{10}}(p')=\emptyset$ by 
\eqref{e:covering}, $\max\{s_p,s_{p'}\}< 10|p-p'|$. 
\end{itemize}
For what concerns \eqref{e:P1.5}, we notice that 
for all $w\in W$
by \eqref{e:covering} we have that $s_w < 10|w-w_0|\leq 10 t$, and therefore
\[
|y_w - w_0|\leq |y_w-w| + 
|w-w_0|\leq s_w +
|w-w_0| \leq 11 |w-w_0| \leq 11 t.
\]
Eq.~\eqref{e:P2} and \eqref{e:P3} are proven in the next two steps.
The proof of \eqref{e:stima misura2} will then be a consequence
of \eqref{e:P0} -- \eqref{e:P3} only and it will be detailed in
step 4.

\medskip

\noindent{\bf 2.}
For what concerns \eqref{e:P2}, for every $z_\alpha\in Z$
we introduce the sets
\[
W^{z_\alpha}
:=
\Big(W^{(2)}\cap \big(B_{s_{z_\alpha}}(z_\alpha)\setminus
\cup_{j=1}^{\alpha -1}B_{s_{z_j}}(z_j)\big)\Big)
\cup
\Big\{w\in W^{(1)}\,:\,y_w\in B_{s_{z_\alpha}}(z_\alpha)
\setminus \cup_{j=1}^{\alpha -1}B_{s_{z_j}}(z_j)\Big\}.
\]
Hence, as by the very definition of $E$
\[
W^{(2)}\cap E=\cup_\alpha \Big(W^{(2)}\cap \big(B_{s_{z_\alpha}}(z_\alpha)\setminus
\cup_{j=1}^{\alpha -1}B_{s_{z_j}}(z_j)\big)\Big),
\]
and by that of $Z$
\[
 W^{(1)}=\cup_\alpha \Big\{w\in W^{(1)}\,:\,y_w\in B_{s_{z_\alpha}}(z_\alpha)
\setminus \cup_{j=1}^{\alpha -1}B_{s_{z_j}}(z_j)\Big\},
\]
then $W=\cup_{z\in Z} W^{z}\cup (W^{(2)}\setminus E)$ and
\begin{equation}\label{e:stima banana}
\mu^{l}(B_t(w_0))= \sum_{z\in Z} \mu^{l}(W^z) +
\mu^{l}\big(W^{(2)}\setminus E\big)
=\sum_{z\in Z} \mu^{l}(W^z) +
\mu^{l}_1\big(W^{(2)}\setminus E\big).
\end{equation}
We will prove \eqref{e:P2} by showing
that, for every $z_\alpha \in Z$, we have
\begin{align}\label{e:stima chiava}
\mu^{l}(W^{z_\alpha}) \leq 2 s_{z_\alpha}^{n-1}.
\end{align}
Indeed, from \eqref{e:stima chiava} we immediately deduce that
\begin{align*}
\mu^{l}(B_t(w_0))&\stackrel{\eqref{e:stima banana}}{=}
\sum_{z_\alpha\in Z} \mu^{l}(W^{z_\alpha}) +
\mu^{l}_1\big(W^{(2)}\setminus E\big)
\stackrel{\eqref{e:stima chiava}}{\leq}
2 \sum_{z_\alpha\in Z} s_{z_\alpha}^{n-1} +
\mu^{l}_1\big(W^{(2)}\setminus E\big)\\
&\leq 
2\,\mu^l_1(B_{11t}(w_0)).
\end{align*}
The key observation to establish \eqref{e:stima chiava} is the following:
let $\bar w \in W^{(1)}$ be such that $z_\alpha = y_{\bar w}$.
Then, by definition
\begin{gather*}
I(z_\alpha, s_{\bar w})- I(\bar w, s_{\bar w}) >
L - L + \tau = \tau.
\end{gather*}
We can then apply Proposition~\ref{p:rigidity}
in $B_{8s_{\bar w}}(z_\alpha)$ 
with parameters $\tau$, $2L_0$. Indeed, $I_u(z_\alpha,8s_{\bar w})\leq 
\Theta_u(0,1)\leq L+\eta\leq 2L_0$ (recall that $L<L_0$ and we have 
chosen $\eta\leq L_0$). Moreover, as we have imposed $\eta\leq 
\eta_{\ref{p:rigidity}}(\tau,2L_0)$,
we deduce that the first case of the dichotomy
of Proposition~\ref{p:rigidity} does not occur:
{\ie}~there exists a $(n-2)$-dimensional
affine subspace passing through $z_{\alpha}$ such that
\begin{align}\label{e:proposizione magica}
\forall q\in \Gamma(u)\cap B_{4s_{\bar w}}(z_\alpha)\quad
\text{with}\quad
\Delta^{2s_{\bar w}}_{s_{\bar w}}(q)\leq \eta
\quad\Longrightarrow\quad
q\in \cT_{2\tau s_{\bar w}}(V).
\end{align}
Eq.~\eqref{e:proposizione magica} is the main ingredient of the proof,
because it implies that all the points in $W^{z_\alpha}$
different from $\bar w$
have clustered around a lower dimensional space $V$, namely
\begin{equation}\label{e:cluster}
W^{z_\alpha}\setminus\{\bar w\}\subseteq \cT_{4\lambda^2 s_{\bar w}}(V).
\end{equation}
Indeed, consider a generic point $w\in W^{z_\alpha}\setminus\{\bar w\}$.
If $w \in W^{(2)}\cap B_{s_{z_\alpha}}(z_\alpha)$, then
$w\in B_{2s_{\bar w}}(\bar w)$ and by \eqref{e:condizioni Al0}
we have $s_{w}\le \lambda^2 s_{\bar w}$.
In turn this implies that $y_{w} \in B_{s_{w}}(w)
\subset B_{2s_{z_\alpha}}(z_\alpha)$ and
\begin{align}\label{e:oscillazione piccola}
\Delta^{2s_{\bar w}}_{s_{\bar w}}(y_{w})
\leq I(y_{w},1) - I(y_{w},s_{w}) \leq L+\eta - L = \eta.
\end{align}
Therefore, by \eqref{e:proposizione magica} we infer that 
$y_{w} \in \cT_{2\tau s_{\bar w}}(V)$ and, since $\tau \leq {\lambda^2}$,
also $w\in \cT_{2\tau s_{\bar w} + s_{w}}(V)\subset \cT_{4\lambda^2 
s_{\bar 
w}}(V)$.
On the other hand, if $w\in W^{(1)}$, then by the selection
procedure (recall the decreasing order of the radii $s_{z_j}$),
we have that $s_{w}\leq s_{\bar w}$: in particular
$w\in B_{s_{\bar w}}(y_{w})\subset B_{3s_{\bar w}}(\bar 
w)$.
Therefore, thanks to \eqref{e:condizioni Al0}
we have also $s_{w}\le \lambda^2 s_{\bar w}$ and \eqref{e:oscillazione piccola}
holds. By \eqref{e:proposizione magica}
$y_{w} \in \cT_{2\tau s_{\bar w}}(V)$ and hence $w\in \cT_{4\lambda^2 s_{\bar 
w}}(V)$, thus showing \eqref{e:cluster}.

Then the proof of \eqref{e:stima chiava} follows from 
an elementary covering argument.
Let $Q'\subset W^{z_\alpha}\setminus\{\bar w\}$
be a maximal collection of points such that the balls
$\big\{B_{\sfrac{\lambda s_{\bar w}}{20}}(p)\}_{p\in Q'}$
are pairwise disjoint: in particular
$W^{z_\alpha}\setminus\{\bar w\}
\subset \cup_{p\in Q'} B_{\sfrac{\lambda s_{\bar w}}{10}}(p)$.
Let $\pi_V:\R^n\to V$ be the nearest point projection on $V$ and note that,
since $\lambda\leq \sfrac{1}{160}$, we have
\[
B_{\sfrac{\lambda s_{\bar w}}{40}}(\pi_V (p))  \subset
B_{\sfrac{\lambda s_{\bar w}}{20}}(p) \subset B_{4 s_{\bar w}}(z_\alpha),
\quad\forall \; p\in Q',
\]
where we used that every $p\in W^{z_\alpha}$ is contained in
$B_{3s_{\bar w}}(z_\alpha)$, and thus $p\in B_{4s_{\bar w}}(\bar w)$.
Therefore, $B_{\sfrac{\lambda s_{\bar w}}{40}}(\pi_V (p)) \cap V$ are 
pairwise disjoint
for $p\in Q'$ and contained in $B_{4s_{\bar w}}(\pi_V (z_\alpha))\cap V$.
This allows us to give an estimate on the cardinality of $Q'$, namely
$\cH^0(Q') \leq \sfrac{160^{n-2}}{\lambda^{n-2}}$. In proving the latter 
estimate we have crucially used that $V$ has dimension $n-2$.
Now by the inductive hypothesis \eqref{e:stima misura2} we get \eqref{e:stima chiava}:
\begin{align*}
\mu^l (W^{z_\alpha}) 
& \leq \mu^l(\{\bar w\})+
\sum_{p\in Q'}
\mu^{l} \big(B_{\sfrac{\lambda s_{\bar w}}{10}}(p)\big)\leq 
s_{\bar w}^{n-1}+\cH^0(Q')
\,C_{\ref{l:misura 2}} \,(\sfrac{\lambda}{10})^{n-1}\, s_{\bar w}^{n-1}
\\
&\leq s_{\bar w}^{n-1}
+16^{n-2}\lambda\,C_{\ref{l:misura 2}}\, s_{\bar w}^{n-1}
\leq 2\,s_{\bar w}^{n-1} = 2 \,s_{z_\alpha}^{n-1},
\end{align*}
thanks to the choice $\lambda < 16^{2-n}C_{\ref{l:misura 2}}^{-1}$.
We can apply the inductive hypothesis to 
$B_{\sfrac{\lambda s_{\bar w}}{10}}(p)$ since
$s_p\leq \lambda^2s_{\bar w} < \sfrac{\lambda s_{\bar w}}{10}
\leq \lambda^{k+1}$ (the first inequality holds thanks to 
\eqref{e:condizioni Al0} because for every $p\in Q'$ we have 
that $p\in B_{4s_{\bar w}}(\bar w)$, and the last one in view of 
$s_w\leq 10\,t\leq 10\,\lambda^k$ for every $w\in W$).

\medskip

\noindent{\bf 3.}
We show next \eqref{e:P3}. Let $p, \rho$ be as in the statement. 
For every $q\in\spt(\mu_1^l)\cap B_\rho(p)$ let $x_q\in\spt(\mu^l)$ 
be a point such that $y_{x_q}=q$ if $q\notin\spt(\mu^l)$ and
coinciding with $q$ itself otherwise.
Then,
\[
|x_{p}-x_q|\leq |x_{p}-p|+|p-q|+|q-x_q|
 \leq s_{p}+\rho+s_q
\stackrel{\eqref{e:P1}}{\leq} \rho+40|q-p|< 
41 \rho.
\]
Therefore, for every point $q\in \spt(\mu^l_1) \cap B_{\rho}(p)$
we have that the corresponding point $x_q$
belongs to 
$\spt(\mu^l)\cap B_{41\rho}(x_p)$,
so that
\[
\mu^l_1\big(B_\rho(p)\big)\leq 
\mu^l\big(B_{41\rho}(x_p)\big).
\]
The proof of \eqref{e:P3} is now a consequence of
the inductive hypothesis \eqref{e:stima misura2}
and a covering argument.
Indeed, 
\begin{itemize}
 \item[(i)] if $41\rho\leq \lambda^{k+1}$: 
we can apply \eqref{e:stima misura 2} directly
(since $s_p\leq 20\rho$ by assumption), and infer
that $\mu^l\big(B_\rho(p)\big)\leq C_{\ref{l:misura 2}}
41^{n-1}\rho^{n-1}$;

\item[(ii)] if $41\rho>\lambda^{k+1}$: 
we cover $\spt(\mu^l)\cap B_{41\rho}(x_p) $ with balls
$B_{\sfrac{\lambda^{k+1}}{10}}(w)$
having centers $w\in\spt(\mu^l)$
such that half the balls are disjoint. 
Since $\rho\leq 10^2\lambda^k$ by assumption (cf.~\eqref{e:P3}) 
and the centers are in $B_1'$, the cardinality of the cover can 
be estimated by $(\sfrac{10^5}{\lambda})^n$.
Moreover, $s_w\leq 20\cdot 41\rho\leq 10^5\,\lambda^k$
in view of $w\in B_{41\rho}(x_p)$ (cf.~\eqref{e:covering}),
and $\rho\leq 10^2\lambda^k$ by assumption (cf.~\eqref{e:P3}) .
Hence, in case $s_w\leq \lambda^{k+1}$ we can use
the inductive hypothesis \eqref{e:stima misura2}
to infer that $\mu^l(B_{\sfrac{\lambda^{k+1}}{10}}(w))\leq
C_{\ref{l:misura 2}}\lambda^{(k+1)(n-1)}$. Otherwise, if
$s_w > \lambda^{k+1}$, 
$\spt(\mu^l)\cap B_{\sfrac{\lambda^{k+1}}{10}}(w) =\{w\}$ 
by \eqref{e:covering}, and thus
$\mu^l(B_{\sfrac{\lambda^{k+1}}{10}}(w)) = s_w^{n-1}\leq
10^{5(n-1)}\,\lambda^{k(n-1)}$.
In conclusion, 
recalling that $\lambda^{k+1}\leq 41 \rho$, we infer that
\begin{align*}
\mu^l_1\big(B_\rho(p)\big)&
\leq \mu^l\big(B_{41\rho}(x_p)\big)
\leq 
(\sfrac{10^5}{\lambda})^n 
10^{5(n-1)}\,\,C_{\ref{l:misura 2}}\lambda^{k(n-1)}\\
&\leq  10^{12 n}\,\lambda^{-2n}\,
C_{\ref{l:misura 2}}\,\rho^{n-1}.
\end{align*}
\end{itemize}

\medskip

\noindent{\bf 4.}
We are now in the position to infer \eqref{e:stima misura2}
from \eqref{e:P0}--\eqref{e:P3}, thus concluding the proof of Lemma~\ref{l:misura 2}.
We start off estimating the generalized Jones' number for $\mu^l_1$
(for simplicity we omit the subscripts in their notation):
for every $\rho \in (0,{44}t]$, with $t\in(\lambda^{k+1},\lambda^k]$ 
by \eqref{e:stima misura2}, and $w\in\spt(\mu^l_1)$,
using Proposition~\ref{p:mean-flatness vs freq} 
with parameters $A= 2L_0$
and $R=45$ (recall that $L_0+\eta \leq 2L_0$ and
do not confuse the radius $R$ there with the one in this proof) we infer
\begin{equation}\label{e:jones numbers}
\beta^2(w, \rho) \leq \frac{C_{\ref{p:mean-flatness vs freq}}}{\rho^{n-1}}
\int_{B_{\rho}(w)} 
\Delta^{94\rho}_{20\rho}(z)\,
\chi_{[0,20\rho]} (s_z)\,\d \mu^l_1(z),
\end{equation}
having used that $s_z\leq 20\rho$ if 
$z\in \spt(\mu^l_1)\cap B_\rho(w)$ by \eqref{e:P1}.
Integrating \eqref{e:jones numbers} over $B_R({\bar w})$
for $\bar w\in \spt(\mu_1^l)$, with 
$R\in(0,44 \lambda^k]$ and
$\rho \in (0,R]$, we get
\begin{align}\label{e:integrale jones}
\int_{B_{R}(\bar w)} \beta^2\big(w, \rho\big)
\,\d \mu^l_1(w) &\stackrel{\eqref{e:jones numbers}}{\leq}
\frac{C_{\ref{p:mean-flatness vs freq}}}{\rho^{n-1}}
\int_{B_{R}(\bar w)} \int_{B_{\rho}(w)}
\Delta^{94\rho}_{20\rho}(z)\,
\chi_{[0,20\rho]}(s_z)
\,\d \mu^l_1(z)\,\d\mu^l_1(w)\notag\\
&\leq \frac{C_{\ref{p:mean-flatness vs freq}}}{\rho^{n-1}}
\int_{B_{R+\rho}(\bar w)} \mu^l_1\big(B_{\rho}(z)\big)
\Delta^{94\rho}_{20\rho}(z)\,
\chi_{[0,20\rho]}(s_z)
\,\d \mu^l_1(z)\notag\\
& \stackrel{\eqref{e:P3}}{\leq}
\bar C(\lambda)
\int_{B_{2R}(\bar w)} 
\Delta^{94\rho}_{20\rho}(z)\,
\chi_{[0,20\rho]}(s_z)
\,\d \mu^l_1(z).
\end{align}
In the second inequality we have used Fubini's theorem, and 
we have set for simplicity
$\bar C(\lambda) := C_{\ref{p:mean-flatness vs freq}}\,
C_{\ref{l:misura 2}}\,\sfrac{10^{12n}}{\lambda^{2n}}$.
Let us now introduce the following notation for the
average oscillation of a measure $\mu$ at scale $\lambda$ on the ball
$B_\varrho(\bar w)$:
\[
\textup{Osc}^\lambda_{\mu}(\bar w, \varrho):=\int_{B_\varrho
(\bar w)}\sum_{j=0}^{+\infty}\beta^2(y,\lambda^j\,\varrho)\,\d\mu(y).
\]
Then, summing \eqref{e:integrale jones}
for $\rho = \lambda^j R$ with $j\in\N$
and using $\lambda\leq 10^{-3}$, we get
\begin{align}\label{e:integrale jones 2}
\textup{Osc}^\lambda_{\mu_1^l}(\bar w, R)
& \leq \bar C(\lambda)
\int_{B_{2R}(\bar w)} 
\sum_{j=0}^{\lfloor\log_\lambda(\sfrac{s_z}{20R}) 
\rfloor}
\Delta^{20\lambda^{j-1} R}_{20\lambda^j R}(z)
\,\d \mu^l_1(z) \notag\\&
\leq \bar C(\lambda) \int_{B_{2R}(\bar w)} 
\Delta^{1}_{s_z}(z) \,\d \mu^l_1(z)
\stackrel{\eqref{e:P0}}{\leq} 
2\,\tau\,\bar C(\lambda)\,
\mu_1^l(B_{2R}(\bar w))\notag\\
&\stackrel{\eqref{e:P3}}{\leq }
2^{n}\,\bar C(\lambda)^2 \,\tau\,
R^{n-1},
\end{align}
by taking into account that $\sfrac{s_z}{20}\leq \rho
<R$ and $20\lambda^{-1}R <1$ (being $R\leq 44\lambda^k$
and $k\geq2$).
In addition, we notice that in case $R<\sfrac{s_{\bar w}}{20}$, estimate 
\eqref{e:integrale jones 2} still holds true. Indeed, in such a case
$B_{R}(\bar w) \cap \spt(\mu^l_1) =\{\bar w\}$ and
by definition $\beta(\bar w,\lambda^j\,R) =0$ for every $j\in\N$,
so that $\textup{Osc}^\lambda_{\mu_1^l}(\bar w, R)=0$ 

Note moreover that \eqref{e:integrale jones 2}
can be extended to every ball $B_{R}(p)$
with $p\in B_{22t}(w_0)$ and $R\in(0,22t]$:
indeed, if $B_{R}(p) \cap \spt(\mu_1^l) = \emptyset$, then  
$\textup{Osc}^\lambda_{\mu_1^l}(p, R)=0$;
otherwise, if $w\in B_{R}(p)\cap \sup(\mu_1^l)$, then
$B_{R}(p)\subset B_{2R}(w)$ and 
\begin{equation}\label{e:integrale jones 3}
\textup{Osc}^\lambda_{\mu_1^l}(p, R)\leq 2^{n+1}
\textup{Osc}^\lambda_{\mu_1^l}(w, 2R) 
\stackrel{\eqref{e:integrale jones 2}}{\leq}
2^{2n+1}\,\bar C(\lambda)^2 \,\tau\,R^{n-1},
\end{equation}
being $\beta(y,\rho)\leq 2^{n+1}\beta(y,2\rho)$
for every $y$ and every $\rho>0$.

The conclusion of the proof is now an application
of the following result by Naber--Valtorta 
\cite[Theorem 3.4 \& Remark 3.9]{NaVa1}.

\begin{theorem}[Naber--Valtorta \cite{NaVa1}]\label{t:Reif discreto}
There is a dimensional constant $C_{\ref{t:Reif discreto}}(n)>0$
such that the following holds.
For every  $\lambda>0$, there exists $\delta_{\ref{t:Reif discreto}}(\lambda)>0$
with this property:
for every $\{B_{r_i}(x_i)\}_{i\in I}$ finite collection of pairwise
disjoint balls in
$B_2\subset\R^n$ and $\mu := \sum_{i\in I}r_i^{n-1}\delta_{x_i}$,
\[
\textup{Osc}^\lambda_{\mu}(x, \varrho)
\leq \delta_{\ref{t:Reif discreto}}^2(\lambda)\,\varrho^{n-1}
\quad\forall\; B_\varrho(x) \subset B_2
\quad\Longrightarrow \quad
\mu(B_1)\leq C_{\ref{t:Reif discreto}}.
\]
\end{theorem}

Renaming for simplicity the points in the support
of $\mu_1^l$ as $\mu^l_1 = \sum_i 
s_{p_i}^{n-1}\,\delta_{p_i}$,
we can apply Theorem~\ref{t:Reif discreto}
with $x_i := \sfrac{(p_i - w_0)}{(11t)}$,
$r_i := \sfrac{s_{p_i}}{(440t)}$, and
$\mu:=\sum_i r_i^{n-1}\delta_{x_i}$.
Indeed, from \eqref{e:P1.5} we have that $\spt(\mu)\subset B_1$
and from \eqref{e:P1} we have that $B_{r_i}(x_i)$ are pairwise
disjoint.
Moreover, from \eqref{e:integrale jones 3} and the choice of $\tau$
it follows that, for every $B_r(x)\subset B_2$ we have
\begin{align*}
\textup{Osc}^\lambda_{\mu}(x, r)
& = \frac{1}{(40)^{2(n-1)}(11t)^{n-1}}\,
\textup{Osc}^\lambda_{\mu_1^l}(w_0+11tx, 11tr)
\\
&\leq \frac{2^{2n+1}\,}{(40)^{2(n-1)}(11t)^{n-1}}\,
\bar C(\lambda)^2\,\tau\,(11tr)^{n-1}
\leq\delta_{\ref{t:Reif discreto}}^2(\lambda)\,r^{n-1}.
\end{align*}
We then conclude that $\mu^l_1(B_{11t}(w_0))
= (440t)^{n-1}\mu(B_1) \leq
C_{\ref{t:Reif discreto}}\,(440t)^{n-1}$ 
and, by the choice of the constant
$C_{\ref{l:misura 2}}$, we conclude \eqref{e:stima misura 2}:
\[
\mu^l(B_t(w_0))\stackrel{\eqref{e:P2}}{\leq}
2\mu^l_1(B_{11t}(w_0)) \leq
2\,C_{\ref{t:Reif discreto}}\,(440t)^{n-1}
= C_{\ref{l:misura 2}}\,t^{n-1}.
\]

%
%
\section{Structure of the free boundary $\cH^{n-1}$-a.e.}\label{s:rect}

In this section we give the proof of Theorem~\ref{t:rect}.
It is a consequence of Theorem~\ref{t:misura} and
of the following  rectifiability criterion recently
established by Azzam--Tolsa \cite[Theorem~1.1]{AzTo15}.
A similar criterion has also been established independently
by Naber--Valtorta in \cite[Theorem~3.3]{NaVa1}.

\subsection{Azzam--Tolsa rectifiability criterion}
We recall the following definition:
a Radon measure $\mu$ in $\R^n$ is called $k$-rectifiable
if
\begin{itemize}
\item[(i)] $\mu$ is absolutely continuous with respect
to the Hausdorff measure $\cH^{k}$, {\ie}~for every $E\subset\R^n$
\[
\cH^{k} (E) = 0 \quad \Longrightarrow \quad \mu(E) = 0,
\]

\item[(ii)] there exist at most countable many $C^1$ functions
$f_i:\R^k \to \R^n$, for $i \in \N$, such that 
\[
\mu\Big(\R^n \setminus \bigcup_{i \in \N} f_i(\R^k)\Big) = 0.
\]
\end{itemize}
A set $E\subset\R^n$ is said $\cH^k$-rectifiable if the associated 
measure
$\cH^k\res E$ is $k$-rectifiable.

The following is the rectifiability criterion we are going to
exploit:
in order to state it, we need to recall the notion
of upper-density of a measure
\[
\vartheta^{k,\star}(x, \mu) := \limsup_{r\to 0^+} 
\frac{\mu\big(B_r(x)\big)}{\omega_k\,r^k}.
\]

\begin{theorem}[Azzam--Tolsa \cite{AzTo15}]\label{t:AT}
Let $\mu$ be a finite Borel measure in $\R^n$ with
$\vartheta^{k,\star}(x, \mu) < +\infty$ for $\mu$-a.e.~$x \in \R^n$.
Then, $\mu$ is $k$-rectifiable if
\begin{equation}\label{e:AT}
\int_{0}^1 \frac{\big(\beta_{\mu,2}^{(k)} (x,r)\big)^2}{r}\, \d r < 
+\infty \quad\text{for $\mu$-a.e. } x \in \R^n,
\end{equation}
\end{theorem}

The following two remarks are in order.

\begin{remark}\label{r:densita' uno}
In the case $E \subset\R^n$ is a Borel set with $\cH^k(E) < + \infty$,
then $\mu := \cH^k \res E$ has upper-density finite 
$\mu$-almost everywhere. More precisely, 
$\vartheta^{k,\star}(x, \mu) \leq 1$ for $\mu$-a.e.~$x \in 
E$ (see for instance \cite[(2.43)]{AmFuPa}).
\end{remark}

\begin{remark}\label{r:somma}
Let $\lambda \in (0,1)$ be any number. 
For every $\lambda^{q+1} \leq r < \lambda^q$ (with $q \in \N$)
we have that $\beta_{\mu,2}^{(k)}(x,r) \leq C 
\,\beta_{\mu,2}^{(k)}(x, \lambda^{q})$
for some constant $C=C(\lambda,k)$, and hence
\begin{align}\label{e:intbetadiscreto}
\int_{0}^1 \frac{\big(\beta_{\mu,2}^{(k)} (x,r)\big)^2}{r}\, \d r & =
\sum_{q=0}^{\infty} \int_{\lambda^{q+1}}^{\lambda^q}
\frac{\big(\beta_{\mu,2}^{(k)} (x,r)\big)^2}{r}\,\d r
\leq C\,\sum_{q=0}^{\infty}\big(\beta_{\mu,2}^{(k)} (x,\lambda^q)\big)^2.
\end{align}
\end{remark}

We can now prove that $\Gamma(u)$ is $\cH^{n-1}$-rectifiable.

\subsection{Proof of Theorem~\ref{t:rect}}
We are given a solution $u$ to the 
lower dimensional obstacle problem
in $B_1$ and we want to show that $\Gamma(u) \cap B_{R}$
is rectifiable for every $R<1$.
Set $\delta :=  \sfrac{(1-R)}2$ and 
let $\{B_\delta(x_i)\}_{i\in J}$ be a finite covering
of $\Gamma(u) \cap B_R$, with $x_i \in \Gamma(u)$,
and set $L := \max_{i\in J} \Theta_u(x_i,2\delta)$.
Then, it suffices to show that 
$\Gamma(u) \cap B_\delta(x_i)$ is rectifiable.

After a suitable change of variable ($v(x):=u(x_i +\delta \,x)$
-- cf.~Remark~\ref{r:scaling}), we are left to verify
the following statement:
let $v$ be a solution to the lower dimensional obstacle problem
in $B_2$ with $0\in\Gamma(v)$, then 
$\Gamma(v) \cap B_{1}$ is
rectifiable. To this aim,
for every $l \in \N \setminus \{0\}$ we consider
the following sets:
\begin{equation}\label{e:El}
E_l := \Big\{ x \in \Gamma(v) \cap B_{1}\;:\;
\cH^{n-1}\big(\Gamma(v) \cap B_{r}(x)\big) \leq 2\, \omega_{n-1}\,r^{n-1}
\quad\forall \; r\in \Big(0, \frac1l\Big) \Big\}.
\end{equation}
Note that $E_{l} \subseteq E_{l+1}$; and that
Theorem~\ref{t:misura} and Remark~\ref{r:densita' uno} imply 
\begin{equation}\label{e: quasi esaustione o esaurimento?}
\cH^{n-1}\Big(\Gamma(v) \cap B_{1}
\setminus \cup_{l=1}^{\infty} E_l\Big)
= 0. 
\end{equation}
Therefore, it is now enough to show that $E_{l}$ is 
$\cH^{n-1}$-rectifiable for any fixed integer 
$l\in\N$; in this respect we set $\mu_l := \cH^{n-1}\res E_l$.
We fix $\lambda\in(0,\sfrac{1}{18})$ 
and an integer $q_0$ such that $\lambda^{q_0-1}\leq 1$.
By applying Proposition~\ref{p:mean-flatness vs freq}
(with parameter $R=7$) we have that
\begin{align}\label{e:jones numbers4}
\sum_{q=q_0}^{+\infty}\int_{B_{1}}
\beta_{\mu_l}^2\big(y, \lambda^q\big)\,\d \mu_l(y)
&\leq \sum_{q=q_0}^{+\infty}
\frac{C_{\ref{p:mean-flatness vs 
freq}}}{\lambda^{q (n-1)}}
\int_{B_{1}}\int_{B_{\lambda^q}(y)}
\Delta^{18\lambda^{q}}_{\lambda^{q}}(x)
\,\d \mu_l(x)\,\d\mu_l(y)\notag\\
&\leq 
\sum_{q=q_0}^{+\infty}
\frac{C_{\ref{p:mean-flatness vs 
freq}}}{\lambda^{q(n-1)}}
\int_{B_{\sfrac32}}\mu_l(B_{\lambda^q}(x))\,
\Delta^{18\lambda^{q}}_{\lambda^{q}}
(x)\,\d \mu_l(x)\notag
\allowdisplaybreaks\\
& \stackrel{\eqref{e:El}}{\leq} 
2\,\omega_{n-1}\,C_{\ref{p:mean-flatness vs freq}}
\int_{B_{\sfrac32}} \sum_{q=q_0}^{+\infty}
\Delta^{\lambda^{q-1}}_{\lambda^{q}}(x)\,\d 
\mu_l(x)\notag
\allowdisplaybreaks\\
&\leq C \int_{B_{\sfrac32}} I_v(x,1)\d \mu_l(x)<+\infty,
\end{align}
where we used:
Fubini's Theorem
and $1+\lambda^{q_0}\leq \sfrac{3}2$
in the second inequality, 
$18\lambda^{}< 1$ in the third, and $\lambda^{q_0-1}\leq 1$
and $\mu(B_{\sfrac32})<+\infty$ (by Theorem~\ref{t:misura})
in the last line.
The conclusion now follows straightforwardly: indeed, by 
\eqref{e:jones numbers4} we have that
\[
\sum_{q\in\N}\beta_{\mu}^2\big(y, \lambda^{q}\big)<+\infty
\quad\text{for $\mu_l$-a.e.~$y \in B_{1}$.}
\]
In view of \eqref{e:intbetadiscreto}, we can then apply 
Theorem~\ref{t:AT} to conclude that $E_{l}$ is 
$\cH^{n-1}$-rectifiable.
\hfill\qedsymbol
\begin{remark}
The rectifiability of the free boundary can also be deduced by following 
the argument of Naber--Valtorta \cite{NaVa1,NaVa2}, along the proof of the 
covering argument and the discrete Reifenberg Theorem:
we refer to \cite{NaVa1,NaVa2}
for more details.
\end{remark}

%
%
\section{Classification of blow-ups $\cH^{n-1}$-a.e.}\label{s:blow-up}
In this section we give the proof of the last main
result of the paper, namely Theorem~\ref{t:frequency}.
We recall the rescalings for the blow-up procedure:
\begin{equation}\label{e:rescaling}
u_{x_0,r} (y) := \frac{r^{\frac{n+a}{2}}\,u(ry+x_0)}{H^{\sfrac12}(x_0,r)}
\quad \forall \; r < 1 - |x_0|, \;\forall\; y \in B_1.
\end{equation}
In view of Remark~\ref{r:freq modificata}, the functions $\bar u_{x_0,r}$
and $u_{x_0,r}$ have limits which differ only by a multiplicative
constant. Therefore, Theorem~\ref{t:frequency}
is proven once we show the same conclusions for the
new rescalings \eqref{e:rescaling}.

\subsection{Stratification of the free boundary}
We start off with the first part of 
Theorem~\ref{t:frequency} regarding the
estimate on the dimension of the set of points with frequency
$\lambda \in \{2m, 2m-1+s, 2m+2s\}_{m\in\N\setminus\{0\}}$.
We use a stratification argument for the nodal set 
$\mathcal{N}(u)$ of a solution $u$ to the lower dimensional
obstacle problem \eqref{e:ob-pb local}.
This argument goes back to the work of Almgren \cite[{\S} 2.26]{Alm00};
here for convenience we follow \cite{FMS-15}.

\medskip

We start recalling the definition of nodal points:
\[
\mathcal{N}(u) := \Big\{
(x',0)\in B_R'\,:\,u(x',0) = |\nabla_\tau u(x',0)|
= \lim_{t\downarrow 0^+}t^a\de_{n+1}u(x',t) = 0
\Big\}.
\]
Next we specify the main ingredients of 
\cite[{\S} 3.1]{FMS-15} for the thin obstacle problem:
\begin{itemize}
\item[(a)] the upper semi-continuous function $f:B_1' \to \R$ given by
\[
f(x) :=
\begin{cases}
I_u(x,0^+) & \text{if } \;x\in\mathcal{N}(u),\\
0 & \text{if } \;x\not\in\mathcal{N}(u).
\end{cases}
\]

\item[(b)] the compact class of conical functions $\mathcal{G}(x) \subset
L^\infty(\R^{n})$, for every $x\in B_1'$, defined by
\[
\mathcal{G}(x):=
\begin{cases}
\big\{I_w(\cdot, 0^+)\;:\; w \in \text{BU}(x)\big\} & \text{if } 
\;x\in\mathcal{N}(u),\\
\{0\} & \text{if } \;x\not\in\mathcal{N}(u),
\end{cases}
\]
recall that $\text{BU}(x)$
denotes the set of all blow-ups of $u$ at $x$.
\end{itemize}
We need to verify that $\cG(x)$ is a class of
\textit{compact conical} functions
according to \cite[Definition 3.3]{FMS-15}
(the arguments are analogous to those in 
\cite[{\S} 5.2]{FMS-15}, we repeat them for readers'
convenience).
\begin{itemize}
\item[(1)] An upper semi-continuous function $g:\R^n\to\R$ is said 
to be \textit{conical} if $g(z) = g(0)$ implies that
\[
g(z+\lambda \,x) = g(z +x)\quad\forall\;x\in\R^n,\quad\forall\;\lambda>0.
\]
Then, both the zero function and the frequency of homogeneous solutions
$w$ are conical by Lemma~\ref{l:spine}.

\item[(2)] A class $\cG$ of conical functions is called \textit{compact}
if for every sequence $(g_j)_{j\in\N}\subset \cG$ there exist a 
subsequence $(g_{j_i})_{i\in\N}\subset(g_j)_{j\in\N}$ and $g\in \cG$ 
such that
\begin{equation}\label{e:classe compatta}
\limsup_{i\to+\infty} g_{j_i}(y_i) \leq g(y)
\quad\quad \forall\;y\in\R^n,\;\forall\;
(y_i)_{i\in\N}\subset\R^n\;\text{ with }\; y_i\to y.
\end{equation}
According to item (b), if $(g_j)_{j\in\N}\subset\cG(x)$ we may assume 
without loss of generality $g_j$ not identically $0$ for $j$ big.
Then, $g_j = I_{w_j}(\cdot, 0^+)$ and by Lemma~\ref{l:lim uniforme}
and Corollary~\ref{c:compactness} there exists a subsequence $w_{j_i}$
converging to a homogeneous solution $w$ (recall that 
$I_{w_j}(1) = I_u(x,0^+)$ and $H_{w_j}(1) = 1$).
By a diagonal argument we have that $w\in \text{BU}(x)$,
and \eqref{e:classe compatta} follows from
\begin{align*}
\limsup_{i\to+\infty} I_{w_{j_i}}(y_i, 0^+) \leq \inf_{s>0}
\limsup_{i\to+\infty} I_{w_{j_i}}(y_i, s) = \inf_{s>0} I_{w}(y, s) = 
I_{w}(y,0^+).
\end{align*}
\end{itemize}

We discuss next the structural hypotheses  
\cite[(i) - (ii) {\S} 3.1]{FMS-15}:
\begin{itemize}
\item[(i)] $g(0) = f(x)$ for all $g\in\cG(x)$, because
$I_w(0^+) = I_u(x,0^+)$ for every blowup $w\in \text{BU}(x)$;

\item[(ii)] for all $r_j\downarrow 0$ there exists a subsequence
$(r_{j_i})_{i\in\N} \subset (r_{j})_{j\in\N}$
and $w\in \cG(x)$
such that $u_{x,r_{j_i}} \to w$; hence,
for every $y\in\R^n$ and for every sequence $y_i\to y$,
we have
\begin{align*}
\limsup_{i\to+\infty} I_u(x+r_{j_i}y_j, 0^+) & \leq
\inf_{s>0}
\limsup_{i\to+\infty} I_u(x+r_{j_i}y_j, r_{j_i}\,s)\\
&= \inf_{s>0}
\limsup_{i\to+\infty} I_{u_{x,r_{j_i}}}(y_j, s)
= \inf_{s>0} I_{w}(y, s) 
= I_w(y, 0^+).
\end{align*}
\end{itemize}

\medskip

We are then in the position to apply \cite[Theorem 3.4]{FMS-15}
and conclude that the points whose blow-ups have
spines with dimension not exceeding $l\in \{0,\ldots, n\}$
constitute a set of Hausdorff dimension at most $l$.

\begin{theorem}\label{t:straFICA}
Let $u$ be a solution of the thin obstacle problem \eqref{e:ob-pb local} in 
$B_R$.
For $l\in \{0,\ldots, n\}$, set $\Sigma_l(u) := \{x\in \mathcal{N}(u):\, 
\dim S(w) \leq l,\;\forall\;w\in\textup{BU}(x)\}$.
Then, $\Sigma_0(u)$ is at most countable and 
$\dim_{\cH}\Sigma_l(u)\leq l$
for every $l\in\{1,\ldots, n\}$.
\end{theorem}

\medskip

The first assertion of Theorem~\ref{t:frequency}
is now a direct consequence.

\begin{proof}[Proof of Theorem~\ref{t:frequency}: part I]
We first show that $\dim S(w)\leq n-1$
for every $w\in \textup{BU}(x)$ with $x\in \mathcal{N}(u)$.
To this aim, we observe that by the definition of nodal set
we have that $0\in \mathcal{N}(w)$ 
for every $w\in \textup{BU}(x)$ with $x\in \mathcal{N}(u)$. 
On the other hand, using the notation in Theorem~\ref{t:straFICA}, 
$\Sigma_n(u)\setminus \Sigma_{n-1}(u)=\emptyset$
as noticed in Section~\ref{s:classification}. Indeed, the
only non-trivial homogeneous solutions with $n$-dimensional spine
are the functions $w_c:=c\,|x_{n+1}|^{2s}$ with $c<0$,
and by direct computation $\mathcal{N}(w_c) = \emptyset$.

Therefore, for every $x\in \Gamma(u) \setminus\Sigma_{n-2}(u)$
there exists at least a blowup $w \in \textup{BU}(x)$
with an $(n-1)$-dimensional spine $S(w)$, {\ie}~with $w\in \cH^\top$.
Thus, by the classification of all homogeneous solutions with maximal 
spine
in Lemma~\ref{l:classification}, 
the limiting frequency at any point $x\in \Gamma(u) \setminus\Sigma_{n-2}(u)$
satisfies
\[
I_u(x,0^+) = I_w(0,0^+) \in  \{2m, 2m-1+s, 
2m+2s\}_{m\in\N\setminus\{0\}}.
\]
Taking into consideration that $\dim_{\cH}\Sigma_{n-2}(u)\leq n-2$
by Theorem~\ref{t:straFICA}, we conclude the proof.
\end{proof}

\subsection{Uniqueness of blow-ups with frequency $2m$ and $2m-1+s$}
For the second part of Theorem~\ref{t:frequency}
we need an extension of the classification result in 
Lemma~\ref{l:classification}
to the $\lambda$-homogeneous (even symmetric 
with respect to $x_{n+1}$) solutions of
\begin{equation}\label{e:PDEz-0}
 \begin{cases}
\div(|x_{n+1}|^a\nabla u)=0 & B_1\setminus\Lambda(u)\cr
 u=0 & \Lambda(u),
\end{cases}
\end{equation}
with $\lambda\in\{2m,2m-1+s\}_{m\in\N\setminus\{0\}}$ and 
$\{x \cdot e=x_{n+1}=0\} \subseteq \Lambda(u)$
for some unit vector $e\in\R^n\times\{0\}$.
The main differences with Lemma~\ref{l:classification}
are that  neither the unilateral obstacle 
condition nor any invariance assumption of the solutions
({\ie}~the assignment of  the spine) are imposed in this
framework. In the ensuing statement we keep the 
notation introduced in Lemma~\ref{l:classification}.

\begin{proposition}\label{p:classiFICAzione 2d improved-0}
Let $u:\R^{n+1}\to\R$ be a non-trivial
$\lambda$-homogeneous weak solution
of \eqref{e:PDEz-0}, even w.r.to 
$x_{n+1}$, such that $\lambda\in\{2m,2m-1+s\}_{m\in\N\setminus\{0\}}$
and 
$\{x \cdot e=x_{n+1}=0\} \subseteq\Lambda(u)$
for some unit vector $e\in\R^n\times\{0\}$.
Then, there exists $c>0$ such that
$u(x) =c\,h_{2m}(x\cdot e, x_{n+1})$ or
$u(x) =c\,h_{2m-1+s}(x\cdot e, x_{n+1})$ or
$u(x) =c\,h_{2m-1+s}(-x\cdot e, x_{n+1})$.
\end{proposition}

The proof is postponed to Proposition~\ref{p:classiFICAzione 2d improved}
in the appendix. Given it for granted, we proceed with the conclusion
of the proof of Theorem~\ref{t:frequency}.

\begin{proof}[Proof of Theorem~\ref{t:frequency}: part II]
By Theorem~\ref{t:rect} there exist at most
countably many $C^1$-regular submanifolds $\{M_i\}_{i\in\N}$ such
that $\cH^{n-1}(\Gamma(u) \setminus \cup_{i\in \N} M_i) = 0$.
We consider the sets $\Gamma_i(u):= \Gamma(u)\cap M_i$
and
\[
\Gamma_i'(u) := \left\{ x \in \Gamma_i(u) \,:\;
\lim_{r\downarrow 0^+}\frac{\cH^{n-1}(B_r(x_0)\cap 
\Gamma_i(u))}{\omega_{n-1}\,r^{n-1}} = 1
\right\}.
\]
Note that $\cH^{n-1}(\Gamma_i(u) \setminus \Gamma_i'(u)) = 0$ for every
$i\in\N$ by Besicovitch's differentiation theorem 
(cp.~\cite[Theorem~2.22]{AmFuPa}).
We show that for every $i\in\N$ and for every $x_0 \in \Gamma_i'(u)$
the conclusion of Theorem~\ref{t:frequency} holds, namely
if $I(x_0, 0^+) = \lambda \in \{2m, 2m-1+s\}_{m\in \N\setminus 
\{0\}}$,
then there exists a unit vector $e_{x_0}$ with 
$e_{x_0}\perp \text{Tan}_{x_0} M_i$ at $x_0$ such that
\begin{equation}\label{e:riformulazione}
u_{x_0,r} \to h_{\lambda}(x\cdot e_{x_0},x_{n+1})
\quad\text{in $H^1_{\loc}(\R^{n+1}, |x_{n+1}|^a\cL^{n+1})$
as $r \downarrow 0$},
\end{equation}
where $h_\lambda$ are the functions in Lemma~\ref{l:classification}, 
and
$\text{Tan}_{x_0} M_i$ is the linear tangent space to $M_i$ at $x_0$: {\ie}
\[
v\in \text{Tan}_{x_0} M_i \quad \Longleftrightarrow \quad
\exists\; (x_l)_{l\in \N} \subset M_i \; \text{ such that }
\; \lim_{l\to +\infty} \frac{x_l-x_0}{|x_l - x_0|} = v. 
\]
To this aim we consider the compact sets
\[
K_r := 
\Big\{ y \in \overline{B}_1 \,:\,x_0 + r\,y \in \Gamma_i(u) \Big\}
\quad\text{for}\quad
r \leq r_0:=\frac{1}{2}\,(1-|x_0|).
\]
By Blaschke compactness theorem (cp.~\cite[Theorem~6.1]{AmFuPa})
the sequence of sets
$\{K_r\}_{r\in (0,r_0]}$ is pre-compact in the Hausdorff distance 
on $\overline{B}_1$: namely,
given any sequence $(0,r_0]\ni r_i \to 0$, there exists
a subsequence $(r_{i_k})_{k \in \N}$
and a compact set $K_0 \subseteq \overline{B}_1$
such that we have
$\lim_{k}\dist_{\cH}(K_{r_{i_k}}, K_0) = 0$, {\ie}
\begin{itemize}
\item[(a)] any point $x \in K_0$ is an accumulation point for a 
sequence 
$(x_k)_{k\in \N}$ with $x_k \in K_{r_{i_{k}}}$;

\item[(b)] if $x_k \in K_{r_{i_k}}$, then any 
accumulation point of $(x_k)_{k\in \N}$
belongs to $K_0$.
\end{itemize}
We proceed now with the proof of \eqref{e:riformulazione}
in three steps.

\medskip

\noindent{\bf 1.}
Let $r_j \downarrow 0$ be such that
$\dist_{\cH}(K_{r_{j}}, K_0) \to 0$ for some compact set $K_0$.
Then
\[
\text{Tan}_{x_0} M_i\cap \overline{B}_1 \subseteq K_0.
\]
Assuming this is not the case, there exists an open ball
$B_\rho(y_0) \subset B_1 \subset \R^n$ with $y_0 \in \text{Tan}_{x_0} M_i$
such that $\big(\Gamma_i(u) - x_0\big)/r_j \cap B_\rho(y_0) = 
\emptyset$.
In particular, for sufficiently large $j$ we have that
\begin{align*}
\cH^{n-1}&\big(\sfrac{\big(\Gamma_i(u)-x_0\big)}{r_{j}}
\cap \overline{B}_1 \big)
=\cH^{n-1}\big(
\sfrac{\big(\Gamma_i(u) - x_0\big)}{r_{j}} \cap 
\overline{B}_1\setminus B_\rho(y_0) \big)\\ 
&\leq \cH^{n-1}\big(\sfrac{\big(M_i - x_0\big)}{r_{j}} \cap 
\overline{B}_1\setminus B_\rho(y_0) \big)
\leq \big(1+o(1)\big)\,\omega_{n-1}
\,\big(1- \rho^{n-1}\big)\quad\text{as $j\to+\infty$},
\end{align*}
since $\cH^{n-1}\big(\sfrac{\big(M_i - x_0\big)}{r_{j}}\cap A\big) \to 
\cH^{n-1}\big(\text{Tan}_{x_0} M_i \cap A\big)$
for every open set $A \subset \R^n$.
This yields
\[
\limsup_{j\to+\infty} \frac{\cH^{n-1}\big(
\Gamma_i(u) \cap B_{r_j}(x_0)\big)}{\omega_{n-1} 
r_j^{n-1}} \leq 1 - \rho^{n-1},
\]
against the assumption $x_0\in \Gamma'_i(u)$.

\medskip

\noindent{\bf 2.} In particular, it follows that,
if $u_{x_0}$ is any blow-up limit of $u$ at $x_0 \in \Gamma'_i(u)$, then
\begin{equation}\label{e:contenimento}
\text{Tan}_{x_0} M_i
\subseteq \big\{ u_{x_0} = 0\big\}.
\end{equation}
Indeed, set $Y_r:=\big\{u_{x_0,r_{}} = 0 \big\} \cap \overline{B}_1$
and note that $K_{r_{}}\subset Y_r$.
If $Y_0 \subset \overline{B}_1$ is any Hausdorff limit
of a sequence $Y_{r_j}$, then 
$Y_0 \subset \big\{u_{x_0} = 0 \big\} \cap \overline{B}_1$, because
$u_{x_0,r_{j}}(z_j) \to u_{x_0}(z_0)$ for every $z_j\in Y_{r_j}$ with 
$z_j\to z_0$ (thanks to
the uniform convergence of $u_{x_0,r_{i_k}}$).
In particular, being $K_{r_j}\subset Y_{r_j}$, the conclusion follows
from step 1 and the homogeneity of $u_{x_0}$.

\medskip

\noindent{\bf 3.}
We now conclude the proof of \eqref{e:riformulazione}.
Assume without loss of
generality that $\text{Tan}_{x_0} M_i = \{x_{n} = x_{n+1} = 0\}$.
By Proposition~\ref{p:blow-up} we have that
$\textup{BU}(x_0)\subseteq\cH_{\lambda}$
with $\lambda=I_u(x_0,0^+)$, and
we distinguish two possibilities
(recall also that the blow-ups are renormalized
so to have $H_{u_{x_0}}(1)=1$):
\begin{itemize}
\item[(1)] $I_u(x_0,0^+) = 2m$. By
Proposition~\ref{p:classiFICAzione 2d improved-0} 
the blow-up $u_{x_0}$ needs to be
$h_{2m}(x_{n}, x_{n+1})$, because this 
function is the only blow-up with frequency 
$2m$ and contact set containing 
$\text{Tan}_{x_0} M_i = \{x_{n} = x_{n+1} = 0\}$
by \eqref{e:contenimento};

\item[(2)] $I_u(x_0,0^+) = 2m -1+2s$.
By
Proposition~\ref{p:classiFICAzione 2d improved-0} 
every blow-up $u_{x_0}$ is given by either
$h^+ = h_{2m-1+s} (x_n,x_{n+1})$ or $h^-=h_{2m-1+s} (-x_n,x_{n+1})$.
\end{itemize}
In order to infer the uniqueness of the blowup in this last case, we exploit
the connectedness of the set of blow-up limits.
Namely, assume that there exist $r_i\downarrow 0$
and $\rho_i\downarrow 0$ such 
that $u_{x_0,r_i} \to h^+$ and $u_{x_0,\rho_i} \to h^-$; up
to passing to subsequences, we may take $r_i < \rho_i < r_{i+1}$. 
Then, by continuity there exists $t_i \in (r_i , \rho_i)$ such that 
\[
\|u_{x_0,t_i} - h^+\|_{L^2(B_1)} = \|u_{x_0,t_i} - h^-\|_{L^2(B_1)}.
\]
Since the sequence $(u_{x_0, t_i})_{i\in\N}$ has no subsequence 
converging either to $h^+$ or to $h^-$, this gives a contradiction 
and 
concludes the proof of Theorem~\ref{t:frequency}.
\end{proof}

\subsection{Concerning the optimality of Theorem~\ref{t:frequency}}
\label{ss:optimality}
For every $e\in \R^{n+1}$ with $|e|=1$ and $e\cdot e_{n+1} =0$,
the functions 
$u(x)= h_{2m}(x\cdot e, x_{n+1})$ and $u=h_{2m-1+s}(x\cdot e, 
x_{n+1})$ are examples of solutions
to the lower dimensional problem \eqref{e:ob-pb local} in any ball $B_R$
whose free boundary $\Gamma(u)$ is $(n-1)$-dimensional and is
made of points with frequency $2m$ and $2m-1+s$, respectively.
Note that the latters are explicit cases in which $\Gamma(u) = \other(u)$.

\medskip

On the other hand, as pointed out
in the introduction, at the best of our knowledge there are no 
explicit examples of solutions to the lower dimensional obstacle 
problem \eqref{e:ob-pb local} with free boundary points with 
frequency $2m+2s$ with $m\in\N\setminus\{0\}$
(note that, although $h_{2m+2s}(x\cdot e, x_{n+1})$
are solutions, $\Gamma(h_{2m+2s}) = \emptyset$).

Such points do not occur in the one dimensional case $n=1$.
Following the argument of \cite[Remark~1.2.8]{GaPe09} for $s=\sfrac12$, assume 
that $0\in \Gamma(u)$ is a point with frequency $2m+2s$. Then, 
one can find a sequence $(t_k, 0)\in\R^2$ with $\lim_k t_k =0$
such that $u(t_k,0)> 0$ and, therefore, from \eqref{e:bd condition1},
\begin{equation}\label{e:prima condizione}
\lim_{x_2\downarrow 0} x_2^a\de_{x_2}u(t_k,x_2) = 0.
\end{equation}
Taking the rescalings $u_{0,\sfrac{t_k}{2}}$,
up to passing to a subsequence (not relabeled) there exists
a blowup $w\in \textup{BU}(0)$ such that
(cp.~\eqref{e:cpt3}):
\[
\sign(x_2)|x_2|^a\de_{x_2}u_{0,\sfrac{t_k}{2}} \to
\sign(x_2)|x_2|^a\de_{x_2}w 
\quad\text{in }\; C^0_{\loc}(B_1).
\]
Note that necessarily $w=h_{2m+2s}$, because 
there exists a unique blowup with frequency $2m+2s$.
Moreover, from \eqref{e:prima condizione} we have that
$\lim_{x_2\downarrow 0} x_2^a\de_{x_2}w(\sfrac12,x_2) = 0$.
On the contrary a direct computation shows that 
$\lim_{x_2\downarrow 0} x_2^a\de_{x_2}h_{2m+2s}(\sfrac12,x_2) < 0$,
thus leading to a contradiction and implying that there cannot 
exist free boundary points with frequency $2m+2s$ for $n=1$.

Potential points with frequency $2m+2s$ are sometimes referred to in 
the literature as \textit{degenerate points} (see the final section 
of \cite{GaPe09}).
It is a tempting conjecture to claim that 
there are actually none.
If this were the case, Theorem~\ref{t:frequency} would
then be optimal, both concerning the uniqueness of blow-ups
at $\cH^{n-1}$-almost all points of the free boundary,
and the classification of the frequency at $\cH^{n-2}$-almost all 
points of the free boundary.

%
%
\appendix

\section{Homogeneous solutions}\label{a:appendix}
In this appendix we collect some results concerning
homogeneous solutions to the thin obstacle problem
and more generally
to the corresponding system of Euler--Lagrange
equations. Therefore, we consider functions
$u:\R^{n+1} \to \R$ such that
\[
u(x) = |x|^\lambda u \,\big(\sfrac{x}{|x|}\big) 
\quad \forall\; x\neq 0
\]
for some $\lambda \geq 1+s$ (this restriction 
being in accordance with the homogeneity
of all possible blow-ups).

\subsection{Two-dimensional homogeneous solutions}
Here we provide a classification of the homogeneous
solutions to the equation
\begin{equation}\label{e:sistema 2d}
\begin{cases}
\div(|x_2|^a\nabla u)=0 & B_1\setminus \Lambda(u),\\
u=0 & \Lambda(u),
\end{cases}
\end{equation}
in the two dimensional case, \ie~for $n=1$.
Thus, necessarily, the contact set
is a cone, and we have:
\begin{itemize}
\item[(i)] $\Lambda(u)=\{x_1=x_{2}=0\}$,
\item[(ii)] $\Lambda(u)=\{x_1\leq 0,\, x_{2}=0\}$
or $\Lambda(u)=\{x_1\geq 0,\, x_{2}=0\}$,
\item[(iii)] $\Lambda(u)=\{x_{2}=0\}$.
\end{itemize}
Correspondingly, we introduce three classes of functions $\Phi_m$,
$\Psi_m$ and $\Pi_m$ for $m\in \N$, that are explicitly defined as follows:
\begin{gather}
\Phi_m(x_1,x_2):=\sum_{k=0}^{\lfloor \sfrac m2\rfloor}
\alpha_k\,x_1^{m-2k}x_2^{2k},\label{e:polyn b}\\
\Psi_m(x_1,x_2):=
\Big(\sqrt{x_1^2+ x_2^2}+x_1\Big)^s\sum_{k=0}^{m}
\beta_k\Big(\sqrt{x_1^2+x_2^2}-x_1\Big)^{k}
\Big(\sqrt{x_1^2+x_2^2}\Big)^{m-k},\label{e:hy-geo1b}\\
\Pi_m(x_1,x_2):=|x_2|^{2s}\sum_{k=0}^{\lfloor \sfrac{m}2\rfloor}
\gamma_k x_2^{2k} \frac{\d^{2k}}{\d x^{2k}} p(x_1),\label{e:PI}
\end{gather}
where
$\lfloor \cdot\rfloor$ denotes
the integer part of a real number, and
\[
\alpha_0:=1,\quad
\alpha_{k+1}:=-\frac{(m-2k)(m-2k-1)}{4(k+1)(k+1-s)}\,\alpha_k
\qquad k\in\{0,\ldots,\lfloor \sfrac m2\rfloor-1\}, 
\]
\[
\beta_k:=\frac{(m+1)_k(-m)_k}{2^kk!\,(1-s)_k}
\qquad k\in\{0,\ldots,m\},
\]
\[
\gamma_k:=\frac{(-1)^k}{4^k\,k!\,(1+s)_k}\qquad
\text{$k\in\{0,\ldots, \lfloor \sfrac{m}{2}\rfloor\}$,\quad
$p$ is a $m$-homogenous polynomial,}
\]
and the (increasing) Pochhammer symbol is defined by
\[
(q)_l :=
\begin{cases}
1 & \text{if } l=0,\\
q\,(q+1)\cdots(q+l-1) & \text{if } l\in\N\setminus\{0\}.
\end{cases}
\]
We establish the ensuing classification result 
(for related issues see \cite{CaSaSi08,Ro-Se}).

\begin{proposition}\label{p:classiFICAzione 2d}
Let $u:\R^2\to\R$ be $\lambda$-homogeneous, even symmetric 
w.r.to $x_2$, and assume 
that $u$ is a weak solution of \eqref{e:sistema 2d}.
Then, one of the following occurs:
\begin{itemize}
\item[(i)] $\Lambda(u)=\{x_1=x_2=0\}$, $\lambda=m\in\N\setminus\{0\}$ 
and $u$ is a multiple of $ \Phi_m$;

\item[(ii)] $\Lambda(u)=\{x_1\leq 0,\,x_2=0\}$
(resp.~$\Lambda(u)=\{x_1\geq 0,\,x_2=0\}$), 
$\lambda=m+s$ for some $m\in\N$ and  $u$ is a multiple of $\Psi_m$
(resp.~of $\Psi_m(-x_1,x_2)$);

\item[(iii)] $\Lambda(u)=\{x_2=0\}$, $\lambda=m+2s$ for some
$m\in\N$ and  $u$ is a multiple of $\Pi_m$.
\end{itemize}
Moreover, if 
$u$ is a solution
to the lower dimensional obstacle problem, then
$m$ is even in (i) and (iii), and $m$ is odd in (ii).
\end{proposition}

For the proof we
need to introduce the hypergeometric function 
${}_2F_1(\alpha,\beta;\gamma;\cdot):\C\to\C$ defined by
\begin{align}\label{e:hy-geo}
{}_2F_1 (\alpha,\beta;\gamma;z) := \sum_{k=0}^{\infty}
\frac{(\alpha)_k(\beta)_k}{(\gamma)_k}\,\frac{z^k}{k!} ,
\end{align}
where $\alpha,\beta,\gamma\in \C$, $\gamma$ not a negative integer.
The power series defining ${}_2F_1$ is converging for $|z|<1$,
and it can be analytically continued elsewhere. 
In what follows we shall use several properties of ${}_2F_1$ for which we refer 
to the 
Digital Library of Mathematical Functions,
always quoting the precise formulas  employed in the derivation and referring
to their enumeration in \cite{DLMF}.

We warn the reader that, with a slight abuse of notation, in this section 
$\Gamma$ shall denote both the 
free boundary of a solution and the Euler's Gamma-function on the complex 
plane, 
extended to $Re(z)\leq 0$ 
by analytic continuation using the identity $\Gamma(1+z)=z\Gamma(z)$. In 
particular, $\Gamma$ turns 
out to be a meromorphic function with no zeros and simple poles at $z=-m$, 
$m\in\N$. Thus, 
we adopt the convention that $\Gamma^{-1}(-m)=0$ for all $m\in\N$. 

\begin{proof}[Proof of Proposition~\ref{p:classiFICAzione 2d}]
Using polar coordinates $x_1=r\cos\theta$
and $x_2=r\sin\theta$ with 
$r>0$ and
$\theta \in [0,\pi]$, let  $v(r,\theta):=u\big(r\cos\theta, r\sin\theta\big) 
= r^{\lambda}\,y(\theta)$.
The Euler-Lagrange equation \eqref{e:sistema 2d} then reads as
\begin{equation}\label{e:EL polari}
\cL_{a,\lambda} \big[y (\theta)\big]:= y''(\theta) + a\, \ctg \theta 
\, y'(\theta) + 
\lambda\,(\lambda + a)\,y(\theta) 
= 0
\qquad\theta\in(0,\pi),
\end{equation}
with boundary conditions:
\begin{itemize}
\item case (i)
\begin{gather*}
\lim_{\theta\downarrow0^+} (\sin\theta)^{1-2s}y'(\theta) = 0
\quad\text{and}\quad
\lim_{\theta\uparrow\pi^-} (\sin\theta)^{1-2s}y'(\theta) = 0,
\end{gather*}

\item case (ii) (by symmetry we assume
$\Lambda(u) = \{x_1\leq 0, x_2=0\}$)
\begin{gather*}
\lim_{\theta\downarrow0^+} (\sin\theta)^{1-2s}y'(\theta) = 0
\quad\text{and}\quad y(\pi) = 0,
\end{gather*}

\item case (iii)
\begin{gather*}
y(0) = 0 \quad\text{and}\quad y(\pi) = 0.
\end{gather*}
\end{itemize}

The change of variable
$y(\theta) = (\sin\theta)^s\,h(\cos \theta)$ transforms
the ODE for $y$ in \eqref{e:EL polari} into an associated Legendre differential 
equation 
for $h$. More precisely, we get for $\nu = \lambda-s$ and $\nu\geq 1$
\begin{equation}
(1-x^2)h''(x) -2x  h'(x) + \left(\nu^2+\nu-\frac{s^2}{1-x^2}\right)h(x) = 0
\qquad x\in(-1,1),\label{e:associated ODE}
\end{equation}
with the following boundary conditions:
\begin{itemize}
\item case (i)
\begin{gather}
\lim_{x\uparrow1} (1-x^2)^{-\sfrac{s}{2}}
\left(sxh(x)-(1-x^2)h'(x)\right) = 0,
\label{e:bd i-1}\\
\lim_{x\downarrow-1}(1-x^2)^{-\sfrac{s}{2}}\left(sxh(x)-(1-x^2)h'(x)\right) 
= 0,\label{e:bd i-2}
\end{gather}

\item case (ii)
\begin{gather}
\lim_{x\uparrow1} (1-x^2)^{-\sfrac{s}{2}}
\left(sxh(x)-(1-x^2)h'(x)\right) = 0,
\label{e:bd ii-1}\\
\lim_{x\downarrow-1} (1-x^2)^{\sfrac s2}h(x) = 0,\label{e:bd ii-2}
\end{gather}

\item case (iii)
\begin{gather}
\lim_{x\uparrow1} (1-x^2)^{\sfrac s2}h(x) = 0,
\label{e:bd iii-1}\\
\lim_{x\downarrow-1} (1-x^2)^{\sfrac s2}h(x)= 0.\label{e:bd iii-2}
\end{gather}
\end{itemize}

The associated Legendre equation can be solved explicitly in terms of the 
hypergeometric function
${}_2F_1$ (cf. \eqref{e:hy-geo}). A generic solution in the interval $(-1,1)$ 
is 
given by
\[
h(x)=A_1\, P_\nu^s(x)+A_2\,P_\nu^{-s}(x)
\]
where $A_1,\,A_2\in\R$ and
\begin{align}
P_\nu^{\pm s}(x)&:=\frac{1}{\Gamma(1\mp 
s)}\left(\frac{1+x}{1-x}\right)^{\pm\sfrac s2}
{}_2F_1\Big(\nu+1,-\nu,1\mp s,\frac{1-x}2\Big),\label{e:Ps}\\
&=\frac{1}{\Gamma(1\mp 
s)}\left(1-x^2\right)^{\mp\sfrac s2}
{}_2F_1\Big(\mp s -\nu,1\mp s +\nu,1\mp s,\frac{1-x}2\Big)
\label{e:Ps2},
\end{align}
(cf. \cite[(14.3.1), (14.3.2), Section 15.1 and (15.8.1)]{DLMF}).

\medskip

{\bf 1. Dirichlet boundary conditions.} We now proceed computing the boundary
conditions in terms of the explicit representations \eqref{e:Ps}
and \eqref{e:Ps2}.
First, note that by continuity of 
${}_2F_1(\alpha,\beta,\gamma,\cdot)$ 
and since ${}_2F_1(\alpha,\beta,\gamma,0)=1$ for all $\alpha$, $\beta$ and 
$\gamma$, we get
\begin{equation*}
\lim_{x\uparrow1} (1-x^2)^{\sfrac s2}P^s_\nu(x)=\frac{2^s}{\Gamma(1-s)}
\quad\text{and}\quad
\lim_{x\uparrow1} (1-x^2)^{\sfrac s2}P^{-s}_\nu(x)=0,
\end{equation*}
from which we get
\begin{equation}\label{e:dir 1}
\lim_{x\uparrow1} (1-x^2)^{\sfrac s2}h(x)
= A_1\,\frac{2^s}{\Gamma(1-s)}.
\end{equation}
For the corresponding limit values as $x\downarrow -1$ we use 
\cite[(15.4.20)]{DLMF} to infer
\begin{equation*}
\lim_{x\downarrow-1} (1-x^2)^{\sfrac s2}P^{-s}_\nu(x)=
\frac{2^s\,\Gamma(s)}{\Gamma(s-\nu)\,\Gamma(1+s+\nu)},
\end{equation*}
and from \eqref{e:Ps2} and 
\cite[(15.4.20)]{DLMF}
\begin{equation*}
\lim_{x\downarrow-1} (1-x^2)^{\sfrac s2}P^{s}_\nu(x)=
\frac{2^s\,\Gamma(s)}{\Gamma(-\nu)\,\Gamma(1+\nu)},
\end{equation*}
from which
\begin{equation}\label{e:dir -1}
\lim_{x\downarrow-1} (1-x^2)^{\sfrac s2}h(x)
= A_1\,\frac{2^s\,\Gamma(s)}{\Gamma(-\nu)\,\Gamma(1+\nu)}
+A_2\,\frac{2^s\,\Gamma(s)}{\Gamma(s-\nu)\,\Gamma(1+s+\nu)}.
\end{equation}

\medskip

{\bf 2. Neumann boundary conditions.}
For what concerns the boundary conditions involving the derivative
of $h$ 
we use \cite[(15.5.1)]{DLMF} to compute
\begin{align*}
\frac d{dx}P^{\pm s}_\nu(x)=\frac{1}{\Gamma(1\mp s)}\Big[&
\mp s\frac{(1\pm x)^{\sfrac s2-1}}{(1\mp x)^{\sfrac 
s2+1}}{}_2F_1\big(\nu+1,-\nu,1\mp s,\frac{1-x}2\big)
\notag
\\&+\frac{\nu(\nu+1)}{2(1\mp s)}\left(\frac{1+x}{1-x}\right)^{\pm \sfrac s2}
{}_2F_1\big(\nu+2,1-\nu,2\mp s,\frac{1-x}2\big)\Big].
\end{align*}
Hence, we get 
\begin{align}\label{e:ddxPsnu}
 &(1-x^2)^{-\sfrac s2}\left(sxP^{\pm s}_\nu(x)-(1-x^2)\frac{d}{dx}P^{\pm 
s}_\nu(x)\right)=
 \mp\frac{(1\mp x)^{1-s}}{\Gamma(1\mp s)}\cdot\notag\\
&\cdot\Big[s\cdot{}_2F_1\big(\nu+1,-\nu,1\mp s,\frac{1-x}2\big)
\pm\frac{\nu(\nu+1)}{2(1+s)}(1\pm x)\cdot{}_2F_1\big(\nu+2,1-\nu,2\mp 
s,\frac{1-x}2\big)\Big].
\end{align}
From the latter formula we immediately conclude that 
\begin{equation*}
\lim_{x\uparrow1}  (1-x^2)^{-\sfrac s2}
\left(sxP^s_\nu(x)-(1-x^2)\frac{d}{dx}P^s_\nu(x)\right)=0,
\end{equation*}
and
\begin{equation*}
\lim_{x\uparrow1}  (1-x^2)^{-\sfrac s2}
\left(sxP^{-s}_\nu(x)-(1-x^2)\frac{d}{dx}P^{-s}_\nu(x)\right)
=\frac{s\,2^{1-s}}{\Gamma(1+s)}.
\end{equation*}
Therefore, we have
\begin{equation}\label{e:neu 1}
\lim_{x\uparrow1} (1-x^2)^{-\sfrac{s}{2}}
\left(sxh(x)-(1-x^2)h'(x)\right) 
= A_2\,\frac{s\,2^{1-s}}{\Gamma(1+s)}.
\end{equation}
In addition, from \eqref{e:ddxPsnu}, from the linear transformation of variable 
rule for 
${}_2F_1$ in \cite[(15.8.4)]{DLMF} and from \cite[Section~15.5]{DLMF}, 
elementary calculations 
lead to
\begin{align*}
 &(1-x^2)^{-\sfrac s2}\left(sxP^s_\nu(x)-(1-x^2)\frac{d}{dx}P^s_\nu(x)\right)=
\frac{\pi(1-x)^{1-s}}{\sin(s\pi)\Gamma(-s-\nu)\,\Gamma(1-s+\nu)\,\Gamma(1+s)}
\cdot\notag\\
&\cdot\Big(s\cdot{}_2F_1\Big(\nu+1,-\nu,1+s,\frac{1+x}2\Big)
-\frac{\nu(\nu+1)}{2}(1+x)\cdot{}_2F_1\Big(\nu+2,1-\nu,1+s,\frac{1+x}
2\Big)\Big)\notag\\
&+\frac{2^{s-1}\pi(s+\nu)(1-s+\nu)}{\sin(s\pi)\Gamma(\nu+1)\,\Gamma(-\nu)\,
\Gamma(2-s)}
(1-x^2)^{1-s}\cdot{}_2F_1\Big(1-s-\nu,2-s+\nu,2-s,\frac{1+x}2\Big).
\end{align*}
In turn, this implies
\begin{equation*}
\lim_{x\downarrow-1}  (1-x^2)^{-\sfrac s2}
\left(sxP^s_\nu(x)-(1-x^2)\frac{d}{dx}P^s_\nu(x)\right)=
\frac{2^{1-s}\pi}{\sin(s\pi)\Gamma(-s-\nu)\,\Gamma(1-s+\nu)\,\Gamma(s)}.
\end{equation*}
Finally, by \cite[(15.8.1)]{DLMF} we rewrite \eqref{e:ddxPsnu} for $P^{-s}_\nu$ 
as
\begin{align*}
(1-x^2)^{-\sfrac 
s2}\Big(sxP^{-s}_\nu(x)-&(1-x^2)\frac{d}{dx}P^{-s}_\nu(x)\Big)\notag\\
&=\frac{1}{\Gamma(1+s)}\cdot\Big[s(1+x)^{1-s}{}_2F_1\big(\nu+1,-\nu,
1+s,\frac{1-x}2\big)
\\&\quad\quad-\frac{\nu(\nu+1)}{2^s(1+s)}(1-x)
\cdot{}_2F_1\big(s-\nu,1+s+\nu,2+s,\frac{1-x}2\big)\Big],
\end{align*}
and infer from \cite[(15.4.20)]{DLMF} 
\begin{equation*}
\lim_{x\downarrow-1}  (1-x^2)^{-\sfrac s2}
\left(sxP^{-s}_\nu(x)-(1-x^2)\frac{d}{dx}P^{-s}_\nu(x)\right)=
-\frac{\nu(\nu+1)}{2^{s-1}}\frac{\Gamma(1+s)\Gamma(1-s)}{
\Gamma(2+\nu)\Gamma(1-\nu)},
\end{equation*}
{\ie}
\begin{align}\label{e:neu -1}
\lim_{x\downarrow-1} (1-x^2)^{-\sfrac{s}{2}}
\left(sxh(x)-(1-x^2)h'(x)\right) 
& = A_1\,
\frac{2^{1-s}\pi}{\sin(s\pi)\Gamma(-s-\nu)
\,\Gamma(1-s+\nu)\,\Gamma(s)}\notag\\
&\quad-A_2\,\frac{\nu(\nu+1)}{2^{s-1}}
\frac{\Gamma(1+s)\Gamma(1-s)}{
\Gamma(2+\nu)\Gamma(1-\nu)}
\end{align}

\medskip

{\bf 3.} By means of \eqref{e:dir 1}, \eqref{e:dir -1},
\eqref{e:neu 1} and \eqref{e:neu -1}
we are able
to complete the classification by 
discussing all the cases (i) -- (iii).
We start off with case (i):
using \eqref{e:neu 1} and \eqref{e:neu -1}
we deduce that $A_2 =0$ and $\nu + s = m \in\N$
(in order to have $\sfrac{1}{\Gamma(-s-\nu)} =0$).
Therefore,
\begin{align*}
h(x)&= A_1 
P_\nu^s(x)
\stackrel{\eqref{e:Ps2}}{=}\frac{2^s \, A_1}{\Gamma(1-s)}(1-x^2)^{-\sfrac 
s2}\,{}_2F_1\Big(1+m -2s,-m,1-s,\frac{1-x}2\Big).
\end{align*}
In particular, $(1-x^2)^{\sfrac s2}h(x)$ is a polynomial of degree $d\leq m$
(or a constant if $m=0$), as $(-m)_k=0$ for every $k\geq m+1$.
The case $m=0$ implies $y$ to be constant and thus $u\equiv 0$,
which is excluded from the condition $\Lambda(u) = \{x_1=x_2=0\}$.
Hence, $m>0$ and $y$ is a polynomial of degree $d$
in $\cos\theta$.
As for every $k\geq 0$
\[
\cL_{\lambda,a}\big[\,(\cos \theta)^k\big]
= \big(\lambda(\lambda+a)-k(k+a)\big)\,(\cos \theta)^k + 
k(k-1)\,(\cos 
\theta)^{k-2},
\]
we infer that $d= \lambda= \nu+s = m$
and that $y$ depends only on powers of $\cos\theta$
with the same parity as $m$:
\[
y(\cos\theta) = a_m\,(\cos\theta)^m +
a_{m-2}\,(\cos\theta)^{m-2} +\cdots +
a_{m-\lfloor \sfrac{m}{2}\rfloor}\,(\cos\theta)^{m-\lfloor \sfrac{m}{2}\rfloor},
\quad a_m\neq 0.
\]
Therefore, $u$ is an $m$-homogeneous polynomial of the form in \eqref{e:polyn b}
and by a direct computation
\begin{align*}
\div &\big(x_2^a\nabla u(x)\big) =\\
&\sum_{k=0}^{m-1} \Big((m-2k)(m-2k-1)\, \alpha_k+ 
2(k+1)(2k+1+a)\,\alpha_{k+1}\Big)\,x_1^{m-2k-2}x_2^{2k+a}
=0
\end{align*}
we conclude the explicit form of the coefficients $\alpha_k$.

\medskip 

Next we discuss case (ii): 
from \eqref{e:neu 1} we get $A_2=0$
and from \eqref{e:dir -1} we get $\nu\in\N$.
Thus ${}_2F_1(\nu+1,-\nu,1-s,\cdot)$ is a polynomial of degree at most 
$\nu=m$ with $m\in\N\setminus\{0\}$.
The corresponding representation formula in \eqref{e:hy-geo1b}
follows at once from
\begin{align*}
u\big(r\cos\theta, r\sin\theta\big) &= 
A_1\, r^{m+s}\,(\sin\theta)^s\,P^s_\nu(\cos\theta)\\&
\stackrel{\eqref{e:Ps}}{=}A_1\,r^{m+s}\,
\sum_{k=0}^{m}\frac{(m+1)_k(-m)_k}{2^k k!\,(1-s)_k}\big(1-\cos\theta\big)^k
\big(1+\cos\theta\big)^s\,.
\end{align*}

\medskip

We discuss case (iii): from \eqref{e:dir 1} and \eqref{e:dir -1}
we get that $A_1=0$ and $\nu-s = m \in\N$
and the representation formula for solutions in \eqref{e:PI} follows
by direct verification (alternatively one can also derive it 
from the explicit formula in terms of the hypergeometric function).

\medskip

{\bf 4.} Finally, we discuss the case of solutions $u$
to the lower dimensional obstacle problem \eqref{e:ob-pb local}. 
In particular, $u$ solves \eqref{e:sistema 2d}, $u\vert_{B_1'}\geq 0$
and the normal weighted derivative satisfies a sign condition.
Thus, the following additional boundary conditions need to be satisfied
by $y$:
\begin{gather*}
\lim_{\theta\downarrow0^+} (\sin\theta)^{1-2s}y'(\theta) \leq 0
\quad\text{and}\quad
\lim_{\theta\uparrow\pi^-} (\sin\theta)^{1-2s}y'(\theta) \geq 0,\\
y(0) \geq 0 \quad\text{and}\quad y(\pi) \geq 0.
\end{gather*}
In turn, these for the function $h$ translate into 
\begin{gather}
\lim_{x\uparrow1} (1-x^2)^{-\sfrac{s}{2}}
\left(sxh(x)-(1-x^2)h'(x)\right) \leq 0,
\label{e:extra bd 1}\\
\lim_{x\downarrow-1}(1-x^2)^{-\sfrac{s}{2}}\left(sxh(x)-(1-x^2)h'(x)\right) 
\geq 0,\label{e:extra bd 2}\\
\lim_{x\uparrow1} (1-x^2)^{\sfrac s2}h(x) \leq 0,
\label{e:extra bd 3}\\
\lim_{x\downarrow-1} (1-x^2)^{\sfrac s2}h(x)\geq 0.\label{e:extra bd 4}
\end{gather}
We can then discuss the
implications of \eqref{e:extra bd 1} -- \eqref{e:extra bd 4}
for the tree cases (i) -- (iii).
In case (i), by \eqref{e:dir 1} and \eqref{e:extra bd 3} we get
$A_1\geq 0$; similarly, by \eqref{e:dir -1} and \eqref{e:extra bd 4} we get
that $\Gamma(-\nu) >0$, {\ie}~$2m-1\leq \nu \leq 2m$
for some $m\in \N\setminus\{0\}$. Since $\nu +s \in\N$, we conclude that
$\nu + s=2m$.

In case (ii), using \eqref{e:dir 1} and \eqref{e:extra bd 3}
we conclude that $A_1\geq 0$; moreover, from
\eqref{e:neu -1} and \eqref{e:extra bd 2} we infer that
$\Gamma(-s-\nu)\geq 0$ and therefore
$2m+1\leq \nu+s\leq 2m+2$ for some $m\in\N$.
In particular, since $\nu\in\N$, we have that $\nu =2m+1$
is odd.

Finally, in case (iii),
using \eqref{e:neu 1} in \eqref{e:extra bd 1}
and \eqref{e:neu -1} in \eqref{e:extra bd 2}
we deduce that $A_2\geq 0$ and $\Gamma(1-\nu) \leq 0$,
from which it follows that $\nu-s =2m$ is even.
\end{proof}

Using Proposition~\ref{p:classiFICAzione 2d} we now
complete the proof of the classification of global solutions
$u\in \cH^{top}$ with $(n-1)$-dimensional spine in
Lemma~\ref{l:classification}.

\begin{proof}[Proof of Lemma~\ref{l:classification}]
For every $u\in \cH^{top}$, we have that $u$ depends on $x_{n+1}$
and only one in-plane variable, {\ie}~$u(x) = v (x\cdot e, x_{n+1})$ 
for some $v:\R^2\to\R$
and for some unit vector $e\in\R^{n+1}$ with $e\cdot e_{n+1} =0$.
In particular, $v$ is a two-dimensional solution to the lower
dimensional obstacle problem in $\R^2$.
Therefore, by Proposition~\ref{p:classiFICAzione 2d}
we know that $\lambda\in\{2m,2m-1+s,2m+2s\}_{m\in\N\setminus\{0\}}$
and $v$ is one of the functions in \eqref{e:polyn b} -- \eqref{e:PI}.
The statements about
$\Gamma(u)$, $\Lambda(u)$, $\mathcal{N}(u)$ and $S(u)$
follow from the explicit formulas therein.
\end{proof}

\subsection{Further classification results}
Here we provide
a proof to Proposition~\ref{p:classiFICAzione 2d improved-0}.
We split the argument in two parts.  
We start off classifying 
in any dimension 
all $\lambda$-homogeneous solutions (even symmetric
with respect to $x_{n+1}$) of 
\begin{equation}\label{e:PDEz}
 \begin{cases}
 -\div(|x_{n+1}|^a\nabla u)=0 & B_1\setminus\Lambda(u)\cr
 u=0 & \Lambda(u)
 \end{cases}
\end{equation}
such that $\lambda\in [1+s,2+s)$
and having as contact set $\Lambda(u)$ one of the following
\begin{itemize}
\item[(i)] $\Lambda(u)=\{x_n=x_{n+1}=0\}$,
\item[(ii)] $\Lambda(u)=\{x_n\leq 0,\, x_{n+1}=0\}$,
\item[(iii)] $\Lambda(u)=\{x_{n+1}=0\}$.
\end{itemize}
We follow the arguments
in \cite[Lemma~5.3]{CaSaSi08} and 
\cite[Lemma~A.3]{GaPePoSm},
in which the case
$\lambda=1+s$ with $\Lambda(u)=\{x_n\leq 0,\, x_{n+1}=0\}$
is addressed. To this aim we introduce the following notation: 
$x=(x'',x_n,x_{n+1})\in \R^{n-1}\times\R\times\R$.
 
\begin{lemma}\label{l:2hom polyn}
Let $u:\R^{n+1}\to\R$ be a $\lambda$-homogeneous
solution of \eqref{e:PDEz}, even symmetric w.r.to $x_{n+1}$,
with $\lambda\in[1+s,2+s)$
and $\Lambda(u)$ one of the sets in (i) -- (iii) above.
Then, the following occurs:
\begin{itemize}
\item in case (i), $\lambda =2$ and there exists
a $1$-homogeneous polynomial 
$q:\R^{n-1}\to\R$ and a constant $c\in\R$ such that 
\begin{equation}\label{e:i 2}
u(x)=q(x'')\,x_n+c\,\Phi_2(x_n,x_{n+1});
\end{equation}

\item in case (ii), $\lambda =1+s$ and there exists a
$1$-homogeneous polynomial 
$q:\R^{n-1}\to\R$ and a constant $c\in\R$ such that 
\begin{equation}\label{e:ii  1+s}
u(x)=q(x'')\,\left(x_n+\sqrt{x_n^2+x_{n+1}^2}\right)^s
+c\,\Psi_1(x_n,x_{n+1});
\end{equation}

\item in case (iii), $\lambda =1+2s$ and there exists
a $1$-homogeneous 
polynomial $q:\R^{n}\to\R$ such that 
\begin{equation}\label{e:iii  1+2s}
u(x)=|x_{n+1}|^{2s}\,q(x').
\end{equation}
\end{itemize}
\end{lemma}

\begin{proof}
{\bf 1.}
In case (i), since $\cH^n(\Lambda(u))=0$, it follows
from \cite[Lemma~5.3]{CaSaSi08} that
$u$ is a polynomial.
Therefore, $\lambda=2$
and by symmetry $u(x)=q(x^\prime,x_n)+\alpha\,x_{n+1}^2$, 
with 
$q:\R^n\to\R$ a 
$2$-homogeneous polynomial and 
 $\alpha\in\R$. Furthermore, by taking into account that 
$\Lambda(u)=\{x_n=x_{n+1}=0\}$ we infer that 
 $q(x^\prime,x_n)=q_1(x^\prime)x_n+\beta\,x_n^2$, with $q_1:\R^{n-1}\to\R$ a 
$1$-homogeneous polynomial and  $\beta\in\R$.
Thus, 
$u(x)=q_1(x^\prime)x_n+\beta\,x_n^2+\alpha\,x_{n+1}^2$, and 
imposing that $u$ solves the 
equation we conclude that
\[
\div(|x_{n+1}|^a\nabla u) = 
\div\big(|x_{n+1}|^a\nabla (\beta\,x_n^2+\alpha\,x_{n+1}^2)\big) = 0.
\]
In particular, from the classification in Proposition~\ref{p:classiFICAzione 
2d},
we must have $\beta\,x_n^2+\alpha\,x_{n+1}^2 = c\,\Phi_2$,
thus implying \eqref{e:i 2}.

\medskip

{\bf 2.} In case (ii), we consider the tangential derivatives
up to the third order $\de_{i}u$, $\de_{ij}u$ and $\de_{ijk}u$
in directions $i,j,k\in\{1,\ldots,n-1\}$.
By the regularity estimate in \cite{FaKeSe82}
(cf.~also \cite[Lemma~A.2]{GaPePoSm}) we deduce that
$\de_{i}u$, $\de_{ij}u$ and $\de_{ijk}u \in 
H^1_\loc(\R^{n+1},|x_{n+1}|^a\,\cL^{n+1}) \cap 
L^\infty_\loc(\R^{n+1})$.
In particular, since $\de_{ijk}u$ is $\lambda -3 <0$
homogeneous, it follows that $\de_{ijk}u\equiv 0$
for all $i,j,k\in\{1,\ldots,n-1\}$.
We then infer that
\[
\de_{ij}u(x'',x_n,x_{n+1}) = \de_{ij}u(0,x_n,x_{n+1}).
\]
Being $\de_{ij}u(0,x_n,x_{n+1})$ solution to 
\eqref{e:PDEz}, the analysis in 
Proposition~\ref{p:classiFICAzione 2d}
implies that its homogeneity $\lambda-2$ is at least $s$, 
a condition excluded by the restriction $\lambda < 2+s$.
We then conclude that $\de_{ij}u\equiv 0$
for all $i,j,\in\{1,\ldots,n-1\}$, thus we get
\[
\de_{i}u(x'',x_n,x_{n+1}) = \de_{i}u(0,x_n,x_{n+1}),
\]
and
\begin{equation}\label{e:taylor}
u(x'',x_n,x_{n+1}) = 
u(0,x_n,x_{n+1}) + \sum_{i=1}^{n-1}\de_iu(0,x_n,x_{n+1})\,x_i.
\end{equation}

In particular, we infer from Proposition~\ref{p:classiFICAzione 2d} 
that the only allowed homogeneity is $\lambda=1+s$,
$u(0,x_n,x_{n+1}) = c_0 \Psi_1(x_n,x_{n+1})$
and $\de_{i}u(x'',x_n,x_{n+1}) = 
c_i\Psi_0(x_n,x_{n+1})$, for some constants $c_i\in \R$
(note that all these functions solve \eqref{e:PDEz}
with contact set $\Lambda=\{x_n\leq0, x_{n+1}=0\}$).
Using the explicit formulas in 
Proposition~\ref{p:classiFICAzione 2d},
we conclude \eqref{e:ii 1+s}.

\medskip

{\bf 3.} For case (iii), we can argue analogously
as above. In particular, from the 
$(\lambda -3)$-homogeneity of
$\de_{ijk}u$ and $\lambda -3 <0$,
it follows that $\de_{ijk}u\equiv 0$
for all $i,j,k\in\{1,\ldots,n\}$. 
Therefore, $\de_{ij}u$ are functions which are
$(\lambda-2)$-homogeneous and depend only on $x_{n+1}$.
By a direct computation we get from \eqref{e:PDEz}
that $\de_{ij} u = c\,x_{n+1}^{2s}$, {\ie}~$\lambda =2+2s$:
since $\lambda<2+s$, we infer that $c=0$ and 
$\de_{ij} u = 0$ for all $i,j=\{1,\ldots,n\}$,
in turn implying
\begin{equation}\label{e:taylor2}
u(x',x_{n+1}) = 
u(0,x_{n+1}) + \sum_{i=1}^{n}\de_i u(0,x_{n+1})\,x_i.
\end{equation}
By taking into account the homogeneity of $u$ 
and $\de_iu$ and \eqref{e:PDEz} (which implies, in particular,
that $u(0,x_{n+1}) = 0$), one then
obtains \eqref{e:iii  1+2s}.
\end{proof}

We are now ready to prove the general case
of Proposition~\ref{p:classiFICAzione 2d improved-0}.
Actually, we show a slightly more general result.

\begin{proposition}\label{p:classiFICAzione 2d improved}
Let $u:\R^{n+1}\to\R$ be a non-trivial $\lambda$-homogeneous
function even w.r.to $x_{n+1}$.
Assume that $u$ is a weak solution of \eqref{e:PDEz}.
\begin{itemize}
\item[(i)] If $\lambda=m\in\N\setminus\{0,1\}$
and $\{x\cdot e=x_{n+1}=0\}\subseteq \Lambda(u)$
for some unit vector $e\in\R^n\times\{0\}$, then
\begin{equation}\label{e:polyn no spine}
u(x)= \sum_{k=0}^{m}p_k(x^{\prime \prime})\,\Phi_{m-k}(x\cdot 
e,x_{n+1}), 
\end{equation}
with $p_k:\R^{n-1}\to\R$ harmonic $k$-homogeneous polynomial.  

\item[(ii)] If $\lambda=m+s$, with $m\in\N\setminus\{0\}$,
and $\{x\cdot e=x_{n+1}=0\}\subseteq \Lambda(u)$
for some unit vector $e\in\R^n\times\{0\}$, then
\begin{equation}\label{e:hy-geo no spine-1}
u(x)=\sum_{k=0}^{m}p^+_k(x^{\prime \prime})\,\Psi_{m-k}(x\cdot 
e,x_{n+1})+
\sum_{k=0}^{m}p^{-}_k(x^{\prime \prime})\,\Psi_{m-k}(-x\cdot 
e,x_{n+1}), 
\end{equation}
with $p_k:\R^{n-1}\to\R$ harmonic $k$-homogeneous 
polynomial. 

\item[(iii)] If $\lambda=m+2s$, with $m\in\N$, and 
$\Lambda(u)=\{x_{n+1}=0\}$,
then
\begin{equation}\label{e:polinomio banana}
u(x)=|x_{n+1}|^{2s}\sum_{k=0}^{\lfloor \sfrac{m}2\rfloor}
\gamma_k\, x_{n+1}^{2k} \Delta^k p(x')
\end{equation}
with $p:\R^{n}\to\R$ any $m$-homogeneous polynomial
and $\gamma_k := \frac{(-1)^k}{4^k k!\,(1+s)_k}$.
\end{itemize}
Moreover, if $u$ is a solution to the thin obstacle problem
\eqref{e:ob-pb local}, then
in case (i), respectively (ii), $u$ turns out to be
a positive multiple of $h_{2m}(x\cdot e, x_{n+1})$, 
respectively $h_{2m-1+s}(\pm x\cdot e, x_{n+1})$.
\end{proposition}

\begin{proof}
Without loss of generality we assume that $e=e_n$.
The proof proceeds by induction 
on $m\in\N$, with starting step provided by Proposition~\ref{l:2hom polyn}.

The cases (i) and (ii) can be treated by the same argument. 
We consider the horizontal partial derivatives 
$\partial_{x_j}u$ for 
$j\in\{1,\ldots,n-1\}$.
By the regularity estimate in \cite{FaKeSe82}
we have that
$\de_{x_j}u\in H^1_\loc(\R^{n+1},|x_{n+1}|^a\,\cL^{n+1}) \cap 
L^\infty_\loc(\R^{n+1})$
are solutions to \eqref{e:PDEz-0}
with $\{x_n=x_{n+1}\} \subseteq \Lambda(\de_i u)$
and homogeneity $\lambda-1= m-1$ or $\lambda=m-1+s$,
according to the two cases.
Using the inductive hypothesis $\de_i u = \sum_{k=0}^{m-1} p_{i,k} 
\Phi_{m-1-k}$
for $\lambda = m$
or $\de_i u = \sum_{k=0}^{m-1} p_{i,k}^+ \Psi_{m-1-k}(\cdot,\cdot)
+\sum_{k=0}^{m-1} p_{i,k}^- \Psi_{m-1-k}(-\cdot,\cdot)$ for $\lambda 
= m+s$,
for some harmonic $k$-homogeneous polynomials $p_{i,k}$ and 
$p_{i,k}^\pm$.
Therefore, we infer that
\begin{align*}
u(x) &= u(0,x_n,x_{n+1}) + \int_0^1 \sum_{i=1}^{n-1} 
\de_{x_i}u(tx'',x_{n},x_{n+1})\, x_i\,\d t\\
&=
u(0,x_n,x_{n+1}) + \int_0^1 \sum_{i=1}^{n-1} \sum_{k=0}^{m-1}
t^k p_{i,k}(x'') \Phi_{m-1-k}(x_{n},x_{n+1})\, x_i\,\d t\\
& = u(0,x_n,x_{n+1}) + \sum_{k=0}^{m-1} p_k(x'') \Phi_{m-1-k}(x_{n},x_{n+1}),
\end{align*}
and similarly
\[
u(x) = u(0,x_n,x_{n+1}) + \sum_{k=0}^{m-1} p_k^{+}(x'') 
\Psi_{m-1-k}(x_{n},x_{n+1})
+ \sum_{k=0}^{m-1} p_k^{-}(x'') \Psi_{m-1-k}(-x_{n},x_{n+1}),
\]
with $k\,p^{(\pm)}_k(x'') :=\sum_{i=1}^{n-1}p^{(\pm)}_{i,k}(x'') 
x_i$.
Using the equation \eqref{e:PDEz} (in particular, recall that
$\Phi_l, \Psi_l$ are solutions of \eqref{e:PDEz}),
we deduce that the polynomials $p_k, p_k^{\pm}$ are harmonic and 
$u(0,x_n,x_{n+1})$ is itself a solution ({\ie}~ $u(0,x_n,x_{n+1})
= c\, \Phi_m(x_n,x_{n+1})$ or 
$u(0,x_n,x_{n+1})
= c_1\, \Psi_m(x_n,x_{n+1}) + c_2\, \Psi_m(-x_n,x_{n+1})$
for some $c,c_1,c_2\in \R$),
thus concluding the proof for the cases (i) and (ii).
 
\medskip 

In case (iii) with $\lambda=m+2s$,
we consider instead all the horizontal
derivatives of $u$ and use the inductive hypothesis
\eqref{e:polinomio banana} in the form
\[
\de_iu (x) = |x_{n+1}|^{2s}\sum_{k=0}^{\lfloor \frac{(m-1)}2\rfloor}
x_{n+1}^{2k} q_{i,m-1-2k}(x') \quad\forall\;i\in\{1, \ldots, n\},
\]
where $q_{m-2k}^i$ are $(m-1-2k)$-homogeneous polynomials.
Therefore, 
\begin{align*}
u(x) &= u(0,x_{n+1}) + |x_{n+1}|^{2s}\int_0^1 
\sum_{i=1}^{n} \sum_{k=0}^{\lfloor \frac{(m-1)}2\rfloor}
x_{n+1}^{2k} q_{i,m-1-2k}(tx') \, x_i\,\d t\\
& = u(0,x_{n+1}) + |x_{n+1}|^{2s} \sum_{k=0}^{\lfloor \frac{(m-1)}2\rfloor}
x_{n+1}^{2k} q_{m-2k}(x),
\end{align*}
with
$(m-2k)q_{m-2k}(x) = \sum_{i=1}^{n} q_{i,m-1-2k}(x') x_i$.
Taking into account the homogeneity of $u$, we infer that
$u(0,x_{n+1}) = c |x_{n+1}|^{m+2s}$ and
the exact form for the polynomials $q_k$ in
\eqref{e:polinomio banana} follows by 
using the equation \eqref{e:PDEz}.

\medskip

Finally, we discuss the case of solution to the obstacle 
problem \eqref{e:ob-pb local}.
In case (i), the unilateral condition $u\geq 0$ on $B_1^\prime$
implies that 
\[
u(x^{\prime \prime},x_n,0)= 
\sum_{k=1}^{m}p_k(x^{\prime 
\prime})\,x_n^{m-k}\,\Phi_{m-k}(1,0)+p_0\,\Phi_m(x_n,0)\geq 0
\quad\forall\; x''\in\R^{n-1}\times\{0\}.
\]
This implies that the polynomials $p_k$ with $k\in\{1,\ldots,m\}$ 
are all zero.
Let, indeed, $j:=\min\{k\in\{1,\ldots,m\}:\,p_k\not\equiv 0\}$
and divide $u$ by $x_n^{j}>0$: by taking the limit 
as $x_n\downarrow 0$ we infer that $p_{j}$ is a constant sign homogeneous 
harmonic polynomial, which holds only if $p_{j}\equiv 0$, thus giving 
a
contradiction. Therefore, we conclude 
$u(x) = p_0\Phi_m(x_{n},x_{n+1})$ with $p_0>0$
for solutions to the obstacle problem.

For the case (ii), by the same argument we deduce that all polynomials
$p_k^{\pm}\equiv 0$ for $k\in\{1,\ldots,m\}$, and therefore
$u(x) = p^+_0\Psi_m(x_{n},x_{n+1}) + p^-_0\Psi_m(-x_{n},x_{n+1})$
with $p_0^\pm\geq 0$.
Since $u$ is a function of two variables, the conclusion 
follows now from Lemma~\ref{l:classification}.
\end{proof}

%
%

\bibliographystyle{plain}

\end{document}